\documentclass[11pt]{amsart}

 \usepackage{mathpple}
\usepackage{amsmath,amsfonts, amscd,amsthm, graphicx}
  \usepackage{amsrefs}
\usepackage[all]{xy}
\usepackage{color}

\usepackage{hyperref}

\numberwithin{equation}{section}

\newcommand{\op}{\operatorname}
\newcommand{\C}{\mathbb{C}}

\newcommand{\Q}{\mathbb{Q}}

 \newcommand{\E}{{\mathcal E}}
\newcommand{\Hzero}{{\mathcal H}_{(0)}}

\providecommand{\abs}[1]{\left\lvert#1\right\rvert}

\newcommand{\abracket}[1]{\left\langle#1\right\rangle}
\newcommand{\bbracket}[1]{\left[#1\right]}
\newcommand{\fbracket}[1]{\left\{#1\right\}}
\newcommand{\bracket}[1]{\left(#1\right)}

\newcommand{\mc}{\mathcal}
\newcommand{\cinfty}{C^{\infty}}
\newcommand{\pa}{\partial}

\renewcommand{\dbar}{\bar\partial}
\newcommand{\OO}{{\mathcal O}}

\newcommand{\into}{\hookrightarrow}

\newcommand{\iso}{\cong}

\newcommand{\A}{\mathcal A}

\renewcommand{\Im}{\op{Im}}

\DeclareMathOperator{\Spec}{Spec}

\DeclareMathOperator{\Tr}{Tr}

\DeclareMathOperator{\PV}{PV}

\DeclareMathOperator{\Jac}{Jac}
\DeclareMathOperator{\Res}{Res}

\theoremstyle{plain}
\newtheorem{thm}{Theorem}[section]

\newtheorem{thm-defn}{Theorem/Definition}[section]
\newtheorem{lem}[thm]{Lemma}
\newtheorem{lem-defn}[thm]{Lemma/Definition}
\newtheorem{prop}[thm]{Proposition}
\newtheorem{cor}[thm]{Corollary}

\newtheorem{prop-defn}[thm]{Proposition-Definition}

\newtheorem{defn}[thm]{Definition}%[section]

\newtheorem{eg}[thm]{Example}

\newtheorem{thm-alg}[thm]{Theorem/Algorithm}
\newtheorem{rmk}[thm]{Remark}

\allowdisplaybreaks[4]  %%%%%%%%%%% choose 1-4, where 4 is the strongest desire to breack

\begin{document}

 \title{Primitive forms via polyvector fields}
  \author{Changzheng Li}
  \address{Kavli Institute for the Physics and Mathematics of the Universe (WPI),
Todai Institutes for Advanced Study, The University of Tokyo,
5-1-5 Kashiwa-no-Ha,Kashiwa City, Chiba 277-8583, Japan}
\email{changzheng.li@ipmu.jp}

\author{Si Li}
  \address{Department of Mathematics and Statistics, Boston University, 111 Cummington Mall, Boston, U.S.A.
  }
 \email{sili@math.bu.edu}

\author{Kyoji Saito}
  \address{Kavli Institute for the Physics and Mathematics of the Universe (WPI),
Todai Institutes for Advanced Study, The University of Tokyo,
5-1-5 Kashiwa-no-Ha,Kashiwa City, Chiba 277-8583, Japan}
\email{kyoji.saito@ipmu.jp}
  \date{}

  % ∑‚√Ê
  \maketitle

%%%%%%%%%%%%%%%%%%%%%%%%%%%%%%
%% «∞—‘≤ø∑÷
%%%%%%%%%%%%%%%%%%%%%%%%%%%%%%

\begin{abstract}
We develop a complex differential geometric approach to the theory of higher residues and primitive forms from the viewpoint of Kodaira-Spencer gauge theory, unifying the semi-infinite period maps for Calabi-Yau models and Landau-Ginzburg models. We give an explicit perturbative construction of primitive forms with respect to opposite filtrations and primitive elements. This leads to a concrete algorithm to compute the Taylor expansions of primitive forms as well as the  description of their moduli space for all weighted homogenous cases. As an example, we present unknown perturbative expressions for the primitive form of $E_{12}$-singularity and illustrate its application to Landau-Ginzburg mirror symmetry with FJRW-theory.
\end{abstract}

  % ƒø¬º
 \tableofcontents
  % ±Ì∏Òƒø¬º
%  \listoftables
  % ≤ÂÕºƒø¬º
%  \listoffigures

\section{Introduction}

\subsection{Motivations} Around early 1980's, the third author introduced \emph{primitive forms}  \cite{Saito-lecutures, Saito-residue, Saito-unfolding,Saito-primitive} in his study of periods   as to generalize the elliptic period integral theory to higher dimensions.  One important  consequence of a primitive form is the \emph{flat structure} associated to the universal unfolding of a singularity. Later Dubrovin \cite{Dubrovin} introduced the notion of \emph{Frobenius manifold structure} to axiomatize two-dimensional topological field theories  on the sphere. It was soon realized that the aforementioned flat structure describes precisely the Frobenius manifold structure of topological Landau-Ginzburg models. The birth of mirror symmetry explodes the context of singularity or Landau-Ginzburg theory both in the mathematics and physics literature. As a consequence, primitive forms not only produce intrinsic structures associated to singularities, but also predict geometric quantities as periods on Calabi-Yau manifolds, Gromov-Witten type invariants, etc.

The concept of higher residues  \cite{Saito-lecutures, Saito-residue} plays a key role in the theory of primitive forms. Motivated by mirror symmetry, Barannikov and Kontsevich \cite{Barannikov-Kontsevich} constructed a large class of Frobenius manifold structures on the formal extended moduli space of complex structures on Calabi-Yau manifolds. In the Calabi-Yau case,
the structures of higher residues were formulated by Barannikov as the notion of \emph{variation of semi-infinite   Hodge structures} \cite{Barannikov-thesis, Barannikov-period}, within which the \emph{semi-infinite period map} plays the role of a primitive form. Related concepts on semi-infinite Hodge structures are further interpreted by Givental as the symplectic geometry of Langrangian cone in a loop space formalism \cite{Givental-symplectic}. In \cite{Barannikov-projectivespace}, Barannikov used oscillating integrals to obtain the variation of semi-infinite   Hodge structures for the singularity describing the mirror of projective spaces, and proved the corresponding mirror symmetry. This construction is generalized by Sabbah \cite{Sabbah} and Douai-Sabbah \cite{Douai-Sabbah-I,Douai-Sabbah-II} to wider examples of Laurent polynomials. We refer to  \cite{Hertling-book} for an introduction to this subject.

The Barannikov-Kontsevich construction is closely modeled on the Kodaira-Spencer gauge theory introduced by the physicists Bershadsky, Cecotti, Ooguri and Vafa \cite{BCOV} (which we will call BCOV theory). The systematic mathematical treatment of BCOV theory has been developed recently by Costello and the second author \cite{Si-BCOV} based on the effective quantization of Givental's symplectic formalism. The precise relation between BCOV theory and the variation of semi-infinite  Hodge structures is explained in  \cite{Si-Frobenius}. See \cite{Si-review} also for a review.

The BCOV theory is a gauge theory of polyvector fields on Calabi-Yau manifolds, describing the closed string field theory of the B-twisted topological string.  It has a natural extension \cite{Si-LG} to Landau-Ginzburg model for the pair $(X, f)$, where $X$ is a Stein manifold and $f$ is a holomorphic function on $X$ with finite critical set.  The gauge field is given by $\PV_c(X)[[t]]$, where $\PV_c(X)$ is the space of smooth polyvector fields with compact support and $t$ is a formal variable. The action functional is constructed with the help of the \emph{trace map} (See section \ref{polyvector-descendant})
$$
\Tr: \PV_c(X)[[t]] \to \C[[t]]
$$
by integrating out the polyvector fields with respect to a choice of holomorphic volume form. Surprisingly, this simple integration map lifts the third author's higher residue map at the cochain level (Section \ref{polyvector-descendant}). In particular, the ordinary residue is lifted at the cochain level as the leading order of the trace map (Proposition \ref{compatible-residue}). It is then an extremely intriguing question whether there is a deep connection between the theory of primitive forms and BCOV theory. A related work from string theoretical viewpoint was provided by Losev \cite{Losev}.

The initial purpose of the current paper is to reveal this connection and develop a complex differential geometric approach to the theory of higher residues and primitive forms from the viewpoint of BCOV theory. It turns out that such viewpoint leads to an algebraic perturbative formula of primitive forms, unifying the well-known results on $ADE$ and simple elliptic singularities. In particular, this allows us to compute the potential function of the associated Frobenius manifold structure up to any finite order for arbitrary weighted homogeneous singularities. This is applied in \cite{LLS} to prove the  mirror symmetry for exceptional unimodular  singularities.

Another Hodge theoretical aspect is called the $tt^*$-geometry discovered by Cecotti and Vafa \cite{tt} , whose integrability structure was studied by Dubrovin \cite{Dubrovin-tt}. The algebraic formulation of $tt^*$-geometry was due to Hertling \cite{Hertling-tt}, and recently Fan \cite{Fan} presented an approach through harmonic analysis in the spirit of $N = 2$ supersymmetry  \cite{Cecotti-SUSY}. One essential aspect of $tt^*$-geometry is about the real structure, which is naturally built into the smooth forms and their Hodge theoretical aspects.  This is another motivation for us to develop the theory of higher residues inside the smooth category.

\subsection{Higher residue and primitive forms}
To explain primitive forms, we start with a pair $(X, f)$, where $X\subset \C^n$ is a  Stein domain,  and $f:X\to \C$ is a holomorphic function with finite critical set.
With respect to the choice  of a nonwhere vanishing holomorphic volume form $\Omega_X$ on $X$, we can identify the space $\PV(X)$ of smooth polyvector fields with the space $\mc A(X)$ of smooth complex differential forms on $X$, by using contraction operator $\vdash$ with $\Omega_X$. This gives an isomorphism of cochain complexes (Definition \ref{isom-PV-Diff})
                  $$(\PV(X)((t)), Q_f)\to (\mc A(X)((t)), d+{df\over t}\wedge),$$
where $Q_f$  is the coboundary operator on $\PV(X)((t))$ that is defined by using $d+{df\over t}\wedge$ together with the isomorphism $\PV(X)((t))\cong\mc A(X)((t))$ of vector spaces. Let $\PV_c(X)\subset \PV(X)$ be the subspace of smooth polyvector fields with compact support. It is easy to observe that
the natural embedding of complexes
$$
   \iota: (\PV_c(X)[[t]], Q_f) \into (\PV(X)[[t]], Q_f)
$$
is in fact a quasi-isomorphism. Therefore, we obtain a canonical  isomorphism
   $$\mc H_{(0)}^{f, \Omega}:=H^*(\PV(X)[[t]], Q_f)\overset{\,\iota^{-1}}{\to} H^*(\PV_c(X)[[t]], Q_f).$$
Combining this with the trace map defined by
 $$\Tr: \PV_c(X)\to \C;\,\, \alpha\mapsto\Tr(\alpha):=\int_X(\alpha\vdash \Omega_X)\wedge \Omega_X,$$
we obtain a map  $\widehat{\mbox{Res}}_f$ as the composition of $\Tr$ and the inverse of $\iota$ in the following diagram
$$
\xymatrix{
   \mc H^{f, \Omega}_{(0)}:= H^*( \PV(X)[[t]], Q_{f})\ar[r]^{{}\qquad\iota^{-1}}\ar[dr]_{\widehat{\mbox{Res}}^f} & H^*( \PV_c(X)[[t]], Q_{f})\ar[d]^{\Tr}\\
    & \C[[t]]
    }
$$
One of our main observations is that $\widehat{\mbox{Res}}^f$ realizes the original higher residue map in our smooth set-up of polyvector fields. More precisely,
every holomorphic function $g$ on $X$ represents a cohomology class $[g]$ in $\mc H^{f, \Omega}_{(0)}$. If we write
$$
\widehat{\Res}^f([g])=\sum_{k\geq 0}\widehat{\Res}_{(k)}([g])t^k,
$$
then $\widehat{\Res}_{(0)}(g)$ recognizes the ordinary residue (Proposition \ref{compatible-residue}), and $\{\widehat{\Res}_{(k)}\}_{k}$ constitutes the tower of higher residues.  The higher residue pairing
  $$\mc K^f_{\Omega}(\mbox{-}, \mbox{-}): \Hzero^{f, \Omega}\times \Hzero^{f, \Omega}\to \C[[t]]$$ is constructed in a similar fashion such that $\widehat{\Res}(\mbox{-})=\mc K^f_{\Omega}\bracket{\mbox{-},[1]}$.

Primitive forms are defined with respect to the universal unfolding $F$ of $f$ parametrizes by $S$. We assume that $o\in S$ is the reference point such that $F|_{o}=f$, and fix a family of holomorphic volume form $\Omega$. The cohomology $\mc H^{f, \Omega}_{(0)}$ extends to an    $\mc O_S[[t]]$-module  $\mc H^{F, \Omega}_{(0)}$ on $S$, and the higher residue pairing extends to
$$
\mc K^F_{\Omega}(\mbox{-}, \mbox{-}): \Hzero^{F, \Omega}\times \Hzero^{F, \Omega}\to \OO_S[[t]]
$$
with a compatible Gauss-Manin connection
$$
   \nabla^{\Omega}: \mc H^{F, \Omega}_{(0)}\to \Omega_S^1\otimes t^{-1}\mc H^{F, \Omega}_{(0)}
$$

The basic properties of the triple $(\Hzero^{F, \Omega}, \mc K^F_\Omega(\mbox{-}, \mbox{-}), \nabla^\Omega)$ are summarized to define a variation of semi-infinite Hodge structures (see section \ref{sec-VSHS}).
A primitive form is a section $\zeta\in \Gamma(S, \Hzero^{F, \Omega})$ that satisfies the following four properties: \textit{Primitivity}, \textit{Orthogonality}, \textit{Holonomicity} and \textit{Homogeneity} (see Definition \ref{define-primitive-form}). Such definition is equivalent to   the original description of primitive forms (see Theorem \ref{rmk-identifyprimitiveforms}).

Since we are working with polyvector fields, we use the  superscript notation to distinguish from the original notation of $\mc H_f$, $\mc K_f$, etc.  We have made a choice of the family of holomorphic volume form $\Omega$. However, different choices are essentially equivalent (see section \ref{subsec-intrinsicity}).  The combined volume form $\zeta \Omega$  represents a primitive form in the original approach \cite{Saito-primitive}. Therefore it is equivalent to work inside differential forms as in \cite{Saito-primitive}.  We have made extra efforts to present in the current form in order to compare with BCOV theory on compact Calabi-Yau manifolds \cite{Barannikov-period, Si-BCOV, Si-Frobenius}. This motivates the perturbative formula for primitive forms (Theorem \ref{thm-algorithm}) as a unification of Calabi-Yau and Laudan-Ginzburg models, and it sheds lights on the Landau-Ginzburg/Calabi-Yau correspondence in the B-model.

\subsection{Construction of primitive forms and applications}
In \cite{Saito-primitive}, the third author introduced the notion of \emph{good sections}, which we will call \emph{good opposite filtrations}, to produce analytic primitive forms by solving a version of Riemann-Hilbert-Birkhoff problem. It leads to
M. Saito's general solution on the existence of primitive forms \cite{Mo.Saito-existence} . This rather delicate construction provides a complete description of the space of primitive forms locally, but on the other hand makes the computation of primitive forms very difficult in general. Explicit expressions of primitive forms were only known, when the holomorphic function $f$ is given by either (1) two classes of weighted homogeneous polynomials (namely $ADE$-singularities and simple elliptic singularities)  \cite{Saito-primitive} or (2) a few Laurent polynomials  \cite{Takahashi,Barannikov-projectivespace,Douai-Sabbah-I, Douai-Sabbah-II}.  It has also been known very recently, when   $f$ is an affine cusp polynomial \cite{Ishibashi-Shiraishi-Takahashi-primtive, Shiraishi-Takahashi-primtive}.

In the present paper, we develop new techniques to describe primitive forms locally with respect to good opposite filtrations. The basic idea is to apply Barannikov's construction of semi-infinite period map for compact Calabi-Yaus to the Landau-Ginzburg case once we have completely similar settings via polyvector fields. In the formal neighborhood of $o\in S$, this leads to explicit expressions for the primitive forms which we now briefly describe. Let
$$
  \mc H^{f, \Omega}:=H^*( \PV(X)((t)), Q_{f})
$$
and $\mc H^{F, \Omega}$ be the corresponding extension to the universal unfolding. The vector space $\mc H^{f, \Omega}$ is equipped with a symplectic structure, with $\mc H^{f, \Omega}_{(0)}$ being an isotropic linear subspace. A good opposite filtration $\mc L$ is given by a splitting
$$
    \mc H^{f, \Omega}=\mc H^{f, \Omega}_{(0)}\oplus \mc L
$$
such that $\mc L$ is (1) isotropic, (2) preserved by multiplying by $t^{-1}$ and (3)  a version of $\C^*$-invariance (Definition \ref{def-good-oppofil}).  We will be working in the formal neighborhood of $o\in S$, and there exists a similar notion of formal primitive forms (Definition \ref{formal-primitive-form}). Let $\check{\mc H}^{F, \Omega}$ ($\check{\mc H}^{F, \Omega}_{(0)}$) denote the pull-back of $\mc H^{F, \Omega}$ ($\mc H^{F, \Omega}_{(0)}$) to the formal neighborhood. The advantage of formal setting is that the Gauss-Manin connection $\nabla^\Omega$ on $\check{\mc H}^{F, \Omega}$ is easily trivialized via the transformation $e^{(f-F)/t}$ (Lemma \ref{lem-flatextension}). Let $\mc L_{\mc R}\subset \check{\mc H}^{F, \Omega}$ be the extension of $\mc L$ with respect to $\nabla^\Omega$ (see Definition \ref{lem-def-LR} for precise descriptions). Then there is an induced splitting
$$
  \check{\mc H}^{F, \Omega}=\check{\mc H}^{F, \Omega}_{(0)}\oplus \mc L_{\mc R}
$$
Our main result is the following

\begin{thm}[Theorem \ref{thm-primitive-local}]\label{thm-introduction-main} %\hfill
   There is a bijection between  formal primitive forms and  pairs $(\mc L, \zeta_0)$, where $\mc L\subset \mc H^{f, \Omega}$ is a good opposite filtration and $\zeta_0\in \Hzero^{f, \Omega}$ is a primitive element with respect to $\mc L$. Precisely, the primitive form corresponding to $(\mc L, \zeta_0)$ is given by the projection of
 the formal section  $e^{f-F\over t}\zeta_0$ of $\check {\mc H}^{F, \Omega}=\check {\mc H}_{(0)}^{F, \Omega}\oplus \mc L_{\mc R}$ to $\check {\mc H}_{(0)}^{F, \Omega}$.
  \end{thm}
\noindent Furthermore if $f$ is a weighted homogenous polynomial with single critical point, then every formal primitive form is an
 analytic primitive forms in the germ of the universal unfolding (see Theorem \ref{thm-analytic}).

 The above theorem  unifies the semi-infinite period maps for   Calabi-Yau models and Landau-Ginzburg models.
It has particularly nice applications when $f$ is a weighted homogenous polynomial (of weight degree 1). In this case,  we  always take the standard  fixed family of volume form $\Omega=dz_1\wedge\cdots \wedge dz_n$, where $(z_1, \cdots, z_n)$ denotes the  coordinates of $X=\mathbb{C}^n$.
    %%We note that $\zeta_0\in \Hzero^{f, \Omega}$ is  a primitive element if and only if $\zeta_0$ is represented by a nonzero constant.
Take a set  $\{\phi_i\}_{i=1}^\mu$ of weighted homogeneous polynomials that represent a basis of the Jacobian ring, with the order of  their weight degrees ascending: $\deg\phi_1\leq \deg\phi_2\leq \cdots \leq \deg\phi_\mu$. Then we can completely determine the space of good opposite filtrations and primitive elements with elementary methods. Denote $r(i, j):=\deg\phi_i-\deg\phi_j$ and set
\begin{align*}
D:=\sharp\{(i, j) ~|~r(i, j)\in \mathbb{Z}_{>0}, i+j<\mu+1 \}
+\sharp\{(i, j)~|~r(i, j)\in \mathbb{Z}_{>0}^{\scriptsize\mbox{odd}}, i+j=\mu+1 \}.
\end{align*}
Applying the first part of  Theorem \ref{thm-introduction-main} to the weighted homogeneous polynomial $f$, we obtain

\begin{thm}[Theorem  \ref{thm-space-of-primitiveforms}] The moduli space of primitives forms,
  $$\mc M:=\{\mbox{primitive forms of }\Hzero^{F, \Omega}\}/\sim\,\,,$$
 is parametrized  by $\C^D$. Here we say two primitive forms $\zeta_1, \zeta_2$ to be equivalent, denoted as $\zeta_1\sim \zeta_2$, if
  $\zeta_1=c\zeta_2$ in the germ of the universal unfolding, for some nonzero constant $c\in\C^*$.
\end{thm}
\noindent
This statement was implicitly contained in a  related statement  for Brieskorn lattices \cite{Hertling-classifyingspace} (see also   \cite[Lemma 10.21]{Hertling-book}).
A dimension formula for arbitrary isolated singularities has  been given in \cite{Mo.Saito-uniqueness} recently. As a direct consequence of the theorem, there exists a unique primitive form (up to a scalar) for $ADE$-singularities  \cite{Saito-primitive} and exceptional unimodular singularities (Corollary \ref{thm-primitive-exceptionalsing}).

Applying the second  part of  Theorem \ref{thm-introduction-main} to  $f$, we obtain a  concrete algorithm  to compute the Taylor series expansions of primitive forms up to an arbitrary finite order. The precise formula is explained in \textbf{Theorem \ref{thm-algorithm}}. Actually,  such  perturbative formula
unifies the following known results  \cite{Saito-primitive}.   For $ADE$-singularities, the unique  primitive form is
    given by  $\zeta=1$;   for simple elliptic singularities, there is a   one-parameter family of primitive forms  (up to   a nonzero constant scalar), given by the period integrals on the corresponding elliptic curves.
   Beyond this, our   perturbative formula also give precise expressions of primitive forms for other cases up to a finite order. As an example, we consider $f=x^3+y^7$, which is an exceptional unimodular singularity of type $E_{12}$. The  primitive forms for it were unknown, due to
   the difficulty of the generally expected phenomenon of mixing between deformation parameters of positive and negative degrees. As an application of Theorem \ref{thm-algorithm}, the unique primitive $\zeta$ has the leading expression
   $$\zeta=1+{4\over 3\cdot 7^2}u_{11}u_{12}^2-{64\over 3\cdot 7^4}u_{11}^2u_{12}^4 -{76\over 3^2\cdot 7^4}u_{10}u_{12}^5 +({1\over 7^2}u_{12}^3-{ 101\over 5\cdot 7^4} u_{11}u_{12}^5)x  - {53 u_{12}^6\over 3^2 \cdot 7^4} x^2 \mod\mathfrak{m}^7,$$
    where $\mathfrak{m}$ denotes the maximal ideal of the reference point $0\in S$, and $u_{10}, u_{11}, u_{12}$ are part of the coordinates of $S$ appearing in the coefficients of  the deformations by $\phi_{10}=xy^3$, $\phi_{11}=xy^4$ and $\phi_{12}=xy^5$, respectively.
    It is clear from this result that the mixing phenomenon mentioned above appears in a nontrivial way, and non-trivial elements in the Jacobian ring are also generated along deformation, unlike the $ADE$ and simple elliptic singularities. This strongly suggests that the primitive form not only favors for special flat coordinates on the deformation space, but also for the fibration itself.
 All these cases are discussed in detail in the last   section of examples.

 %% It is well known that a primitive form of $\Hzero^{f, \Omega}$ (for general $f$, not necessary to be weighted homogeneous polynomials) induces a   Frobenius manifold structure  on the deformation space $S$ in a canonical way.
 Our perturbative formula for primitive forms are extremely useful for applications in mirror symmetry. The theory of primitive forms presents the mathematical context of Landau-Ginzburg B-model in physics. The mirror mathematical theory of Landau-Ginzburg A-model is established by Fan, Jarvis and Ruan \cite{FJRW} following Witten's proposal \cite{Witten}, which is nowadays popularized as the FJRW-theory. The following theorem is proved in \cite{LLS} as a direct application of our method:
 
 \begin{thm}[\cite{LLS}]
 There is an equivalence between the theory of primitive forms for Arnold's exceptional unimodular singularities and the FJRW-theory of their mirror singularities with maximal orbifold group. 
 \end{thm}
 
 Our current construction is mainly focused on the local theory, and it would be extremely interesting to understand the global behavior of primitive forms. Related works on global mirror symmetry along this line have appeared in \cite{Coates-Iritani-Tseng,Coates-Iritani,Milanov-Ruan,IMRS,Krawitz-Shen}.

The present paper is organized as follows: Section 2 is devoted to the basic set-up of polyvector fields. We formulate the complex differential geometric description of higher residues via the trace map on compactly supported polyvector fields. In section 3, we explain the variation of semi-infinite   Hodge structures associated with our model and define the corresponding primitive forms. In section 4, we describe the general constructions of primitive forms in the formal neighborhood of the universal unfolding and explain our perturbative formula. In section 5, we describe the moduli space of primitive forms for weighted homogeneous polynomials, and
 provide a concrete  algorithm to compute the Taylor series expansions of the primitive form. Finally in section 6, we apply our method to explicit examples.

\bigskip

\noindent \textbf{Acknowledgement}: The authors would like to thank   Alexey Bondal, Kwok-Wai Chan,  Kentaro Hori, Hiroshi Iritani, Todor Eliseev Milanov, Hiroshi Ohta, Mauricio Andres Romo Jorquera, Yefeng Shen, and especially, Claus Hertling,  Morihiko Saito and Atsushi Takahashi    for valuable comments and inspiring  discussions.
   The second author would like to thank in addition Huai-Liang Chang,
  Kevin Costello,  Yongbin Ruan, and Cumrun Vafa, for various communications on Landau-Ginzburg theory. Part of the work was done while the second author was visiting Kavli IPMU at University of Tokyo   and MSC at Tsinghua University, and he thanks for their hospitality.
 The first author is partially supported by JSPS Grant-in-Aid for Young Scientists (B) No. 25870175. The second author is partially supported by NSF DMS-1309118. The third author is partially supported by  JSPS Grant-in-Aid for Scientific Research (A) No. 25247004.
\section{Polyvector fields}\label{polyvector}

Throughout the present paper, we will
 assume $X$ to be  a Stein domain  in $\C^n$, and assume   $f: X\to \C$ to be a holomorphic function with
 finite  critical points Crit$(f)$. We will fix a nowhere vanishing holomorphic volume form $\Omega_X$ on $X$ once and for all.
As we will see in section \ref{subsec-intrinsicity} (Proposition \ref{primitive-change}), the choice of $\Omega_X$ is not essential.

As a notation convention, we will always denote by $[\mbox{-}, \mbox{-}]$ the graded-commutator. That is, for operators $A, B$ of degree $|A|, |B|$, the bracket  is given by
 \begin{align}\label{eqn-gradedcomm}
    [A, B]:=AB-(-1)^{|A|\cdot |B|}BA.
 \end{align}

\subsection{Compactly supported polyvector fields}
 Let $T_X$ denote the holomorphic tangent bundle of $X$, and  $\A^{0,j}_c (X, \wedge^i T_X) $ denote the space of compactly supported smooth $(0,j)$-forms valued in $\wedge^i T_X$. We consider
 the space of compactly supported smooth polyvector fields on $X$:
\begin{align*}
\PV_c(X)=\bigoplus_{0\leq i,j \leq n}\PV^{i,j}_c(X) \ \ \ \ \ \
\mbox{with}\quad \PV^{i,j}_c(X):= \mathcal A^{0,j}_c (X, \wedge^i T_X).
\end{align*}
 Clearly, $\PV_c(X)$ is a bi-graded  $C^\infty(X)$-module. As we will see below, the space  $\PV_c(X)$ is a double  complex: the total degree of $\alpha\in \PV_{c}^{i, j}(X)$ is defined\footnote{The convention for the degree of $\alpha$ is  $j+i$ in \cite{Si-BCOV}.}  to be $|\alpha|:=j-i$.
 There is a natural wedge product structure (see formula \eqref{eqnwedgeprod}) on polyvector fields, which makes $\PV_c(X)$ a graded-commutative algebra.

Using   the contraction   $\vdash$ with  $\Omega_X$,  we obtain an     isomorphism of $C^\infty(X)$-modules
between compactly supported polyvector fields and compactly supported differential forms
$$
   \Gamma_{\Omega_X}:\PV^{i,j}_c(X)\stackrel{\vdash \Omega_X}{\iso} \A_c^{n-i,j}(X).
$$
As a consequence, every operator $P$ on $\mc A_c(X):=\A_c^{*,*}(X)$ induces an operator $P_{\Omega_X}$ on $\PV_c(X)$:
$$
P_{\Omega_X}(\alpha):= \Gamma_{\Omega_X}^{-1}\big(P(\Gamma_{\Omega_X}(\alpha))\big), \quad \forall \alpha\in \PV_c(X).
$$
We will simply denote
 $P_{\Omega_X}$ as $P_\Omega$. Furthermore, if the induced operator $P_{\Omega_X}$ is independent of choices of $\Omega_X$, we simply denote it as $P$ by abuse of notation. For instance,
%where $\vdash$ denotes the contraction map.
  the differential operators $\dbar, \pa$ on differential forms, respectively, induce   operators    on polyvector fields
\begin{align*}
    \dbar_{\Omega_X}&: \PV_c^{*,*}(X)\to \PV_c^{*,*{\scriptscriptstyle+1}}(X)\quad\mbox{and}\quad  %\quad \dbar (\alpha):=\Gamma_{\Omega_X}^{-1}\big(\dbar (\Gamma_{\Omega_X}(\alpha))\big),\\
      \pa_{\Omega_X}: \PV_c^{*,*}(X)\to \PV_c^{*{\scriptscriptstyle -1},*}(X). %;  \quad \pa_{\Omega}(\alpha):=\Gamma_{\Omega_X}^{-1}\big(\pa (\Gamma_{\Omega_X}(\alpha))\big).
\end{align*}
The operator $\dbar_{\Omega_X}$  is simply denoted as  $\dbar$, while the operator  $\pa_{\Omega_X}$ that depends on $\Omega_X$  is denoted by $\pa_\Omega$. In other words,
$$
(\dbar \alpha)\vdash\Omega_X=\dbar(\alpha\vdash \Omega_X), \ \mbox{and}\  (\pa_\Omega \alpha)\vdash\Omega_X=\pa(\alpha\vdash \Omega_X)
$$
for any $\alpha\in \PV_c(X)$. Both $\dbar$ and $\pa_\Omega$ are of cohomology degree $1$.

The operator $\pa_\Omega$ further defines a bracket     $$\{\mbox{-}, \mbox{-}\}: \PV_c(X)\times \PV_c(X)\rightarrow \PV_c(X),$$
\begin{align}\label{bracked}
  \fbracket{\alpha, \beta}:=\pa_\Omega\bracket{\alpha\wedge\beta}-\bracket{\pa_\Omega\alpha}\wedge\beta-(-1)^{|\alpha|}\alpha\wedge\pa_\Omega\beta
\end{align}
 for   $\alpha\in\PV^{i, j}_c(X)$ and $\beta\in \PV_c^{k, l}(X)$.
 Here $\wedge$ is the wedge product (see formula \eqref{eqnwedgeprod}). Unless otherwise stated, we will simply denote
 $\alpha\beta:=\alpha\wedge\beta$. In the special case when the polyvector field $\alpha=g$ is a function, one can easily check that
   \begin{equation}\label{eqnWbracket}
      \{g, \beta\}\vdash \Omega_X=\pa g\wedge (\beta\vdash\Omega_X).
   \end{equation} %, which equips $\PV_c(X)$ an algebraic structure.

 The bracket $\{\mbox{-}, \mbox{-}\}$  corresponds to the Schouten-Nijenhuis bracket (up to a sign), which is independent of the choice of $\Omega_X$.

We   also consider the larger space
$$
\PV(X)=\bigoplus\limits_{0\leq i, j\leq n} \PV^{i, j}(X)
$$
of smooth polyvector fields (not necessarily compactly supported), which is naturally identified with the smooth differential forms  by extending $\Gamma_{\Omega_X}$ above.
The operators $\dbar$, $\pa_\Omega$ and  $\{\mbox{-}, \mbox{-}\}$  are all defined on $\PV(X)$.

\begin{rmk} The triple  $\big\{\{\mbox{-}, \mbox{-}\}, \wedge, \pa_\Omega\big\}$  make  $\PV(X)$ a   \emph{Batalin-Vilkovisky algebra}, which is used in \cite{Bogomolov, Tian, Todorov, Barannikov-Kontsevich} to prove the Formality Theorem for polyvector fields in Calabi-Yau geometry. %, in which way it is also called the Todorov-Tian's lemma. Relevant  details   can be found in

 \end{rmk}

We describe the above constructions in   coordinates $\mathbf{z}=(z_1, \cdots, z_n)$ of $\mathbb{C}^n$ and fix our conventions. We will write
$$
\Omega_X= {dz_1\wedge dz_2\wedge\cdots \wedge dz_n \over \lambda(\mathbf{z})}
$$
where $\lambda: X\to \C$ is a nowhere vanishing holomorphic function on $X$.  For an ordered subset $I=\{a_1, a_2, \cdots, a_i\}$ of $\{1, 2, \cdots, n\}$ with $|I|=i$, we use the %following
 following notations:
 $$
 d\bar z_I:=d\bar z_{a_1} \wedge\cdots \wedge d\bar z_{a_i}, \quad  {\partial_I}:={\partial\over \partial z_{a_1}}\wedge{\partial\over \partial z_{a_2}}\wedge\cdots \wedge {\partial\over \partial z_{a_i}}.
 $$
In particular, we will use $\pa_k$ for the polyvector field ${\pa\over \pa z_k}\in \PV^{1, 0}(X)$ in order to distinguish it from the partial derivative
 $\pa_{z_k}:= {\pa\over \pa z_k}$ on $C^\infty(X)$. We will occasionally use ${\pa\over \pa z_k}$ when its meaning is self-evident.

 For $\alpha\in \PV^{i, j}(X)$ and
 $\beta\in \PV^{k, l}(X)$,
  $$
  \alpha=\sum\limits_{|I|=i, |J|=j}\alpha^I_{J} d\bar z_{J}\otimes {\partial_I}
,\quad  \beta=\sum_{|K|=k, |L|=l}\beta^K_{ L}d\bar z_{L}\otimes {\partial_K},
  $$
where $\alpha^I_J$, $\beta^K_L\in C^\infty(X)$,
 the wedge product is given by
   \begin{equation}\label{eqnwedgeprod} \alpha\wedge \beta :=\sum_{I, J, K, L}(-1)^{il}\alpha^I_{J}\beta^K_{L}(d\bar z_{J}\wedge d\bar z_{L})\otimes  ({\partial_I} \wedge {\pa_K}),\end{equation}
which satisfies the graded-commutativity
$$
\alpha\wedge\beta=(-1)^{|\alpha|\cdot|\beta|}\beta \wedge \alpha.
$$

We will use  ${\pa\over \pa\pa_r}$ for the derivation with respect to the polyvector \mbox{field $\pa_r$:}
$$
{\pa \over {\pa{\pa_r}}}\bracket{\alpha^I_Jd\bar z_{J}\otimes {\partial_I}}:=\begin{cases}(-1)^{|J|+m-1}\alpha^I_Jd\bar z_{J}\otimes{\pa_{a_1}}\cdots \pa_{a_{m-1}} {\partial_{a_{m+1}}}\cdots {\partial_{a_i}},& \mbox{if} \ a_{m}=r\ \mbox{for some}\ m,\\
0, & \mbox{otherwise}.
\end{cases}
$$
These operators are of degree $1$, satisfying the graded-commutativity relations
$$
   {\pa\over \pa\pa_k}  {\pa\over \pa\pa_l}=-  {\pa\over \pa\pa_l}  {\pa\over \pa\pa_k}, \quad \forall 1\leq k,l\leq n.
$$

Under the above conventions, we have
$$
 \dbar \alpha=\sum_{I, J}(\dbar \alpha^I_{J})\wedge d\bar z_{J}\otimes {\partial_I},\quad
  \pa_\Omega \alpha=\sum_{I, J}\sum_{r=1}^n{\lambda} \pa_{z_r}\big({\alpha^I_{J}\over \lambda}\big){\pa \over \pa{\pa_r}}(d\bar z_{J}\otimes {\partial_I}),
  $$
and
 $$\{\alpha, \beta\}=\sum_{r=1}^n \Big({\pa\over \pa\pa_r} \alpha\Big)\wedge\Big(\sum_{K, L}(\pa_{z_r}\beta^{K}_L)d\bar z_{L}\otimes {\pa_K}\Big)
              +(-1)^{|\alpha|}\sum_{r=1}^n \Big(\sum_{I, J}(\pa_{z_r}{\alpha^I_J}) d\bar z_{J}\otimes {\pa_I}\Big)\wedge {\pa\over \pa\pa_r} \beta.
$$

 \subsection{Cohomology of polyvector fields}\label{subsec-coh-of-PV} In this subsection, we will study the cohomology of (compactly supported) polyvector fields with respect to the following  coboundary operator
  $$\dbar_f:=\dbar+\fbracket{f,\mbox{-}},$$
acting on $\PV(X)$ and preserving $\PV_c(X)$.
It is easy to see that $\dbar_f^2=0$ and $\dbar_f g=0, \forall g\in \Gamma(X, \OO_X)$. Let $\Jac(f):=\Gamma(X, \mc O_X)/(\pa_{z_1}f, \cdots, \pa_{z_n}f)$, which is called the \textit{Jacobian ring} or  \textit{Chiral ring}.
\begin{lem}\label{cohomology-polyvectorfield} $H^k(\PV(X), \dbar_f)$ vanishes unless $k=0$. Furthermore, there is a natural isomorphism
 {\upshape $\Jac(f)\overset{\cong}{\rightarrow} H^0(\PV(X), \dbar_f)$} sending  $[g]$ in $\Jac(f)$, $g\in \Gamma(X, \OO_X)$, to the cohomology class $[g]$.
\end{lem}
\begin{proof}   $\PV(X)$ is a double complex with horizontal operator $\{f, \mbox{-}\}$ and vertical operator $\dbar$.  We consider the spectral sequence associated to the descending filtration
$$
\mc F^k \PV(X)=\bigoplus_{i\leq n-k}\PV^{i, *}(X)
$$
for $k\in \mathbb{Z}$. Since $X$ is Stein, the $E_1$-term is given by holomorphic sections  $H^*(\PV(X), \dbar)=\bigoplus_k H^0(X, \wedge^k T_X)$, with the differential $\{f, \mbox{-}\}$. Since Crit$(f)$ is   finite,  by the de Rham lemma in \cite{Saito-deRham},  the cohomology of the holomorphic Koszul complex
$$
  0\overset{\{f, \mbox{-}\}}{\longrightarrow} H^0(X, \wedge^n T_X)\overset{\{f, \mbox{-}\}}{\longrightarrow} H^0(X, \wedge^{n-1}T_X)
        %\stackrel{\vdash dW}{\to}
       \overset{\{f, \mbox{-}\}}{\longrightarrow} \cdots \overset{\{f, \mbox{-}\}}{\longrightarrow} H^0(X, \OO_X)\overset{\{f, \mbox{-}\}}{\longrightarrow} 0,
$$
which is $E_2$,
is concentrated at the top term ${H^0(X, \OO_X)\over \Im\big(\{f, \mbox{-}\}:  H^0(X, T_X)\to H^0(X, \OO_X)\big)}=\Jac(f)$. Clearly, the  spectral sequence degenerates at $E_2$-stage,   and the statement follows.
\end{proof}

In order to represent  cohomology classes by compactly supported polyvector fields, we introduce some operators.
Let
$$
\|df\|^2:=\sum\limits_{k=1}^n \abs{\pa_{z_k} f}^2,
$$
and we consider the following operator
\begin{align*}
    V_f= \sum_{i=1}^n {\overline{\pa_{z_i} f}\over \|df\|^2} \pa_{i}\wedge:\,\,\, \PV^{*, *}({X\setminus \mbox{Crit}(f)})\longrightarrow \PV^{*{\scriptscriptstyle +1}, *}({X\setminus  \mbox{Crit}(f)})\,.
\end{align*}
Let $\rho$ be a smooth cut-off function on $X$, such that $\rho|_{U_1}\equiv 1$ and $\rho|_{X\setminus U_2}\equiv 0$
 for some  relatively compact open neighborhoods $U_1\subset \overline{U_1}\subset U_2$ of Crit$(f)$ in $X$. %with smooth boundaries  $\pa U_1$, $\pa U_2$.

The operators $\dbar$ and $V_f$ are of degree $1$ and $-1$ respectively (where  defined), and
 $[\dbar, V_f]=\dbar V_f+ V_f\dbar$ is of degree $0$. We   define
 $$
 T_\rho: \PV(X)\to \PV_c(X), \quad \mbox{and}\quad R_\rho: \PV(X)\to \PV(X)
 $$
 by
     $$T_\rho(\alpha):=\rho \alpha+(\dbar\rho)V_f {1\over 1+[\dbar, V_f]}(\alpha),\quad R_\rho(\alpha):= (1-\rho) V_f {1\over 1+[\dbar, V_f]}(\alpha), \quad \alpha\in \PV(X).
$$
Here as an operator
$$
{1\over 1+[\dbar, V_f]}:=\sum\limits_{k=0}^\infty(-1)^k[\dbar, V_f]^k,
$$
which is well-defined since $[\dbar, V_f]^k(\alpha)=0$ for any $k>n$.
%%Clearly,  $T_\rho$, $R_\rho$ are both well-defined on the whole space $\PV(X)$.

\begin{lem}\label{lemmaforquasiiso}
 %%{\upshape $\dbar_W  R_\rho + R_\rho \dbar_W=\mbox{id}-T_\rho$  }
  {\upshape $[\dbar_f,  R_\rho]=1-T_\rho$  } as operators on $\PV(X)$.
\end{lem}

\begin{proof}
% For any $\alpha\in \PV(X)|_{X\setminus \scriptsize \mbox{Crit}(W)}$, we note $(V_W(\alpha)) \vdash \Omega_X= V_W\vdash (\alpha\vdash \Omega_X)$, on the right hand side of which  $V_W$ is treated as an element in $\PV^{1, 0}(X)|_{X\setminus \scriptsize \mbox{Crit}(W)}$ naturally.
It is easy to see that
  $$[\{f, \mbox{-}\},  V_f]=1\quad \text{on}\  \PV({X\setminus   \mbox{Crit}(f)}).
  $$
Moreover,
       $$[P, [\dbar, V_{f}]]=0$$
 for   $P$ being   $\{f, \mbox{-}\}, \dbar$ or $V_{f}$. Therefore, where defined,
   \begin{align*}
             [\dbar_f,  R_\rho]
                           &=[\dbar_f, 1-\rho]V_f {1\over 1+[\dbar, V_f]}+(1-\rho)[\dbar_f,  V_f] {1\over 1+[\dbar, V_f]} \\
                           &=-(\dbar\rho)V_f {1\over 1+\bbracket{\dbar, V_f}}+(1-\rho)\\
                          &=1-T_\rho.
         \end{align*}
\end{proof}

\begin{cor}\label{quasi-isom}
The embedding $\bracket{\PV_c(X), \dbar_f}\into \bracket{\PV(X),\dbar_f}$ is a quasi-isomorphism.
\end{cor}

\begin{proof}
  This follows from Lemma \ref{lemmaforquasiiso} that $
  H^*(\PV(X)/\PV_c(X), \dbar_f)\equiv0.
  $
\end{proof}

We will let $\iota$ be the natural embedding
 $$
 \iota: \PV_c(X)\hookrightarrow  \PV(X).
 $$
Corollary \ref{quasi-isom} and Lemma \ref{cohomology-polyvectorfield} lead to a canonical  isomorphism
$$
\iota: H^*\bracket{\PV_c(X), \dbar_f}%%\overset{\cong}{\to}  H^*\bracket{\PV_c(X), \dbar_W}
  \overset{\cong}{\longrightarrow}  \Jac(f),
$$
which we still denote by $\iota$. By Lemma \ref{lemmaforquasiiso}, a representative of $\iota^{-1}([g])$ for $[g]\in \Jac(f)$, where $g\in \Gamma(X, \OO_X)$, is given by the compactly supported polyvector field $T_\rho(g)$. In particular, the class $[T_\rho(g)]$ represented by $T_\rho(g)$ depends only on the class $[g]\in \Jac(f)$, not on the choice of the coordinates or the cut-off function $\rho$.

For compactly supported polyvector fields, we can  define the \emph{trace map}
 %$\Tr: \PV_c(X)\to \C$ by
\begin{align*}
  \Tr: \PV_c(X)\to \C,\, \Tr(\alpha):=\int_X \bracket{\alpha\vdash \Omega_X}\wedge \Omega_X.
\end{align*}
This induces  a non-degenerate
pairing  $\abracket{\mbox{-},\mbox{-}}:   \PV_c(X)\otimes \PV_c(X) \to \C$  defined by
\begin{align*}
%%    \PV_c(X)\otimes \PV_c(X) &\to \C\\
        \alpha\otimes \beta  &\mapsto \abracket{\alpha,\beta}:= \Tr\bracket{\alpha\beta}.
\end{align*}
For any $\alpha\in \PV_c^{i, j}(X)$ and $\beta\in\PV_c^{k, l}(X)$, it is easy to see that
%see that $\dbar$ is skew self-adjoint for this pairing and $\pa$ is self-adjoint.
 \begin{align}\label{Trproperties}
    \langle \dbar \alpha, \beta\rangle =-(-1)^{|\alpha|}\langle \alpha, \dbar \beta\rangle \quad  \mbox{ and } \quad
  \langle \pa_\Omega\alpha, \beta\rangle =(-1)^{|\alpha|}\langle \alpha, \pa_\Omega \beta\rangle.
   \end{align}
Furthermore, we have  $$\Tr(\dbar \alpha)=0\quad\mbox{and}\quad \Tr(\{f, \alpha\})=0,$$
which imply   that the trace map is well defined on the cohomology
$$  \Tr:  H^*(\PV_c(X), \dbar_f) \to \C. $$
The above pairing descends to a pairing $\abracket{\mbox{-},\mbox{-}}$ on the cohomology which will be non-degenerate as well (Proposition \ref{compatible-residue}). %$:  H^*(\PV_c(X), \dbar_W) \otimes H^*(\PV_c(X), \dbar_W) \to \C$.

The trace map is closely related to the residue as follows. Let $\{V_1, \cdots, V_n\}$ be a $\OO_X$-basis of holomorphic  vector fields on $X$ such that
 $(V_n\wedge V_{n-1}\wedge \cdots \wedge V_1)\vdash \Omega_X\equiv 1$. Then $V_1(f), \cdots, V_n(f)$ form a regular sequence on $X$, and we have the concept of residue map (see e.g. \cite{Hartshorne-residue}, \cite{Griffiths-Harris} for details)
 $$
  \Res: \Jac(f)\to \C, \quad g\to \Res(g)=\sum_{x\in\scriptsize \mbox{Crit}(f)} \Res_x {g\Omega_X\over V_1(f)V_2(f)\cdots V_n(f)}.
$$
which is independent of the choice of $\{V_1,\cdots, V_n\}$ (but depends on  the choice of $\Omega_X$). In  coordinates $\mathbf{z}$ of $\C^n$, we have
 $$
 \Res(g)=\sum_{x\in\scriptsize \mbox{Crit}(f)} \Res_{x} {g(z) dz_1\wedge dz_2\wedge\cdots \wedge dz_n\over \lambda^2 (\pa_{z_1}f)\cdots (\pa_{z_n}f)}
 $$
 for $\Omega_X= {dz_1\wedge dz_2\wedge\cdots \wedge dz_n \over \lambda(\mathbf{z})}$.

One of the key observations is that the trace map is compatible with the residue map in the following sense.
\begin{prop}\label{compatible-residue}
There exists a constant $C_n$, depending only on $n=\dim X$, such that the following diagram commutes
$$
   \xymatrix{ \Jac(f)\ar[dr]_{C_n \Res}\ar[r]^{\iota^{-1}\qquad{}}& H^*(\PV_c(X), \dbar_f) \ar[d]^{\Tr}\\ & \C  }
$$
\end{prop}
\begin{proof}   The statement follows from similar calculations as in \cite{Griffiths-Harris} page 654.  The proof is given below, while our reader may skip the details.

Let Crit$(f)=\{x_1, \cdots, x_r\}$. %% be  critical loci  of $W$ (where $p_i$'s are distinct).
For  $0<\varepsilon<<1$, we denote by $B_{\varepsilon}$   a closed ball  of radius $\varepsilon$  contained in  $X$, $S_{\varepsilon}$ being the boundary of $B_\varepsilon$.
We choose $\varepsilon_1, \cdots, \varepsilon_r$ %(where $1\leq r\leq m$)
and small disjoint balls $B_{2\varepsilon_i}$ such that each $\mbox{Int}(B_{\varepsilon_i})$ contains precisely $x_i$ in Crit$(f)$, hence Crit$(f)\subset \bigcup_{i=1}^r \mbox{Int}(B_{\varepsilon_i})$.
Take a smooth cut-off function $\rho$ such that
$$
\rho|_{\bigcup_{i=1}^r  B_{\varepsilon_i}}\equiv 1, \quad \rho|_{X\setminus \bigcup_{i=1}^r  B_{2\varepsilon_i}}\equiv 0.
$$
Let  $X'=\bigcup_{i=1}^r \mbox{Int}(B_{2\varepsilon_i})-\bigcup_{i=1}^r B_{\varepsilon_i}$. For
 $1\leq k\leq n$, we denote  $f_k:=\pa_{z_k} f$ and $h_k:={\bar f_k\over \|df\|^2}$, where
   we recall $\|df\|^2=\sum\limits_{j=1}^n|f_j|^2$.

  The operator $T_\rho$ gives a realization of $\iota^{-1}$. Let $g$ be a holomorphic function on $X$ representing an element of $\Jac(f)$. Let
  $$
  A_j:= V_f[\dbar, V_f]^j(g)\in\PV^{j+1, j}(X).
  $$

\noindent \textbf{Claim}: $\dbar A_{n-1} =0$ and
    $$A_{n-1}=C_n'g\sum_{i=1}^n(-1)^{i-1}{\bar f_i \over \|df\|^{2n}}d\bar f_1\wedge \cdots \wedge \widehat{d\bar f_i}\wedge \cdots \wedge d \bar f_n\otimes (\pa_1\cdots \pa_n)$$
for some nonzero constant $C_n'$   depending only on $n$.

Let us assume the Claim first, and let
   $\Omega_{X'}:=\Omega_X|_{X'}$. Then
  \begin{align*}
    & \Tr(\iota^{-1}(g))= \Tr(T_\rho g)=\Tr((1-\rho)g+(\dbar \rho)\sum_{k=1}^n(-1)^kA_k)\\
            &=(-1)^{n-1}\int_X\big(((\dbar \rho)A_{n-1})\vdash\Omega_X\big)\wedge \Omega_{X}=(-1)^{n-1}\int_{X'}\big(((\dbar \rho)A_{n-1})\vdash\Omega_{X'}\big)\wedge \Omega_{X'}\\
 %&=(-1)^{n-1}\int_{X'}\dbar \big((\rho A_{n-1}\vdash\Omega_{X'})\wedge \Omega_{X'}\big)\\
          &=(-1)^{n-1}\int_{X'}d \big((\rho A_{n-1}\vdash\Omega_{X'})\wedge \Omega_{X'}\big)  =(-1)^{n-1}\int_{\pa X'}(\rho A_{n-1}\vdash\Omega_{X'})\wedge \Omega_{X'}\\
          &  =(-1)^{n}\sum_{j=1}^r\oint_{S_{\varepsilon_j}}(A_{n-1}\vdash\Omega_{X'})\wedge \Omega_{X'}.
  \end{align*}
Hence, there exists  nonzero constant $C_n$  depending only on $n$ such that
 \begin{align*}
     \Tr(T_\rho g)%%%&=\Tr((1-\rho)g+(\dbar \rho)\sum_{k=1}^d(-1)^kA_k)\\
          %%% &=(-1)^{d-1}\int_X\big(((\dbar \rho)A_{d-1})\vdash\Omega_X\big)\wedge \Omega_{X}\\
         %%%  &=(-1)^{d-1}\int_{X'}\big(((\dbar \rho)A_{d-1})\vdash\Omega_{X'}\big)\wedge \Omega_{X'}\\
            %&=(-1)^{d-1}\int_{X'}\dbar \big((\rho A_{d-1}\vdash\Omega_{X'})\wedge \Omega_{X'}\big)\\
         %%%  &=(-1)^{d-1}\int_{X'}d \big((\rho A_{d-1}\vdash\Omega_{X'})\wedge \Omega_{X'}\big)\\
         %%%  &=(-1)^{d-1}\int_{\pa X'}(\rho A_{d-1}\vdash\Omega_{X'})\wedge \Omega_{X'}\\
        %%%   &=(-1)^{d}\sum_{j=1}^r\oint_{S_{\varepsilon_j}}(A_{d-1}\vdash\Omega_{X'})\wedge \Omega_{X'}\\
          &=C_n\sum_{j=1}^r\oint_{S_{\varepsilon_j}}{1\over \lambda^2}g\sum_{i=1}^n(-1)^{i-1}{\bar f_i \over \|df\|^{2n}}d\bar f_1\wedge \cdots \widehat{d\bar f_i}\wedge \cdots \wedge d \bar f_n \wedge dz_1\cdots \wedge dz_n\\
          &\overset{(*)}{=}C_n \sum_{x_j}\mbox{Res}_{x_j}{g\over (\pa_{{z}_1}f)\cdots (\pa_{z_n}f)}\cdot {1\over \lambda^2} dz_1\wedge \cdots \wedge dz_n\\
          &=C_n \Res(g).
  \end{align*}
 The equality $(*)$  follows from the Residue Theorem (see e.g. section 5.1  of \cite{Griffiths-Harris}).

 It remains to show our claim. We have
 $$
 V_f=\sum_i h_i \pa_i, \quad  \bbracket{\dbar, V_f}^k(g)=\bracket{\sum_i\dbar h_i \pa_i}^k g
 $$
 where $\dbar h_i={\dbar \bar f_i\over \|df\|^2}-{\bar f_i \over \|df\|^4}\dbar \|df\|^2$. Using
 $$
 \dbar \|df\|^2\wedge \dbar \|df\|^2=0\ \mbox{and}\  \bracket{\sum_i h_i \pa_i} \bracket{\sum_i {\bar f_i \over \|df\|^4}\dbar \|df\|^2 \pa_i}=-{\dbar \|df\|^2\over \|df\|^2} \bracket{\sum_i h_i \pa_i}^2=0,
 $$
 it follows easily that
 \begin{align*}
    A_{n-1}&=\bracket{\sum_i h_i \pa_i}\bracket{\sum_i {\dbar \bar f_i\over \|df\|^2}\pa_i }^{n-1} g-(n-1)\bracket{\sum_i h_i \pa_i}\bracket{\sum_i{\dbar \bar f_i\over \|df\|^2}\pa_i}^{n-2} \bracket{\sum_i {\bar f_i \over \|df\|^4}\dbar \|df\|^2\pa_i}g\\
    &=\bracket{\sum_i h_i \pa_i}\bracket{\sum_i {\dbar \bar f_i\over \|df\|^2}\pa_i }^{n-1} g\\
    &=(-1)^{n(n-1)/2}(n-1)! g\sum_{i=1}^n(-1)^{i-1}{\bar f_i \over \|df\|^{2n}}d\bar f_1\wedge \cdots \widehat{d\bar f_i}\wedge \cdots \wedge d \bar f_n\otimes (\pa_1\cdots \pa_n).
 \end{align*}
The equation $\dbar A_{n-1}=0$ follows from the above formula. This proves the Claim.
\end{proof}

\subsection{Polyvector fields with a descendant variable}\label{polyvector-descendant}
Following  \cite{Saito-lecutures, Saito-residue}, we consider the  vector space of $\PV(X)$ (resp. $\PV_c(X)$) valued Laurent series in $t$:
$$
   \PV(X)((t))=\PV(X)[[t]][t^{-1}] \quad(\mbox{resp.}\quad \PV_c(X)((t))=\PV_c(X)[[t]][t^{-1}] ).
$$
Here $t$ is a formal variable of cohomology degree zero, called descendant variable.

\begin{rmk}
 We are following the notation $t$ in \cite{Costello, Si-BCOV} for the   descendant variable, which   is also denoted by various notations like $\delta^{-1}$ \cite{Saito-lecutures, Saito-residue, Saito-unfolding,Saito-primitive}, $\hbar$
  \cite{Barannikov-projectivespace,Barannikov-period}, $u$ \cite{Pantev-Konstevich}, $z$ \cite{Givental-symplectic}, etc.
\end{rmk}

 Both $\PV(X)((t))$ and $\PV_c(X)((t))$   are  complexes with respect to the following twisted coboundary operator
\begin{align}\label{Qf-fiber}
    Q_{f}:=\dbar_f+t\pa_\Omega.
\end{align}
The complex $\PV(X)((t))$ has a natural decreasing filtration $\mc F^\bullet \PV(X)((t))$ preserved by $Q_f$, which is defined by
\begin{align}\label{eqn-filtrcpx}
  \mc F^k (\PV(X)((t))):=t^k \PV(X)[[t]].
\end{align}

  On the other hand,   the space  $\mc A(X)((t))$ (resp.  $\mc  A_c(X)((t))$ of $\mc A(X)$ (resp. $\mc A_c(X)$) valued   Laurent series  is a complex with
respect to either of the following   twisted coboundary operator: % $d^{\pm}$,
\begin{align}
   d^+_f:=d+{df\over t}\wedge,\qquad d^-_f:=td+(df)\wedge.
\end{align}
\begin{prop-defn}\label{isom-PV-Diff}
  There are two   isomorphisms $\hat\Gamma_{\Omega_X}^{\pm}$ of filtered complexes,
   $$\hat\Gamma_{\Omega_X}^+: (\PV(X)((t)), Q_{f})\to (\mc A(X)((t)), d^+_f),\qquad \hat\Gamma_{\Omega_X}^-: (\PV(X)((t)), Q_{f})\overset{\cong}{\to} (\mc A(X)((t)), d^-_f),$$
 which, for  $t^k\alpha \in t^k \PV^{i,j}(X)$, are respectively defined by
 $$\hat\Gamma_{\Omega_X}^+(t^k\alpha):=t^{k+i-1}\bracket{\alpha\vdash \Omega_X}\in t^{k+i-1} \A^{n-i,j}(X)\quad\mbox{and}\quad \hat\Gamma_{\Omega_X}^-(t^k\alpha):=t^{k+j}\bracket{\alpha\vdash \Omega_X}\in t^{k+j} \A^{n-i,j}(X).$$
Here the filtered complex structure of  $(\mc A(X)((t)), d^{\pm})$  are respectively given by
 \begin{align}
        \hat\Gamma_{\Omega_X}^+(t^k\PV(X)[[t]])&=\prod_{m\in \mathbb{Z}} \bigoplus\limits_{i\geq m+k, j\in\mathbb{Z}}\A^{i,j}(X)\otimes t^{n-m-1},\\
          \hat\Gamma_{\Omega_X}^-(t^k\PV(X)[[t]])&=\prod_{m\in \mathbb{Z}}   \bigoplus_{j\leq m-k, i\in \mathbb{Z}}\A^{i,j}(X)\otimes t^{m},
 \end{align}
% \noindent
%%In the present paper,  we will denote   $\hat\Gamma_{\Omega_X}^+$ as $\hat\Gamma_{\Omega_X}$, and denote $d^+$ as $d_f$.
\end{prop-defn}
\begin{proof}
   Clearly,   $\hat\Gamma_{\Omega_X}^+$ and $\Gamma_{\Omega_X}^-$ are both isomorphisms of $C^\infty(X)((t))$-modules. Furthermore, it is easy to check   $\hat\Gamma_{\Omega_X}^+\circ Q_{f}=d^+_f\circ \hat\Gamma_{\Omega_X}^+$ and  $\hat\Gamma_{\Omega_X}^-\circ Q_{f}=d^-_f\circ \hat\Gamma_{\Omega_X}^-$. Therefore  $\hat\Gamma_{\Omega_X}^{\pm}$ are isomorphisms of complexes.
It is direct to check that  $\hat\Gamma_{\Omega_X}^{\pm}(\mc F^\bullet \PV(X)((t)))$  makes   $(\mc A(X)((t)), d^{\pm}_f)$ a filtered-complex as desired.
  \end{proof}
We remark that the filtration $\mc F^\bullet \PV(X)((t))$ plays the role of the Hodge filtration under the isomorphism $\hat\Gamma_{\Omega_X}^+$.
 The restriction of  $\hat\Gamma_{\Omega_X}^-$ to the    holomorphic data $\Gamma(X, \wedge^*T_X)[[t]]$ gives   the same filtration $\hat\Gamma_{\Omega_X}^-\big(t^k\Gamma(X, \wedge^*T_X)[[t]]\big)= t^k \Omega_X^*[[t]]$ as in
 \cite{Saito-primitive}. The coboundary operator $d^-_f=td+(df)\wedge$ was   introduced by the third author  in \cite{Saito-lecutures, Saito-residue}.
As we will see in Proposition-Definition \ref{prop-def-hodgefil},  $\hat\Gamma_{\Omega_X}^-$  provides a precise identification of the filtrations of cohomology between our approach of using polyvector fields and the original setting up via differential forms. However, we have chosen to work with polyvector fields, $Q_{f}$ and $\hat\Gamma_{\Omega_X}^+$ in order to compare with the compact Calabi-Yau manifolds \cite{Si-BCOV} and illustrate aspects of mirror symmetry.
The variable $t$ is related to the gravitational descendant in physics, mirror to the descendant variable in Gromov-Witten theory.
The complex $(\PV(X)[[t]], Q_{f})$ realizes a natural deformation of $(\PV(X), \dbar_f)$, capturing the central aspects of the non-commutative Hodge structures \cite{Pantev-Konstevich}.

\subsubsection{Cohomology of polyvector fields with descendants}

%%In analogy with the isomorphism $\Gamma_{\Omega_X}: \PV(X)\to \mc A(X)$,

\begin{defn}
We define the following vector spaces $\mc H^{f, \Omega}_{(0)}, \mc H^{f, \Omega}$ associated to $f$
    %%and simply denote $\Gamma(X, \OO_X)$ as $\OO_X$ whenever there is no confusion.     We define
   $$\mc H^{f, \Omega}_{(0)}:= {\Gamma(X, \OO_X)[[t]]\over \Im\big(Q_{f}: \Gamma(X, T_X)[[t]]\rightarrow \Gamma(X, \OO_X)[[t]]\big)},\, \mc H^{f, \Omega}:=\mc H^{f, \Omega}_{(0)}\otimes_{\mathbb{C}[[t]]}\mathbb{C}((t)).$$
 where $\Gamma(X, T_X)$  is the space of  holomorphic vector fields  on $X$.
There is a canonical isomorphism
$$
\mathcal{H}_{(0)}^{f, \Omega}/t\mathcal{H}_{(0)}^{f, \Omega}\iso \Jac(f).
$$
%%and $\mathcal{H}_{(0)}^{f, \Omega}$ describes a deformation of $\Jac(f)$.
\end{defn}

\begin{prop}\label{cohomology-field}
There are canonical isomorphisms
$$
  \mc H^{f, \Omega}_{(0)}\overset{\iso}{\longrightarrow}   H^*(\PV(X)[[t]], Q_{f})\quad \mbox{and}\quad
   \mc H^{f, \Omega}\overset{\iso}{\longrightarrow} H^*(\PV(X)((t)), Q_{f}),
$$
both defined by $[\alpha]\mapsto [\alpha]$, where $\alpha\in \Gamma(X, \OO_X)[[t]]$ represents a class on both sides of the first isomorphism, and similarly for the second isomorphism.
\end{prop}

\begin{proof}
\iffalse
We prove the case for $\PV(X)[[t]]$. We filter the complex by the power of $t$:
$$
 \mc F^k(\PV(X)[[t]]):= t^k \PV(X)[[t]].
$$
The first stage  of the spectral sequence has differential $\dbar_f$. By Lemma \ref{cohomology-polyvectorfield}, the $E_1$-term is
$$
   E_1= \Jac(f)[[t]].
$$
Clearly, $t\pa_{\Omega}$ is zero on $E_1$, and the spectral sequence degenerates. Since the filtration is complete, it follows that the natural map
$$
\mc H^{f}_{(0)}\overset{\iso}{\longrightarrow}   H^*(\PV(X)[[t]], Q_{f})
$$
is an isomorphism. The proof for $\PV(X)((t))$ is similar.
\fi

 This is essentially an application of the spectral sequence associated to the   filtration \eqref{eqn-filtrcpx}. We present the details here. Denote $A^k=\bigoplus_{j-i=k}\PV^{i, j}(X)$.  Since  $Q_f(\Gamma(X, \mathcal{O}_X)[[t]])=0$,   we have a   well-defined morphism $\varphi$ of abelian groups,
          $$\varphi: \Gamma(X, \mathcal{O}_X)[[t]]\longrightarrow H^0(\PV(X)[[t]], Q_f)={\mbox{Ker}(Q_f: A^0[[t]]\rightarrow A^1[[t]])\over\Im(Q_f: A^{-1}[[t]]\rightarrow A^0[[t]])},$$
given by $\alpha\mapsto [\alpha]$.
Let $\sum_{m=0}^\infty a_mt^m\in \mbox{Ker}(Q_f: A^0[[t]]\rightarrow A^1[[t]])$  and   $a_{-1}:=0$. We have
\begin{align}\label{eqn-relation}
   \sum_{m=0}^\infty (\dbar_f a_m+\pa_\Omega a_{m-1})t^m=0.
\end{align}
Since  $[a_0]\in H^0(\PV(X), \dbar_f)$, by Lemma \ref{cohomology-polyvectorfield} there exist $b_0\in \Gamma(X, \mathcal{O}_X)$ and $c_0\in A^{-1}$ such that
  $a_0=b_0+\dbar_f c_0$. Using \eqref{eqn-relation} and Lemma \ref{cohomology-polyvectorfield},   we can  inductively   define  sequences $b_0, b_1, \cdots$ of holomorphic functions on $X$ and sequences $c_0, c_1,  \cdots$
      of elements in $A^{-1}$ such that $$\sum_{m=0}^{\infty}a_mt^m=\sum_{m=0}^\infty b_mt^m+ Q_f(\sum_{m=0}^\infty c_mt^m).$$ Therefore $\varphi$ is surjective.

Similarly, let $\alpha=\sum_{m\geq 0}a_m t^m\in \ker(\varphi)$. Then $\alpha=Q_f(\beta)$ for $\beta=\sum_{m\geq 0}b_m t^m\in A^{-1}[[t]]$. The leading term gives the relation
$
   a_0=\dbar_f b_0
$
which implies that $[a_0]=0$ in $\Jac(f)$ by Lemma \ref{cohomology-polyvectorfield}.  Therefore we can find $c_0\in \Gamma(X, T_X)$ such that $\dbar_f c_0=a_0$, and
$$
  \alpha- Q_f(c_0)\in \ker(\varphi)\cap t \Gamma(X, \OO_X)[[t]]
$$
Inductively, we find $c=\sum_{m\geq 0}c_m t^m\in \Gamma(X, T_X)[[t]]$ such that $\alpha=Q_f(c)$. Thus   $\mbox{Ker}(\varphi) =Q_f(\Gamma(X, T_X)[[t]])$.

   By similar but easier arguments, we conclude $H^k(\PV(X)((t)), Q_f)=0$ if $k\neq 0$.

   The proof for   $\mc H^{f, \Omega}\overset{\iso}{\rightarrow} H^*(\PV(X)((t)), Q_f)$ is completely similar.
\end{proof}

\begin{prop-defn} \label{prop-def-hodgefil}
For any $k\in \mathbb{Z}$, there are canonical isomorphisms of $\C[[t]]$-modules:
$$
    \mathcal{H}_{(-k)}^{f, \Omega}:=t^k \mathcal{H}_{(0)}^{f, \Omega}\overset{\iso}{\to}H^*\bracket{t^k \PV(X)[[t]], Q_{f}}\overset{\cong}{\to}H^*\bracket{t^k \Omega_X^*[[t]], d^-_f}=:\mathcal{H}_f^{(-k)},
$$
where the second isomorphism is induced by $\hat\Gamma_{\Omega_X}^-$. The   composition of the two isomorphisms is given by $[\alpha]\mapsto[\alpha\vdash \Omega_X]$. Here we have used
the notation  $ \mathcal{H}_{(-k)}^{f, \Omega}$ to distinguish it from  $\mathcal{H}_f^{(-k)}$.
The family $\{\mc H_{(-k)}^{f, \Omega}\}_{k\in \mathbb{Z}_{\geq 0}}$ form a filtration of  $\Hzero^{f, \Omega}$, called the  (semi-infinite) Hodge filtration following \cite{Saito-primitive}. The associated graded piece of the filtration $\mc F^k \Hzero^{f, \Omega}=\mc H_{(-k)}^{f, \Omega} $
   is given by
$$
   Gr_{\mc F}^k (\Hzero^{f, \Omega})\iso t^k \Jac(f).
 $$

\end{prop-defn}

\begin{proof} By Lemma \ref{isom-PV-Diff}, we have isomorphic filtered complexes
$$
  \hat\Gamma_{\Omega_X}^-: \bracket{\PV(X)[[t]], Q_f}\to \prod_{m\in \mathbb{Z}}   \bigoplus_{j\leq m, i\in \mathbb{Z}}\A^{i,j}(X)\otimes t^{m}, d^-_f).
$$
It is easy to see that it identifies the filtered sub-complexes
$$
  \hat\Gamma_{\Omega_X}^-: (\Gamma(X, \wedge^*T_X)[[t]], Q_f)\to (\Omega_X^*[[t]], d^-_f)
$$
passing to the cohomology yields the proposition.

 \end{proof}

Using the same $V_f$ and $\rho$ as in the previous subsection, we can define the following operators,  in analogy with  $T_\rho$, $R_\rho$.
\begin{defn}
   Let $Q:=\dbar + t\pa_\Omega$. We define   $T_\rho^t: \PV(X)((t))\to \PV_c(X)((t))$ and $R_\rho^t:\PV(X)((t))\to \PV(X)((t))$ by
    \begin{align}\label{TrhoRrho}
       T_\rho^t:=\rho +[Q, \rho]V_f {1\over 1+[Q, V_f]} \quad \mbox{and}\quad R_\rho^t:=(1-\rho) V_f {1\over 1+[Q, V_f]}.
   \end{align}
\end{defn}
 Clearly,  $R_\rho^t$   preserves $\PV_c(X)[[t]]$, and   $T_\rho^t(\PV(X)[[t]])\subset \PV_c(X)[[t]]$.
  Furthermore, Lemma \ref{lemmaforquasiiso} has the following generalization whose proof is completely similar.

\begin{lem}\label{lemmaforquasiisoWithdescendant}
 {\upshape $[Q_{f},  R_\rho^t]=1-T_\rho^t$  } as operators on $\PV(X)((t))$.
\end{lem}

 As a direct consequence of Proposition \ref{cohomology-field} and Lemma \ref{lemmaforquasiisoWithdescendant}, we have
\begin{cor}\label{cohomology-field-cpt-supp}
The embeddings
$$
  \bracket{\PV_c(X)[[t]], Q_{f}}\into \bracket{\PV(X)[[t]], Q_{f}}, \quad \bracket{\PV_c(X)((t)), Q_{f}}\into \bracket{\PV(X)((t)), Q_{f}}
$$
are both quasi-isomorphisms.
\end{cor}
Similarly, we will denote by $\iota_t$ the induced canonical isomorphisms
$$
\iota_t:  H^*(\PV_c(X)[[t]], Q_{f})\overset{\iso}{\longrightarrow} \mc H_{(0)}^{f, \Omega}, \quad \iota_t:  H^*(\PV_c(X)((t)), Q_{f})\overset{\iso}{\longrightarrow} \mc H^{f, \Omega}
$$
whose inverse can be realized by
$$
  \iota_t^{-1}= T_\rho^t.
$$

 The trace map $\Tr: \PV_c(X)\to \C$ is naturally extended to
 $\PV_c(X)[[t]] \to \C[[t]]$.  It  further descends to cohomologies, still denoted   as $\Tr$,
\begin{align}\label{higher-trace}
   \Tr: H^*(\PV_c(X)[[t]], Q_{f})\to \C[[t]]. %;\quad \Tr: H^*(\S_c(X), Q_{F})\to \C((t)).
\end{align}
The composition
\begin{align}\label{higher-residue-fiberversion}
   \widehat{\mbox{Res}}^f=\Tr\circ \iota^{-1}_t:  \mc H_{(0)}^{f, \Omega}  {\to}  \C[[t]]
\end{align}
provides a complex differential geometric interpretation  of the \emph{higher residue map} introduced by the third author \cite{Saito-lecutures, Saito-residue}.

Moreover, we can use the   trace map $\Tr: \PV_c(X)\to \C$ to define a pairing
  $$\PV_c(X)[[t]]\times \PV_c(X)[[t]]\to \C[[t]]\quad\mbox{given by }\,\, (\alpha_1 \nu_1(t), \alpha_2 \nu_2(t))\mapsto \nu_1(t)\nu_2(-t)\Tr(\alpha_1\alpha_2),$$
where $\alpha_1, \alpha_2\in\PV_c(X)$ and $\nu_1, \nu_2\in \C[[t]]$.
It   follows from Equations \eqref{Trproperties} that the differential  $Q_{f}$ is graded skew-symmetric: for   $\alpha_1\in\PV_c^{i, j}(X)$,
            $$
             \Tr\big(Q_{f}(\alpha_1 \nu_1(t))\wedge \overline{\alpha_2 \nu_2(t)}\big)=-(-1)^{|\alpha_1|}\Tr\big(\alpha_1 \nu_1(t) \wedge \overline{Q_{f}(\alpha_2 \nu_2(t))}\big),
            $$
where       the operator
$
  {}^- : \PV(X)((t))\to \PV(X)((t))$ is defined by
  $$
  \overline{\nu(t)\alpha}:= \nu(-t)\alpha, \quad \text{where}\ \alpha\in \PV(X), \nu(t)\in \C((t))
  $$
 Hence,  the above pairing   gives rise to a paring on the cohomology:
  $$H^*(\PV_c(X)[[t]], Q_{f})\times H^*(\PV_c(X)[[t]], Q_{f})\to \C[[t]].$$ Combining it with the isomorphism $\iota_t$,  we obtain a pairing
  \begin{align}\label{higher-residue-pairing-fiberversion}
     \mc K^f_{\Omega}: \Hzero^{f, \Omega}\times \Hzero^{f, \Omega}\to \C[[t]],
  \end{align}
called the  \emph{higher residue pairing}. We  remark that $\widehat{\mbox{Res}}^f=\mc K^f_{\Omega}(\cdot, 1)$.

The above constructions are naturally extended to $\PV_c(X)((t))$ and $\mc H^{f, \Omega}$ by replacing $\C[[t]]$ with $\C((t))$, and we have
$$
   \mc K^f_{\Omega}: \mc H^{f, \Omega}\times \mc H^{f, \Omega}\to \C((t)).
$$

\begin{rmk}
  The isomorphisms $\hat\Gamma_{\Omega_X}^+$ and $\hat\Gamma_{\Omega_X}^-$ are related via a pairing as follows.
 Consider  $$\PV_c(X)((t))\times \PV_c(X)((t))\to \C((t))$$ defined by
  $(\alpha_1 \nu_1(t), \alpha_2 \nu_2(t))\mapsto t^{1-n}\int_X \hat\Gamma_{\Omega_X}^+(\alpha_1 \nu_1(t))\wedge \overline{\Gamma_{\Omega_X}^-(\alpha_2 \nu_2(t))}$,
  which differs slightly from the pairing defined above by a sign. It is easy to check that this pairing descends to the cohomology,
which coincides with $\mc K^f_\Omega$.
\end{rmk}

\begin{eg} [$A_m$-singularity]\label{example-An} Let $X=\C$, $f={z^{m+1}\over m+1}$, and $\Omega_X=dz$. It is easy to compute that
$
    V_f={1\over f^\prime}\pa_z\wedge
$. Let $h(z)$ be a holomorphic function, and $C$ be a small circle centered at the origin inside which the value of the cut-off function $\rho$ is identically equal  to $1$. Then
{\upshape  \begin{align*}
   \widehat{\mbox{Res}}^f(h)&=\sum_{r\geq 0}(-t)^r \Tr\bracket{ \bracket{\dbar \rho}V_f \bbracket{Q, V_f}^r h} =\sum_{r\geq 0}(-t)^r \Tr\bracket{ \bracket{\dbar \rho}{\pa_z \wedge\over f^\prime} \bracket{{\pa\over \pa z}{1\over f^\prime}}^r h}\\
   &=\sum_{r\geq 0}(-t)^r \int_{\C}\bracket{ \bracket{\dbar \rho}{dz \over f^\prime} \bracket{{\pa\over \pa z}{1\over f^\prime}}^r h}=\sum_{r\geq 0}(-t)^r \oint_{C}\bracket{ {dz \over f^\prime} \bracket{{\pa\over \pa z}{1\over f^\prime}}^r h}\\
  &= \sum_{r\geq 0}(-t)^r\Res_{z=0}\bracket{{1\over f^\prime} \bracket{{\pa\over \pa z}{1\over f^\prime}}^r h dz}\\
&=\sum_{r\geq 0}(-t)^r \prod_{k=0}^{r-1}(m+k(m+1))\Res_{z=0}\bracket{hdz\over z^{r(m+1)+m}}.
\end{align*}
}
In particular, the leading term with $r=0$ gives the ordinary residue $\displaystyle \Res_{z=0}{hdz\over f^\prime}$.
\end{eg}

\subsubsection{Opposite filtrations}
%%\subsubsection{Hodge filtration}

%%%\begin{rmk}
 %%% It will be very interesting to compare this set-up of smooth category with the algebraic set-up of Steenbrink {\color{red}{(add a reference here)}}.
%%%\end{rmk}

Following Givental \cite{Givental-quantization}, we equip
 $\PV_c(X)((t))$ with a  (graded) symplectic pairing
\begin{align*}
  \omega(\alpha_1 \nu_1(t), \alpha_2 \nu_2(t)):=\Res_{t=0}(\nu_1(t)\nu_2(-t)dt)\Tr(\alpha_1  \alpha_2),% \quad \mbox{for }\,\,\alpha,\beta\in \PV_c(X)\mbox{ and } f, g\in \C((t)).
\end{align*}
 where $\alpha_1,\alpha_2\in \PV_c(X)$   and $\nu_1, \nu_2\in \C((t))$.
Clearly,   $t^k\PV_c(X)[[t]]$ is an isotropic  subspace of $\PV_c(X)((t))$ with respect to $\omega$, for any $k\in \mathbb{Z}_{\geq 0}$.
This gives rise to a paring on the cohomology. Combing it with the isomorphism $\iota_t$, we obtain a symplectic paring:
  $$\omega: \mc H^{f, \Omega}\times \mc H^{f, \Omega}\to \C \quad\mbox{given by }\,\, \omega(s, s'):= \Res_{t=0}\mc K^f_{\Omega}(s, s')dt.$$
With respect to this pairing, all  the Hodge filtered  pieces  $\mc H^{f, \Omega}_{(-k)}$ are   isotropic  subspaces of $\mc H^{f, \Omega}$.

\begin{lem-defn}\label{opposite filtration} Let $\mc L$ be a linear  subspace of $\mc H^{f, \Omega}$ that satisfies  both  $(1)\,\,\mc H^{f, \Omega}=\Hzero^{f, \Omega}\oplus \mc L $ and $(2) \,\,  t^{-1}\mc L\subset \mc L.$ Let $B:=\Hzero^{f, \Omega}\cap t\mc L$. Then the following are equivalent:
$$(3)\,\,\,\omega(\mc L, \mc L)=0; \qquad (3)'\,\,\, \mc K^f_\Omega(\mc L, \mc L)\subset t^{-2}\C[t^{-1}]; \qquad (3)''\,\,\, \mc K^f_\Omega(B, B)\subset \C.$$
 If $\mc L$ further satisfies $(3)$, $(3)'$ or $(3)''$,  we call it an
  \emph{opposite filtration} of $\mc H^{f, \Omega}$.
\end{lem-defn}
\begin{proof}
   Since $\mc L$ satisfies (1) and (2), we have $\mc H^{f, \Omega}=t^k\mc H^{f, \Omega}=\mc H_{(-k)}^{f, \Omega}\oplus t^k\mc L, \forall k\in \mathbb{Z}_{\geq 0}$. It leads to the decomposition $\mathcal H_{(-k)}^{f, \Omega}=\mc H^{f, \Omega}_{(-k-1)}\oplus (\mathcal H_{(-k)}^{f, \Omega}\cap t^{k+1}\mc L)$ and induces an isomorphism of vector spaces
     $$
   \mc H_{(-k)}^{f, \Omega}/\mc H_{(-k-1)}^{f, \Omega}\iso  \mc H_{(-k)}^{f, \Omega}\cap t^{k+1}\mc L
$$
for any $k\in \mathbb{Z}_{\geq 0}$. In particular, we have $B=  \Hzero^{f, \Omega}\cap t\mc L\subset \Hzero^{f, \Omega}$ and $B\cong  \mc H_{(0)}^{f, \Omega}/\mc H_{(-1)}^{f, \Omega}\cong \Jac(f)$.
As a consequence, we have
   $$\mc H^{f, \Omega}=B((t)), \quad \Hzero^{f, \Omega}=B[[t]], \quad\mbox{and } \,\mc L=t^{-1}B[t^{-1}].$$
Clearly,  $(3)''\Rightarrow (3)'$, and $(3)'\Rightarrow (3)$. We show that $(3)\Rightarrow (3)''$.

Assuming $(3)$, we only need to show that $\mc K^f_\Omega(B, B)\subset \C[t^{-1}]$ since $\mc K^f_\Omega(B, B) \subset \C[[t]]$. We prove it by contradiction. Let $s, s'\in B$ such that $\mc K^f_{\Omega}(s, s')=\sum_{i=-N}^\infty a_i t^i$ with $a_k\neq 0$ for some $k>0$.
This would imply that $t^{-k}s, t^{-1}s'\in \mc L$ and $\omega(t^{-k}s, t^{-1}s')=-a_k\neq 0$,  contradicting to  the  hypothesis (3).
\end{proof}

\begin{rmk} The notion of opposite filtration, introduced by Deligne \cite{Deligne}, was used by M. Saito for the Hodge structures of the singularity \cite{Scherk-Steenbrink,Steenbrink} to construct primitive forms. However, in the present paper, we will use it for Deﬁnition 2.16 since in the parallel compact Calabi-Yau case, the analogue opposite filtration is
%used to define the correlation functions in BCOV theory \cite{Si-BCOV}. In such case an opposite filtration is
equivalent to a usual splitting of the Hodge filtration on $H^*(X)$ (see for example \cite[Lemma 5.2.2]{Si-BCOV}). Therefore we will keep the terminology of opposite filtration with respect to the semi-infinite Hodge filtration $\{\mc H_{(-k)}^{f, \Omega}\}$, which is sometimes called a polarization in the context of symplectic geometry \cite{Givental-symplectic}.

\end{rmk}

   As we will see in section \ref{sectgeneralconstruction}, the space of primitive forms is essentially identified with the space of appropriate  opposite filtrations.

\section{Primitive forms}

\subsection{Frame setup}\label{subsecTechCond}  We recall some notions from  \cite{Saito-lecutures}, \cite{Saito-Takahashi}.
Recall that   $X\subset \mathbb{C}^n$ is a Stein domain and $f: X\rightarrow \mathbb{C}$ is a holomorphic function  with finite critical points.
\begin{defn}\label{def-fourtuple}
 We call a 4-tuple    $(Z, S, p, F)$ a frame associated to $(X, f)$, if the following are satisfied.
\begin{enumerate}
  \item The space $Z$ is a Stein domain in  $\mathbb{C}^{n+\mu}$,  $S$ is a Stein open neighborhood    of the origin $\mathbf{0}$ of $\C^\mu$,
             $p: Z\to S$ is a Stein map with $p^{-1}(\mathbf{0})\cong X$, and  $F: Z\rightarrow \mathbb{C}$ is a holomorphic function on $Z$.

   %%%     the origin $\mathbf{0}$ of $\C^\mu$

  %%%$\mathbb{C}^{n}\times \mathbb{C}^\mu$, and $\mc W: Z\rightarrow \mathbb{C}$ is a holomorphic function on $Z$.
  %%%   The map $p$ is the restriction of    the  natural projection  $ \mathbb{C}^n\times \mathbb{C}^\mu \rightarrow \mathbb{C}^\mu$ to $Z$, whose image is given by %%%a  Stein open neighborhood  $S=p(Z)$ of
   %%%     the origin $\mathbf{0}$ of $\C^\mu$. The fiber $p^{-1}(\mathbf{0})$ is given by $X\times \{0\}$.

      \item The restriction $p|_{\mc C(F)}$ of $p$ to the relative critical set $\mc C(F)$ (see Equation \eqref{relative-criticalset})   is a proper morphism.

   \item The function  $F$ is a universal unfolding of $f$ in the following sense.
      \begin{enumerate}
        \item The restriction $F|_{p^{-1}(\mathbf{0})}$ coincides with the composition $p^{-1}(\mathbf{0})\overset{\cong}{\rightarrow} X\overset{f}{\rightarrow}\C$.
        \item The     Kodaira-Spencer map
{\upshape\begin{align}\label{KSmap}
 \mbox{KS}: \mc T_{S}\to p_*\OO_{\mc C(F)}, %%= p_*\big(\OO_{Z}\big / (\pa_{z_1}F, \cdots, \pa_{z_n}F)\big),
\end{align}}
to be described below, is an isomorphism of sheaves of $\mc O_{S}$-modules.  Let $\mc T_{S}$  (resp. $\mc T_{Z}$)  denote the holomorphic tangent sheaf of $S$  (resp. ${Z}$). For any open set $U\subset S$ and any   $V\in \Gamma\bracket{U,\mc T_{S}}$,  we take a lifting $\widetilde V\in \Gamma\bracket{p^{-1}(U), \mc T_{Z}}$ with $p_*(\widetilde V)=V$. Then the Kodaira-Spencer map is defined by  {\upshape $\mbox{KS}|_{U}(V):=  \widetilde{V} (F)|_{\mc C(F)}$}, which  is independent of the choice of the lifting $\widetilde V$.
     \end{enumerate}
    \end{enumerate}
\end{defn}
 In practice, it is convenient to fix a
     projection $\pi_X: Z\to X$ (where we make $Z$ smaller if necessary), so that there is an embedding $Z\hookrightarrow X\times S\subset \C^n\times \C^\mu$.
      In particular, the fiber   $p^{-1}(\mathbf{0})$ is    given by $X\times \{0\}$.
       With respect to such embedding,    the Stein map $p$ and the projection $\pi_X$ are the restriction of the natural projections $X\times S\to S$ and $X\times S\to X$ to $Z$, respectively. In terms of coordinates $(\mathbf{z}, \mathbf{u})$ of $\C^n\times\C^\mu$, the relative critical set $\mc C(F)$
       is a subvariety  of $Z$  defined by
    \begin{align}\label{relative-criticalset}
       \mc C(F):=\{\pa_{z_1}F= \cdots=\pa_{z_n}F=0\},\quad\mbox{and}\,\,
    \OO_{\mc C(F)}= \OO_{Z}\big / (\pa_{z_1}F, \cdots, \pa_{z_n}F) .
    \end{align}

We can always achieve a frame   $(Z, S, p, F)$  (see e.g. \cite{Saito-Takahashi} for a precise choice when $f: X\to \C$ has an isolated critical point). Furthermore, we have
\begin{prop}[\cite{Saito-lecutures}]
  The structure sheaf
   $\OO_{\mc C({F})}$ is flat over $\OO_S$, so that    $p_*\OO_{\mc C(F)}$ is a locally free sheaf on $S$ of
  rank $\dim_{\mathbb{C}}\Jac(f)$.
\end{prop}

 \begin{defn}
The \emph{Euler vector field} associated to $(Z,  S, p,  F)$ is defined by %the global section
{\upshape \begin{align}\label{def-Euler-field}
  E:=\mbox{KS}^{-1}([F])\in \Gamma(S, \mc T_{S}).
  \end{align}}
%%and the \emph{identity vector} is defined by $$    \delta_0=KS^{-1}(1) \in T_{\mc M} $$
\end{defn}

\noindent\textbf{Notations: } In this section, we will study properties in the unfolding  case parallel to those in the previous section. Therefore, we will use the same notations
 as in section 2 to denote the relevant operators in the unfolding case, such as $\rho, T_\rho, R_\rho, T^t_\rho, \iota, \iota_t, \Tr$, etc..
\subsection{Relative polyvector fields}\label{subset-relativepoly} Given a frame $(Z, S, p, F)$,  we denote by
 $T_{Z/S}$ the fiberwise holomorphic tangent bundle on $Z$ relative to $S$, and consider the space of smooth
  relative polyvector fields
  $$
    \PV(Z/S)=\!\!\bigoplus_{0\leq i\leq n\atop 0\leq j\leq n+\mu}\!\! \PV^{i,j}(Z/S),\quad\mbox{where}\quad
\PV^{i,j}(Z/S):=\mc A^{0,j}\bracket{Z, \wedge^i T_{Z/S}}.
$$
Similarly, we consider the space of smooth relative differential forms
$$
\mathcal{A}({Z/S})=\!\!\bigoplus_{0\leq i\leq n\atop 0\leq j\leq n+\mu}\!\! \mc A^{i,j}(Z/S), \quad \mc A^{i, j}(Z/S)=\mc A^{0,j}\bracket{Z, \Omega^i_{Z/S} }.
$$
We note that  $T_{Z/S}$ contains $n$-directions tangent to the fiber of $p$, whereas $\mc A^{0,1}$ contains $(n+\mu)$-directions of anti-holomorphic cotangent vectors.
We   fix a family of  holomorphic volume form (which is nowhere vanishing)
\begin{align}\label{relativevolumeform}
   \Omega_{Z/S} \in H^0(Z, \Omega^n_{Z/S}).
 \end{align}
 Contraction with $\Omega_{Z/S} $   induces an isomorphism
   $$
   \Gamma_{\Omega}:  \PV^{i,j}(Z/S)\rightarrow  \mc A^{n-i,j}({Z/S})
   $$
of $C^\infty(Z)$-modules.
Consequently, we   have a
  relative version of the  operators
$$
  \dbar:  \PV^{i,j}(Z/S)\to  \PV^{i,j+1}(Z/S), \quad   \pa_{\Omega}:  \PV^{i,j}(Z/S)\to  \PV^{i-1,j}(Z/S),
$$
where $\pa_{\Omega}$ is defined by
       $$ \pa_{\Omega} (\alpha):=\Gamma_{\Omega}^{-1}\big(\pa(\Gamma_{\Omega}(\alpha))\big).$$
\begin{eg}
In     coordinates   $(\mathbf{z}, \mathbf{u})$   of $\mathbb{C}^n\times \mathbb{C}^\mu$,
  every     $\alpha \in  \PV^{i,j}(Z/S)$ is of the form
     $$
     \alpha=\sum_{I, J_1, J_2}\alpha^I_{J_1, J_2}(\mathbf{z}, \mathbf{u})d\bar z_{J_1}d\bar u_{J_2}\otimes {\pa \over \pa z_I}
     $$
  where the summation is over  subsequences $I$, $J_1$, $J_2$  of $[1, \cdots, n], [1, \cdots, n]$ and $[1, \cdots, \mu]$, respectively, with  $|I|=i$ and $|J_1|+|J_2|=j$.
The volume form can be expressed by
$$
\Omega_{Z/S}={1\over \lambda(\mathbf{z}, \mathbf{u})}{dz_1\wedge \cdots \wedge dz_n }.
$$
Then we have
          $$\pa_\Omega  \alpha = \sum_{k=1}^n\sum_{I, J_1, J_2}\lambda\pa_{z_k} \big({{\alpha^I_{J_1, J_2}\over \lambda}}\big){\pa\over\pa\pa_k} \big(d\bar z_{J_1}d\bar u_{J_2}\otimes {\pa \over \pa z_I}\big) \quad\mbox{and}\quad \dbar  \alpha=  \sum_{I, J_1, J_2}(\dbar   {\alpha^I_{J_1, J_2}}) d\bar z_{J_1}d\bar u_{J_2}\otimes {\pa \over \pa z_I}.$$
\end{eg}

The operator $\pa_\Omega$ induces a
  bracket $\{\mbox{-},\mbox{-}\}$ on $\PV(Z/S)$,  defined in the same way as formula \eqref{bracked}. We define the twisted coboundary operator  in the relative setting by
$$
   \dbar_{F}=\dbar+\{F, \mbox{-}\}: \PV(Z/S)\rightarrow \PV(Z/S).
$$

We also consider the subspace
  $$\PV_c(Z/S):=\{\alpha\in \PV(Z/S)~\big|~ p|_{\scriptsize\mbox{supp}(\alpha)} \mbox{ is proper}\}.$$
 Clearly, all the operators $\dbar, \pa_\Omega, \{\mbox{-}, \mbox{-}\},  \dbar_{F}$ preserve the subspace $\PV_c(Z/S)$.
 Since $p: Z\to S$ is a Stein map, we can obtain  sheaf-theoretic version of similar propositions in the previous section as follows.

Let $(\mathbf{z}, \mathbf{u})$ be the   coordinates of $\mathbb{C}^n\times \mathbb{C}^\mu\supset Z$. Consider the following smooth operator $V_{F}$ of degree $-1$ on $\PV((Z\setminus \mc C(F))/S)$:
    $$V_{F} = {1\over \sum\limits_{j=1}^n|\pa_{z_j}F|^2}\sum_{i=1}^n ({\overline{\pa_{z_i} F}}) \pa_{i}\wedge.$$
Choose a smooth function $\rho \in C^\infty({Z})$     such that    $\rho|_{U_1}\equiv 1$ and $\rho|_{Z\setminus U_2}\equiv 0$ for
       some  open neighborhoods  $U_1, U_2$  of $\mc C(F)$ in $Z$ with the properties  (a) $U_1\subset \overline{U_1}\subset U_2\subset \overline{U_2}\subset Z$, and (b) $p|_{\overline{U_2}}$ is proper. %%; (c) the boundaries $\pa U_1$, $\pa U_2$ are both smooth in $Z$.
   Then   the   operators  %$T_\rho, R_\rho$  (where we use the same notations as in formula \eqref{TrhoRrho} by abuse of notation), defined by
     $$T_\rho:=
    \rho+ (\dbar\rho)V_{F} {1\over 1+[\dbar, V_{F}]}  \quad\mbox{and}\quad R_\rho:=(1-\rho)  V_{F}{1\over 1+[\dbar, V_{F}]}$$ are both defined on the whole space
   $\PV(Z/S)$. Furthermore, we have
   $$
   T_\rho(\PV(Z/S))\subset \PV_c(Z/S),
   $$
such that
   \begin{align}\label{dbarRrho}
      \bbracket{\dbar_{F},R_\rho}=1-T_\rho.
  \end{align}

Let  $R^\bullet p_*\bracket{\PV_c(Z/S), \dbar_{F}}$ be  the sheaf of $\OO_S$-modules on $S$ obtained by the sheafification of     $\{H^*\bracket{\PV_c(p^{-1}(U)/U), \dbar_{F}}~|~ U\subset S \mbox{ is open Stein}\}$. Similarly, we obtain  $R^\bullet p_*\bracket{\PV(Z/S), \dbar_{F}}$.

\begin{prop}\label{prop-PVZoverS}
There is a natural  isomorphism of sheaves
$$
  p_*\OO_{\mc C(F)}\overset{\iso}{\longrightarrow}
       R^\bullet p_*\bracket{\PV(Z/S), \dbar_{F}}.
$$
\end{prop}

\begin{proof}
  The statement follows immediately from the sheafification of Lemma \ref{cohomology-polyvectorfield}.
\end{proof}
\begin{prop} The natural embedding  $\bracket{\PV_c(Z/S), \dbar_{F}}\into \bracket{\PV(Z/S), \dbar_{F}}$ is quasi-isomorphic, inducing a canonical isomorphism of sheaves
 $$
    \iota:  R^\bullet p_*\bracket{\PV_c(Z/S), \dbar_{F}}\overset{\iso}{\longrightarrow}p_*\OO_{\mc C(F)}
 $$
whose inverse morphism  can be represented by $T_\rho$.
\end{prop}
\begin{proof}
Formula \eqref{dbarRrho} implies that $H^*\big(\PV(p^{-1}(U)/U)/\PV_c(p^{-1}(U)/U), \dbar_{F}\big)=0$ for any open (Stein) subset $U$ of $S$. The first statement follows, and the second statement becomes a direct consequence of   Proposition  \ref{prop-PVZoverS}.
\end{proof}

In analogy with the descriptions in section \ref{polyvector-descendant}, we consider the complex $\PV(Z/S)((t))$ (resp. $\PV_c(Z/S)((t))$) of
$\PV(Z/S)$ (resp. $\PV_c(Z/S)$)-valued Laurent series in the descendant variable $t$, with respect to the following  twisted coboundary operator
$$
   Q_{F}:=\dbar_{F}+t\pa_\Omega.
$$
  %On the $\OO_S$-modules of  descendant relative polyvector fields, we

\begin{prop-defn}\label{isom-complex-sheaf} There are isomorphisms of sheaves of cochain complexes $$\hat \Gamma_{\Omega}^{\pm}: \bracket{\PV(Z/S)((t)), Q_{F}}\rightarrow \bracket{\mc A(Z/S)((t)), d_{F}^{\pm}}\quad\mbox{with}\quad d_{F}^+:=d+{dF\over t} \wedge,\,\, d_{F}^-=td+({dF})\wedge,
$$
defined by
\begin{align*}
   t^k\alpha&\in t^k\PV^{i, j}(Z/S)\mapsto  \hat\Gamma_\Omega^+(t^k\alpha):= t^{k+i-1}\alpha\vdash \Omega_{Z/S} \in t^{k+i-1}\mc A^{n-i, j}(Z/S), \\
      t^k\alpha&\in t^k\PV^{i, j}(Z/S)\mapsto \hat\Gamma_\Omega^-(t^k\alpha):= t^{k+j}\alpha\vdash \Omega_{Z/S} \in t^{k+j}\mc A^{n-i, j}(Z/S).
\end{align*}
\begin{proof} The proof is the same as Proposition-Definition \ref{isom-PV-Diff}.
\end{proof}

\end{prop-defn}

The following operators are defined on $\PV(Z/S)((t))$:
   $$T_\rho^t:=\rho +[Q, \rho]V_{F} {1\over 1+[Q, V_{F}]}\quad \mbox{and}\quad
 R_\rho^t:=(\mbox{id}-\rho) V_{F} {1\over 1+[Q, V_{F}]}, $$
 where   $Q:=\dbar +t\pa_\Omega$.
The relation
$$
[Q_{F},  R_\rho^t]=1-T_\rho^t
$$
still holds where defined. In addition, we have
$$
T_\rho^t(\PV(Z/S)[[t]])\subset \PV_c(Z/S)[[t]], \quad T_\rho^t(\PV(Z/S)((t)))\subset \PV_c(Z/S)((t)).
$$

Consider the sheaves $\Hzero^{F, \Omega}$ and  $\mc H^{F, \Omega}:=\Hzero^{F, \Omega}\otimes_{\C[[t]]}\C((t))$, where
    $$\Hzero^{F, \Omega}(U):= {\Gamma(p^{-1}(U), \mathcal{O}_{Z})[[t]]\over \Im\big(Q_{F}: \Gamma(p^{-1}(U), \mc T_{Z/S})[[t]]\rightarrow \Gamma(p^{-1}(U), \mathcal{O}_{Z})[[t]]\big)}, \mbox{ for Stein } U\subset S.$$

\begin{prop}\label{cohomology-field-relative}
There are canonical isomorphisms of sheaves of $\mc O_{S}$-modules
$$
 \iota_t:  R^\bullet p_*\big({\PV_c(Z/S)[[t]], Q_{F}}\big)\overset{\iso}{\rightarrow}   R^\bullet p_*(\PV(Z/S)[[t]], Q_{F})\iso\Hzero^{F, \Omega}.
$$
The inverse of the composition $\iota_t$ of the above isomorphism  can be represented by $T_\rho^t$.
\end{prop}
\begin{proof}
  The proof is completely similar  to the proofs of Proposition \ref{cohomology-field} and Corollary \ref{cohomology-field-cpt-supp}.
\end{proof}
\begin{prop-defn}\label{prop-def-hodgefil-family}
For any $k\in \mathbb{Z}$, we define $\mathcal{H}^{F, \Omega}_{(-k)}:=t^k \mathcal{H}_{(0)}^{F, \Omega}$. The family version of the filtered complex isomorphism  $\hat\Gamma_{\Omega_X}^-$ gives a  canonical isomorphism of sheaves of \(\OO_S[[t]]\)-modules:
\[
   \hat\Gamma_\Omega^-: \mathcal{H}^{F, \Omega}_{(-k)} \overset{\cong}{\to} R^\bullet p_*(t^k\Omega_{Z/S}^*[[t]], d_F^-)=:\mathcal{H}_F^{(-k)}
\]
defined by $t^k[\alpha]\mapsto [t^k \alpha\vdash \Omega_{Z/S}]$. The (semi-infinite) Hodge filtration  of $\mathcal{H}^{F, \Omega}_{(0)}$ is defined by  $\{\mathcal{H}^{F, \Omega}_{(-k)}\}_{k\in \mathbb{Z}_{\geq0}}$.
\end{prop-defn}
\begin{proof}
   The argument  is completely similar to the proof of Proposition-Definition \ref{prop-def-hodgefil}.
\end{proof}

\subsection{Variation of semi-infinite   Hodge structures}\label{sec-VSHS}
The general structures of higher residues and primitive forms developed in \cite{Saito-primitive, Saito-residue} were geometrically reformulated by Barannikov as the notion of \emph{variation of semi-infinite  Hodge structures} \cite{Barannikov-thesis, Barannikov-period}. We will adopt this geometric notion and present the corresponding constructions in our case.

\subsubsection{Gauss-Manin connection} The sheaf $\mc H^{F, \Omega}$ inherits a flat Gauss-Manin connection  \cite{Brieskorn-GM, Saito-primitive}
$$
    \nabla^{\Omega}: \mc H^{F, \Omega}\to \Omega_{S}^1\otimes \mc H^{F, \Omega}
$$
from its identification with differential forms via ${\Omega_{Z/S}}$. To describe this, we consider
the isomorphism of cochain complexes in Proposition-Definition \ref{isom-complex-sheaf}
$$\hat \Gamma_{\Omega}^{+}: \bracket{\PV(Z/S)((t)), Q_{F}}\rightarrow \bracket{\mc A(Z/S)((t)), d_{F}^{+}}\quad\mbox{with}\quad d_{F}^+:=d+{dF\over t} \wedge,
$$
which induces an isomorphism of sheaves of $\OO_S$-modules, still denoted as
$$
   \hat \Gamma_\Omega^+:  \mc H^{F, \Omega}\iso R^\bullet p_*\bracket{\PV(Z/S)((t)), Q_{F}}\overset{\cong}{\to} R^\bullet p_*\bracket{\mc A(Z/S)((t)), d_{F}^+}.
$$

\begin{lem-defn}
The sheaf $R^\bullet p_*\bracket{\mc A(Z/S)((t)), d_{F}^+}$ naturally carries the Gauss-Manin connection
{\upshape $$
   \nabla^{\scriptsize\mbox{GM}}: R^\bullet p_*\bracket{\mc A(Z/S)((t)), d_{F}^+}\to \Omega_S^1\otimes R^\bullet p_*\bracket{\mc A(Z/S)((t)), d_{F}^+}.
$$
}
\end{lem-defn}
\begin{proof} This is a twisted version of the usual differential geometric construction.

Let $V\in \mc T_S$ be a holomorphic vector field on $S$, and we choose an arbitrary smooth lifting $\tilde V$ to a vector field of type $(1,0)$ on $Z$  such that
$$
  p_*(\tilde V)=V.
$$
Let $[\alpha]\in R^\bullet p_*\bracket{\mc A(Z/S)((t)), d_{F}^+}$ be a section represented by $\alpha\in \mc A(Z/S)((t))$. Then the Gauss-Manin connection is defined by
\begin{align}\label{eqn-GMconn}
   \nabla^{\scriptsize\mbox{GM}}_V[\alpha]=\bbracket{ \mc L_{\tilde V}\alpha+{\pa_{\tilde V}(F)\over t}\alpha }
\end{align}
where $\mc L_{\tilde V}$ is the Lie derivative with respect to $\tilde V$. It is routine to check that $\nabla^{\scriptsize\mbox{GM}}$ is well-defined and gives a flat connection.
\end{proof}

\begin{defn} We define the Gauss-Manin connection $\nabla^{\Omega}$ on $\mc H^{F, \Omega}$ associated with $\Omega_{Z/S}$ by
{\upshape $$
   \nabla^{\Omega}:= \big(\hat \Gamma_{\Omega}^+\big)^{-1}\circ \bracket{\nabla^{\scriptsize\mbox{GM}}}\circ\hat\Gamma_\Omega^+.
$$}
\end{defn}

The next properties follow directly from the definition of the  connection  $\nabla^{\Omega}$:
 \begin{lem}[Transversality]\label{Transversality}
 $
   \nabla^{\Omega}|_{\Hzero^{F, \Omega}}: \Hzero^{F, \Omega}\to \Omega^1_{S}\otimes t^{-1}\Hzero^{F, \Omega}=\Omega^1_{S}\otimes  \mc H^{F, \Omega}_{(1)}$.
\end{lem}

\begin{lem}[Flatness] $\nabla^{\Omega}\circ \nabla^{\Omega}=0$.
\end{lem}

Now we describe the Gauss-Manin connection $\nabla^{\Omega}$ in    coordinates  $(\mathbf{z}, \mathbf{u})$   of $\mathbb{C}^n\times \mathbb{C}^\mu$. Since $p: Z\rightarrow  S$ is the restriction of the natural projection $\mathbb{C}^n\times \mathbb{C}^\mu\to \mathbb{C}^\mu$, a special lifting of a local section $V=\sum\limits_{j=1}^\mu g_j(\mathbf{u}){\pa\over \pa u_j}$ of $\mc T_{S}$ is given by
 \begin{equation}\label{eqn-special-lifting}\hat V:=\sum_{j=1}^\mu (p^*g_j){\pa\over \pa u_j},
 \end{equation}
where ${\pa\over \pa u_j}$ is   defined using the projection $\pi_X: Z\to X$.

\begin{prop}\label{corwelldefine} If $\Omega_{Z/S}={1\over \lambda(\mathbf{z})}dz_1\wedge\cdots\wedge dz_n\in \pi_X^*(\Gamma(X,\Omega_X^n))$, where $\lambda$ is a nonwhere vanishing holomorphic function on $Z$ depending only on the fiber direction, then
   \begin{align}\label{eqn_simplifiedGMconn}
  \nabla^\Omega_V \bbracket{s}=\bbracket{\pa_{\hat V}s+{\pa_{\hat V}F\over t}s},\end{align}
 where $[s]\in \mc H^{F, \Omega}$ is represented by $s\in\Gamma\bracket{Z, \OO_Z}((t))$.
\end{prop}

As we will see in section \ref{subsec-intrinsicity}, the choice of $\Omega_{Z/S}$ is not essential. Hence, in this paper we will always choose the volume form $\Omega_{Z/S}$ to be independent of the deformation parameters. In such   case, the calculation of the connection $\nabla^{\Omega}$ is simplified as in the above proposition.

\subsubsection{Extended Gauss-Manin connection}
We would like to think about $t$ as a coordinate on the formal punctured disk $\hat \Delta^*=\Spec \C((t))$, and $\mc H^{F, \Omega}$ is naturally a locally free sheaf on $S\times \hat \Delta^*$. Then we can extend the flat connection $\nabla^{\Omega}$ to $S\times \hat \Delta^*$ as follows. %on $S$ to $S\times \hat \Delta^*$. Namely,
 For $s \in \PV^{i,*}(Z/S)((t))$, we  define
\begin{align}\label{extendGM}
       \nabla^{\Omega}_{t\pa_t}s=t\nabla^{\Omega}_{\pa_t}s:=\bracket{t{\pa\over \pa t}+i-{F\over t}}s.
\end{align}

Clearly, $\nabla^{\Omega}_{t\pa_t}$ preserves $\PV_c(Z/S)((t))$. The next lemma follows from direct calculations
\begin{lem}\label{extend-GMconn-flat}      $[\nabla^{\Omega}_{t\pa_t}, Q_{F}]=0$ holds as operators  on $\PV(Z/S)((t))$, hence
  $\nabla^{\Omega}_{t\pa_t}$ descends to an operator on $\mc H^{F, \Omega}$.   Furthermore for any  $V\in \Gamma(S, \mc T_{S})$, we have
 $[\nabla^{\Omega}_{t\pa_t}, \nabla^{\Omega}_V]=0$.
 \end{lem}

\noindent It follows that $\nabla^{\Omega}$ is extended to a well-defined flat connection  on $S\times \hat \Delta^*$ for the sheaf $\mc H^{F, \Omega}$ viewed as $\OO_{S\times \hat \Delta^*}$-module, which we will call the \textit{extended Gauss-Manin connection}.  The originally defined connection $\nabla^\Omega$ on $S$ will be referred to as the \textit{non-extended} connection.

 %, which is called the \textit{extended   Gauss-Manin connection}.

%%%It is obvious that  all the aforementioned  operators acting on $\PV(Z/S)((t))$ (resp. $\mc A(Z/S)((t))$)  preserve $\PV_c(Z/S)((t))$ (resp. $\mc A_c(Z/S)((t))$). Hence, the same construction of $\nabla^\Omega$ and all the propositions in this subsection remain valid for $R^\bullet p_*(\PV_c(Z/S)((t)), Q_{F})$.

%%%%%%%%%%%%%%%%%%%%%%central fiber%%%%%%%%%%%%%%%
\iffalse
In addition, we can restrict  the operator $\nabla^{\Omega}_{t\pa_t}$ to $\PV(X)((t))$ at the central fiber:
    $$\nabla^{\Omega}_{t\pa_t}s=t\nabla^{\Omega}_{\pa_t}s:=\bracket{t{\pa\over \pa t}+i-{W\over t}}s$$ for  $s\in \PV^{i, *}(X)((t))$.
Such restriction    also gives rise to a map on   $\mc H^{f}$:
% for any $s\in \PV^{i, *}(X)((t))$, the definition  $\nabla^{\Omega}_{t\pa_t}s=t\nabla^{\Omega}_{\pa_t}s:=\big(t{\pa\over \pa t}+i-{F\over t}\big)s$ makes sense, and $[\nabla^{\Omega}_{t\pa_t}, Q_{W}]=0$ holds on  $\PV(X)((t))$. Hence,
  \begin{equation}\label{def-extendGM-onfiber}
    \nabla^{\Omega}_{t\pa_t}: \mc H^{f}\rightarrow \mc H^{f};\quad [s]\mapsto \nabla^{\Omega}_{t\pa_t} [s]:=[\nabla^{\Omega}_{t\pa_t} s].%:=[\big(t{\pa\over \pa t}+i-{ W\over t}\big)s].
 \end{equation}
\fi
%%%%%%%%%%%%%%%%%%%%%%%%%%%%%%%%%%%%%%%%%%%

 \subsubsection{Higher residues} We define a cochain map $\Tr$ of complexes of sheaves on $S$,
 $$
    \Tr:   (\PV_c(p^{-1}(U)/U)((t)), Q_{F}) \to (\mc A^{0,*}(U)((t)), \dbar_S),
 $$
by
  $$
  \Tr(\nu(t)\alpha):=\nu(t)\int_{p^{-1}(U)/U}(\alpha\vdash  \Omega_{Z/S} ) \wedge \Omega_{Z/S}, \quad \nu(t)\in \C((t)), \alpha \in \PV_c(p^{-1}(U)/U)
  $$
  where $\dbar_S$ denotes the usual $\dbar$ operator on differential forms on $S$, $\int_{p^{-1}(U)/U}$ denotes the  fiberwise integration, and $U\subset S$ is a Stein open   subset. It gives rise  to a map of   sheaves of cohomologies
  \begin{align}
     \Tr:  R^\bullet p_*(\PV_c(Z/S)((t)))\rightarrow \mc O_{S}((t)),
  \end{align}
 as the family version of the trace map in formula \eqref{higher-trace}.

Similarly, we define a pairing
 $$
   \mc K^{F, c}_{\Omega}\bracket{\mbox{-},\mbox{-}}
    : \PV_c(p^{-1}(U)/U)((t)) \otimes \PV_c(p^{-1}(U)/U)((t)) \to \mc A^{0, *}(U)((t)),
$$
by
$$
\mc K_{\Omega}^{F, c}\bracket{\nu_1(t)\alpha_1, \nu_2(t)\alpha_2}:=\nu_1(t)\nu_2(-t)\Tr\bracket{{\alpha_1\alpha_2}}
$$
where $\nu_1, \nu_2\in \C((t))$ and $\alpha_1, \alpha_2 \in \PV_c(p^{-1}(U)/U)$.

\begin{lem}
 The pairing $\mc K_{\Omega}^{F, c}\bracket{\mbox{-},\mbox{-}}$ on $\PV_c(p^{-1}(U)/U)$ gives rise to that on cohomology:
 $$
 \mc K_{\Omega}^{F, c}\bracket{\mbox{-},\mbox{-}}: R^\bullet p_*(\PV_c(Z/S)((t))) \otimes R^\bullet p_*(\PV_c(Z/S)((t))) \to \OO_{S}((t)).$$
\end{lem}
\begin{proof} The relative version of formula \eqref{Trproperties} says that $\mc K^{F, c}_{\Omega}$ is a cochain map. Let $\alpha, \beta\in \PV_c(Z/S)((t))$ be $Q_F$-closed, representing sections of $R^\bullet p_*(\PV_c(Z/S)((t)))$.  We can assume that both $\alpha, \beta$ have degree $0$, hence $\mc K_{\Omega}^{F, c}(\alpha, \beta)\in \cinfty(S)((t))$. Moreover,
$$
  \dbar_S \mc K_{\Omega}^{F, c}(\alpha, \beta)=\mc K_{\Omega}^{F, c}(Q_F\alpha, \beta)+\mc K_{\Omega}^{F, c}(\alpha, Q_F\beta)=0,
$$
i.e., $\mc K_{\Omega}^{F, c}(\alpha, \beta)\in \OO_S((t))$.
\end{proof}

This allows us to give the complex differential geometric construction of the higher residue pairing \cite{Saito-residue}.

\begin{defn}\label{propdefn-higherResPair}
We define the \textbf{higher residue pairing}
$$
\mc K^{F}_{\Omega}\bracket{\mbox{-},\mbox{-}}%= \bracket{\mbox{-},\mbox{-}}_{F, \E_c}^{\Omega}\circ (T_\rho^t\otimes T_\rho^t)
 : \mc H^{F, \Omega} \otimes \mc H^{F, \Omega} \to \OO_{S}((t))
 $$
     by $ \mc K^{F}_{\Omega}\bracket{[s], [s']}:=\mc K^{F, c}_{\Omega}\bracket{\iota_t^{-1}([s]), \iota_t^{-1}([s'])}
      %=\Tr\Big(\big(T_\rho^t(s)\big)(t) \big(T_\rho^t(s')\big)(-t)\Big)
      $ for sections $[s], [s']$ of $\mc H^{F, \Omega}$ (see Proposition \ref{cohomology-field-relative} for the isomorphism  $\iota_t$).
\end{defn}

Every local section  $\iota_t^{-1}[s]$ of $R^\bullet p_*(\PV_c(Z/S)((t)))$  has a representative $T_\rho^t(s)$. Hence we have the following explicit formula
$$
 \mc K^{F}_{\Omega} \bracket{[s], [s']}=\Tr \bracket{T_\rho^t(s) \overline{T^{t}_\rho(s')}}
$$
where the ${}^-$ operator is defined by
$$
  {}^- : \PV(Z/S)((t))\to \PV(Z/S)((t)), \quad \overline{\nu(t)\alpha}:= \nu(-t)\alpha\quad \mbox{ for } \alpha\in \PV(Z/S).
$$

\begin{rmk}
Since both $T_\rho^t(s)$ and $T_\rho^t(s')$ have compactly support along the fiber direction,
$$
 \mc K^{F}_{\Omega}\bracket{[s], [s']}=\Tr \bracket{T_\rho^t(s) \overline{s'}}=\Tr \bracket{s\ \overline{T^{t}_\rho(s')}}.
$$
\end{rmk}
\begin{defn}  The \textit{higher residue map} associated to $(Z, S, p, F)$ and $\Omega$ is defined by the composition
  {\upshape $$\widehat{\mbox{Res}}^{F}:=\Tr\circ \iota_t^{-1}:  \mc H^{F, \Omega} \rightarrow %% R^\bullet p_*(\Ezero_c(Z/S), Q_{F})\overset{\Tr}{\longrightarrow}
     \OO_{S}((t)).$$
}
\end{defn}
As before, if we restrict $ \mc K^{F}_{\Omega}$ to $\mc H_{(0)}^{F, \Omega}$, it takes values in $\OO_S[[t]]$. That is,
 $$
   \mc K^{F}_{\Omega}\bracket{\mbox{-},\mbox{-}}: \Hzero^{F, \Omega} \otimes \Hzero^{F, \Omega} \to \OO_{S}[[t]].
$$
 The triple $\{\Hzero^{F, \Omega}, \mc K^{F}_{\Omega}\bracket{\mbox{-},\mbox{-}}, \nabla^{\Omega} \}$  satisfy the following properties.
\begin{prop}\label{lem-higherpairing-properties} Let   $s_1, s_2$ be local sections of $\Hzero^{F, \Omega}$.
\begin{enumerate}
\item $\mc K^{F}_{\Omega}\bracket{s_1, s_2} =%(-1)^{|\mu_1||\mu_2|}
\overline{\mc K^{F}_{\Omega}\bracket{s_2, s_1}}$.
\item $\mc K^{F}_{\Omega}({\nu(t)s_1, s_2}) =\mc K^{F}_{\Omega}({s_1,\nu(-t)s_2})=\nu(t)\mc K^{F}_{\Omega}({s_1, s_2})$\,\, for any $\nu(t)\in \OO_{S}[[t]]$.
\item  $\pa_V \mc K^{F}_{\Omega}({s_1, s_2})=\mc K^{F}_{\Omega}({\nabla_V^{\Omega}s_1, s_2})+\mc K^{F}_{\Omega}({s_1,\nabla_V^{\Omega}s_2})$ \,\,  for any local section $V$ of $\mc T_{S}$.
\item $ \bracket{t\pa_t+n}\mc K^{F}_{\Omega}\bracket{s_1, s_2}=\mc K^{F}_{\Omega}(\nabla^{\Omega}_{t\pa_t}s_1, s_2)+\mc K^{F}_{\Omega}(s_1, \nabla^{\Omega}_{t\pa_t}s_2).$
\item The higher residue pairing induces a   pairing on
$$     \Hzero^{F, \Omega}/t \Hzero^{F, \Omega}\otimes_{\OO_{S}} \Hzero^{F, \Omega}/t \Hzero^{F, \Omega}\to \OO_{S},$$
which coincides with the classical   residue pairing
$p_*\OO_{\mc C(F)} \otimes_{\OO_{S}} p_*\OO_{\mc C(F)}\to \OO_{S}.$ In particular, the induced pairing is non-degenerate.
\end{enumerate}
\end{prop}
\begin{proof}
The first four statements follow  from direct calculations and the definition of the pairing $\mc K^{F}_{\Omega}(\mbox{-}, \mbox{-})$.

 The fiber of $\Hzero^{F, \Omega}/t \Hzero^{F, \Omega}$ at   $u\in S$ is given by $H^*(\PV_c(Z_{u}), \dbar_{F|_{Z_u}})$ $=\Jac(F|_{Z_u})$, where $Z_u:=p^{-1}(u)$. By Proposition \ref{compatible-residue}, the pairing at $u$ coincides with the   residue pairing $
    \Jac(F|_{Z_u})\otimes \Jac(F|_{Z_u})\to \C,$
which is known to be non-degenerate  \cite{Hartshorne-residue}. Hence (5) follows.
\end{proof}

%%\begin{rmk}
  %%As shown in {\color{red}{(add a reference)}}, the pairing $(\mbox{-}, \mbox{-})$ satisfying all the properties in the above lemma is unique up to a constant.
%%\end{rmk}

   The triple $\{\Hzero^{F, \Omega},\mc K^{F}_{\Omega}\bracket{\mbox{-}, \mbox{-}}, \nabla^{\Omega} \}$ satisfying properties of Proposition \ref{lem-higherpairing-properties} and the Transversality (Lemma \ref{Transversality}) %satisfying the above properties
   is called a \emph{variation of  semi-infinite Hodge structure}, a notion due to Barannikov \cite{Barannikov-period}.

%%%{\color{red}{\begin{rmk}Need to double check if the IMAGE of all the high residue map/pairing are well-defined.\end{rmk}}}

\subsection{Primitive forms}
With the triple $\{\Hzero^{F, \Omega}, \mc K^{F}_{\Omega}\bracket{\mbox{-}, \mbox{-}}, \nabla^{\Omega} \}$, we can define the notion of primitive forms \cite{Saito-primitive}. In terms of variation of semi-infinite   Hodge structures, the primitive forms correspond geometrically to the semi-infinite period maps \cite{Barannikov-period}.

\begin{defn}\label{define-primitive-form} A   section $\zeta\in \Gamma(S, \Hzero^{F, \Omega})$ is called a \emph{primitive form} if it satisfies the following conditions:
\begin{enumerate}
\item[(P1)]{\upshape (Primitivity)}\quad The section  $\zeta$ induces an $\OO_{S}$-module isomorphism
$$
    t\nabla^{{\Omega}}\zeta:\mc T_{S}\to \Hzero^{F, \Omega}/t \Hzero^{F, \Omega}=p_*\OO_{\mc C(F)}; \quad V\mapsto t\nabla^{\Omega}_V \zeta.
$$
\item[(P2)]{\upshape (Orthogonality)}\quad
For any local sections $V_1, V_2$ of   $\mc T_S$,
\begin{align}
\label{primitive-eq21}   \mc K_\Omega^{F}\big({\nabla^{\Omega}_{V_1} \zeta, \nabla^{\Omega}_{V_2} \zeta}\big)&\in t^{-2}\OO_{S}.
\end{align}
\item[(P3)]{\upshape (Holonomicity)}\quad
For any local sections $V_1, V_2, V_3$ of   $\mc T_S$,
\begin{align}
\label{primitive-eq22}   \mc K_\Omega^{F}\big({\nabla^{\Omega}_{V_1} \nabla^{\Omega}_{V_2} \zeta, \nabla^{\Omega}_{V_3} \zeta}\big)&\in t^{-3}\OO_{S}\oplus t^{-2}\OO_{S};\\
\label{primitive-eq23}  \mc K_\Omega^{F} \big({\nabla^{\Omega}_{t\pa_t} \nabla^{\Omega}_{V_1} \zeta, \nabla^{\Omega}_{V_2} \zeta}\big)&\in t^{-3}\OO_{S}\oplus t^{-2}\OO_{S}.
\end{align}
\item[(P4)] {\upshape (Homogeneity)}\quad Recall   the Euler vector field $E$ is defined in \eqref{def-Euler-field}. There is a constant $r\in \C$ such that
\begin{align*}%\label{primitive-homogenety}
    \bracket{\nabla^{\Omega}_{t\pa_t}+\nabla^{\Omega}_{E}}\zeta=r\zeta.
\end{align*}
\end{enumerate}
\end{defn}

\begin{rmk}
   The hypothesis \eqref{primitive-eq23} follows from the combination  of {\upshape (P2)}, \eqref{primitive-eq22} and {\upshape (P4)}.
\end{rmk}
%%The geometric meaning of the properties of primitive forms will be explained in details in the next section when we discuss the constructions.
 %%In modern language, a primitive form associates to the universal deformation space $S$ the structure of Frobenius manifold.
The properties of a primitive form  give rise to a set of flat coordinates and an associated Frobenius manifold structure on the universal unfolding $S$, which was originally called by the third author the \emph{flat structure}  \cite{Saito-primitive}. This is exposed in detail via the modern point of view  in \cite{Saito-Takahashi}. An exposition for Calabi-Yau manifolds parallel with our current setting is given in \cite{Si-Frobenius}   . As we will see in the next section, primitive forms can be constructed by using good opposite filtrations $\mc L$ of $\mc H^{f, \Omega}$. This leads to a concrete way to compute the flat coordinates and the potential function of the associated Frobenius manifold structure     \cite{LLS}.

\subsection{Intrinsic   properties}\label{subsec-intrinsicity} So far we have fixed a relative holomorphic volume form $\Omega:=\Omega_{Z/S}$ (see formula \eqref{relativevolumeform}) in order to define a primitive form. In this subsection, we will discuss  properties of higher residue pairings and primitive forms associated with different choices  $ {1\over \lambda}\Omega$,   where $\lambda\in H^0(Z, \OO_Z^*)$ is a nowhere vanishing holomorphic function on $Z$. We will use the subscript ${1\over \lambda}\Omega$ to indicate such a choice. For instance, we will denote  $Q_{F, {{1\over \lambda}\Omega}}$ for the differential and $\mc H^{F, {1\over \lambda}\Omega}$  to specify the dependence on the choice. We will denote by
 $[x]_{{1\over \lambda}\Omega}$ a local section of $\mc H^{F, {1\over \lambda}\Omega}$, represented by
 a $\PV(Z/S)$-valued Laurent series $x$ in $t$. In this subsection, we will also identify our primitive forms with the primitive forms in the original approach via differential forms.

Recall that  we have   isomorphisms of cochain complexes
   $$\hat \Gamma_{{1\over \lambda}\Omega}^{\pm}: \bracket{\PV(Z/S)((t)), Q_{F, {1\over \lambda}\Omega}}\overset{\cong}{\longrightarrow}\bracket{\mc A(Z/S)((t)), d_{ F}^{\pm}}.$$
%%Here we use $\hat \Gamma_{{1\over \lambda}\Omega}$ to relate a complex whose coboundary operator depends on ${1\over \lambda}\Omega$ with another complex whose  coboundary operator $d_F$  does not depend on ${1\over \lambda}\Omega$.
Here we will use $\hat \Gamma_{{1\over \lambda}\Omega}^+$, while we remark that the other isomorphism $\hat \Gamma_{{1\over \lambda}\Omega}^-$ also works.

Let us still  denote
by $\hat \Gamma^+_{{1\over \lambda}\Omega}$   the composition
          $$\hat \Gamma^+_{{1\over \lambda}\Omega}: \mc H^{F,  {1\over \lambda}\Omega}\overset{\cong}{\longrightarrow} R^\bullet p_*\bracket{\PV(Z/S)((t)), Q_{F, {1\over \lambda}\Omega}}\overset{\cong}{\longrightarrow}R^\bullet p_*\bracket{\mc A(Z/S)((t)), d_{ F }^+}$$ of isomorphisms of sheaves of $\OO_S$-modules.

\begin{lem}
  For any $x\in \PV(Z/S)((t))$ and $\lambda\in H^0(Z, \OO_Z^*)$, we have
   $$(1)\,\hat\Gamma_{{1\over \lambda}\Omega}^+(\lambda x)=\hat\Gamma_\Omega^+(x);\quad(2)\,Q_{F, {1\over \lambda}\Omega}(\lambda x)=\lambda Q_{F, \Omega}(x) %;\quad(3)\, T_{\rho, {1\over \lambda}\Omega}^t (\lambda x)=\lambda T_{\rho, {1\over \lambda}\Omega}^t (x)
   . $$
\end{lem}
\begin{proof}
   Statement (1) is obvious. (2) follows from the observation $\pa_{{1\over \lambda}\Omega}(\lambda x)=\lambda\pa_\Omega(x)$.
   %.  Note $\pa_{{1\over \lambda}\Omega}(\lambda x)=\lambda\pa_\Omega(x), V_{F}(\lambda x)=\lambda V_{F}(x)$, $\dbar (\lambda x)=\lambda \dbar x$ and $\{F, \lambda x\}=\lambda\{F, x\}$. Thus   statement (2), (3) follow.
\end{proof}
As a direct consequence, we have
\begin{cor} For every  $\lambda\in H^0(Z, \OO_Z^*)$,
 the map $\psi_\lambda: \mc H^{F,   \Omega} \to \mc H^{F,  {1\over \lambda}\Omega}$ by $[x]_{\Omega}\mapsto [\lambda x]_{{1\over \lambda}\Omega}$ is well defined, and defines an isomorphism of sheaves of $\OO_S$-modules. Furthermore, we have $\hat \Gamma^+_{{1\over \lambda}\Omega}\circ \psi_\lambda =\hat\Gamma_{\Omega}^+$.
\end{cor}

We refer to the next proposition as the \textit{intrinsic  properties} among the higher residue pairings and the associated Gauss-Manin connection.
\begin{prop}\label{prop-funcpro-HigherRes} Let   $\lambda\in H^0(Z, \OO_Z^*)$. For any local sections $[x]_{\Omega}, [y]_{\Omega}\in  \mc H^{F,   \Omega}$,
       $$  \mc K_{{1\over \lambda}\Omega}^{F} \bracket{[\lambda x]_{{1\over \lambda}\Omega}, [\lambda y]_{{1\over \lambda}\Omega}} =\mc K_\Omega^{F}\bracket{[x]_{\Omega}, [y]_{\Omega}}.$$
Furthermore, we have
   $$\nabla^{{1\over \lambda}\Omega}_V[\lambda x]_{{1\over \lambda}\Omega}=[\lambda \nabla^{\Omega}_Vx]_{{1\over \lambda}\Omega}\quad \mbox{and}\quad \nabla^{{1\over \lambda}\Omega}_{t\pa_t}[\lambda x]_{{1\over \lambda}\Omega}=[\lambda \nabla^{\Omega}_{t\pa_t}x]_{{1\over \lambda}\Omega},$$
   where $V$ is a  local section   of the holomorphic tangent sheaf $\mc T_{S}$.
\end{prop}

\begin{proof} Recall that we have defined the operator $V_{F}$ and $T_\rho^t$, and $\iota_t^{-1}$ can be realized by $T_\rho^t$.
    Note
      $$
      Q_{{1\over \lambda}\Omega}(\lambda x)=\lambda Q_{\Omega}(x), \quad
      V_{F}(\lambda x)=\lambda V_{F}(x),  \quad \{F, \lambda x\}=\lambda\{F, x\}.
      $$
      It follows that  $T_{\rho, {1\over \lambda}\Omega}^t (\lambda x)=\lambda T_{\rho, \Omega}^t (x)$. Hence,
      \begin{align*}
           \mc K_{{1\over \lambda}\Omega}^{F}\bracket{[\lambda x]_{{1\over \lambda}\Omega}, [\lambda y]_{{1\over \lambda}\Omega}}
           &=
       \int_{Z/S}  \bracket{\bracket{T_{\rho, {1\over \lambda}\Omega}^t(\lambda x)\wedge \overline{ T_{\rho, {1\over \lambda}\Omega}^{t} (\lambda y)}}\vdash {1\over \lambda}\Omega}\wedge {1\over \lambda}\Omega\\
       &=\int_{Z/S} \bracket{\bracket{T_{\rho, \Omega}^t(x)\wedge  \overline{T_{\rho,\Omega}^{t} (y)}}\vdash \Omega}\wedge \Omega\\
       &=\mc K_\Omega^{F}\bracket{[x]_{\Omega}, [y]_{\Omega}}.
      \end{align*}

The compatible relations on the Gauss-Manin connections $\nabla^{{1\over \lambda}\Omega}$ follow from that fact that they are all induced from that on differential forms via the isomorphism $\hat \Gamma^+_{{1\over \lambda}\Omega}$.
       \end{proof}
As a direct consequence of the above proposition, we have
 \begin{prop}\label{primitive-change}
    A   section $\zeta=[x]_{ \Omega}\in\Gamma(S,  \Hzero^{F, \Omega})$ is a primitive form with respect to the choice $\Omega$ if and only if the  section $[\lambda  x]_{{1\over \lambda}\Omega}\in\Hzero^{F, {1\over \lambda}\Omega}$    is a primitive form  with respect to the choice ${1\over \lambda}\Omega$.
 \end{prop}

In order to identify the primitive forms in our setting with the original ones, we let   $t=\delta_w^{-1}$ following the notation convention in  \cite{Saito-Takahashi}, and recall some facts therein.
The natural embedding of $(\Omega_{Z/S}^*((t)), d_F^-)\hookrightarrow   (\mathcal{A}(Z/S)((t)), d_F^-)$
is a quasi-isomorphism  (by Lemma \ref{cohomology-polyvectorfield} together with $\hat\Gamma_{\Omega_X}^-$).
This gives a canonical isomorphism
   \[\mathcal{H}_F:=  R^\bullet p_*(\Omega_{Z/S}^*((t)), d_F^-)\overset{\cong}{\to}  R^\bullet p_* (\mathcal{A}(Z/S)((t)), d_F^-).\]
The de Rham cohomology group $\mathcal{H}_F$  is equipped a filtration $\{\mathcal{H}_F^{(-k)}\}_{k\in \mathbb{Z}}$ given in
 Definition \ref{prop-def-hodgefil-family}. There is   also  the Gauss-Manin connection $\nabla^{\scriptsize\mbox{GM}}: \mathcal{H}_F\to\mathcal{H}_F$ defined the same formula as in \eqref{eqn-GMconn}, which is extended along $\delta_W$ by
  \[\nabla^{\scriptsize\mbox{GM}}_{\delta_w\partial_{\delta_w}}[\alpha]:= [(\delta_w\partial_{\delta_w}+\delta_w F)\alpha].\]

\begin{thm}\label{rmk-identifyprimitiveforms}
{\upshape (1) }The extended Gauss-Manin connections $\nabla^\Omega$ and {\upshape $\nabla^{\scriptsize\mbox{GM}}$} are compatible in the  sense:
  {\upshape \[\hat\Gamma_\Omega^- \circ\nabla^\Omega_{t\partial_t}=-\nabla^{\scriptsize\mbox{GM}}_{\delta_w\partial_{\delta_w}}\circ \hat\Gamma_\Omega^-, \quad \hat\Gamma_\Omega^- \circ\nabla^\Omega_V=\nabla^{\scriptsize\mbox{GM}}_V\circ \hat\Gamma_\Omega^-,\mbox{\itshape { for any local holomorphic vector filed }} V.\]}
\noindent{\upshape (2) } The next pairing  $\mathcal{K}_F$ coincides with the third author's   higher resider pairing on $\mathcal{H}_F$ in \cite{Saito-residue}.
\begin{align*}
    \mathcal{K}_F:& \mathcal{H}_F\otimes\mathcal{H}_F\to \OO_S((t))\\
    \mathcal{K}_F([\alpha_1], [\alpha_2]):&= t^n\mathcal{K}^F_\Omega\big((\hat\Gamma_\Omega^-)^{-1}([\alpha_1]),(\hat\Gamma_\Omega^-)^{-1}([\alpha_2])\big)
\end{align*}

\noindent{\upshape (3) }
 A section $[x]\in\Gamma(S,  \Hzero^{F, \Omega})$ is a primitive form with respect to $\{\Hzero^{F, \Omega}, \mathcal{K}^F_\Omega(\mbox{-},\mbox{-}), \nabla^\Omega\}$
  if and only if  the section  $[x \Omega]\in \Gamma(S, \mathcal{H}_F^{(0)})$ is a primitive form with respect to {\upshape $\{\mathcal{H}^{(0)}_{F}, \mathcal{K}_F(\mbox{-},\mbox{-}), \nabla^{\scriptsize\mbox{GM}}\}$}.
\end{thm}

\begin{proof}
  (1) The relation $t=\delta_w^{-1}$ implies   $t\partial_t=-\delta_w\partial_{\delta_w}$.  The first identify follows from direct calculations.
  The second identity also follows from the definitions, by observing that each section $[x]$ of $\mathcal{H}^{F, \Omega}$ is represented  by an element $x$ of
    $\OO_Z((t))$.

  (2) It follows directly from Proposition \ref{lem-higherpairing-properties} that the pairing $\mathcal{K}_F(\mbox{-}, \mbox{-})$ satisfies all the properties of Theorem 5.1 of \cite{Saito-Takahashi}. Therefore it coincides with the third author's  higher residue pairing on $\mathcal{H}_F$ by the uniqueness as stated in the same theorem.

  (3) The statement also  follows by  directly checking according to  the definitions.
\end{proof}

Primitive forms are invariant under the stably equivalence. Precisely, for any $m\in \mathbb{Z}_{>0}$, we define a holomorphic function $f_m$ on $X_m:=X\times \C^m$, by $f_m:=f+\sum_{i=1}^m z_{n+i}^2$. Let $\pi_m$ denote the natural projection from  $Z_m:= Z\times \C^m$ to $Z$.
  Then we obtain a frame  $(Z_m, S, p_m, F_m) $ associated to $(X_m, f_m)$, where      $p_m:=p\circ \pi_m$ is a Stein map from  $Z_m$ to $S$, and
    $F_m:=F+\sum_{i=1}^m z_{n+i}^2$ is the universal unfolding of $f_m$.
We let $\Omega_{Z_m/S}:= \Omega_{Z/S}\wedge d{z_{n+1}}\wedge \cdots\wedge d{z_{n+m}}$. Because of the coincidence of primitive forms in our setting with the original sense by using differential forms in Theorem \ref{rmk-identifyprimitiveforms}, we have
\begin{prop}[see section 6 of \cite{Saito-lecutures}]\label{prop-stableequiv}
  A section $\zeta\in \Gamma(S, \Hzero^{F, \Omega})$ is a primitive form with respect to the choice $\Omega_{Z/S}$ if and only if
  $\pi_m^*(\zeta)$ is a primitive form in   $\Gamma(S, \Hzero^{F_m})$ with respect to the choice $\Omega_{Z_m/S}$.

Furthermore if $m$ is even, then there is an isomorphism of $\OO_S$-modules,
 $$\rho_m: \mc H^{F, \Omega}_{(\ell)}\overset{\cong}{\longrightarrow}\mc H_{(\ell)}^{F_m},\,\, \forall \ell\in \mathbb{Z}_{\leq 0},$$
which is equivariant with respect to the extended Gauss-Manin connection.
\end{prop}

\begin{rmk}
   In the original approach to primitive forms by differential forms, there is a degree shifting on the left hand side of the isomorphism $\rho_m$.
   More precisely, the   degree on the left should be shifted to $\ell-{m\over 2}$, and the isomorphism is given by {\upshape $\big(\nabla^{\scriptsize\mbox{GM}}_{\pa_t}\big)^{m\over 2} dz_{n+1}\wedge\cdots\wedge dz_{n+m}$} with the     Gauss-Manin connection {\upshape $\nabla^{\scriptsize\mbox{GM}}_{\pa_t}$} for differential forms.

\end{rmk}

\section{Perturbative approach of primitive forms}\label{sectgeneralconstruction}
In this section, we will give an explicit description of  primitive forms in the germ of the universal unfolding of $(X, f)$. The main idea is based on the third author's good sections \cite{Saito-unfolding} and  Barannikov's formula \cite{Barannikov-period} of semi-infinite period maps. See also \cite{Si-Frobenius} on its connection with Kodaira-Spencer gauge theory. The  existence of a primitive form as a germ was already proved in  \cite{Mo.Saito-existence} and \cite{Douai-Sabbah-I} in the corresponding situations.

We will first construct primitive forms in the formal neighborhood of the reference point in the universal unfolding. The use of formal setting not only simplifies the presentation of the primitive forms, but also gives an explicit formula (Theorem \ref{thm-primitive-local}) which is closely related to the oscillatory integral. Via this new approach, we find an explicit algorithm to compute the Taylor series expansion of the primitive forms up to arbitrary finite order. This is applied in the next section to recover known examples of primitive forms and new results for exceptional unimodular singularities. The analytic nature of the formal primitive forms follows by comparing the formal construction with the original analytic construction \cite{Saito-primitive}.

 Let $(Z, S, p, F)$ be a frame (see Definition \ref{def-fourtuple}) associated to  $(X, f)$. In addition, a  projection $\pi_X: Z\to X$ has been chosen, so that we obtain an embedding  $\pi_X\times p: Z {\hookrightarrow} X\times S$ (where $Z$ is made smaller if necessary). We will choose the relative holomorphic volume form
 $$
 \Omega_{Z/S}=\pi_X^*\Omega_X
  $$ %%such that $L_V\Omega_{Z/S}=0$ for any local section $V$ of $\mc T_{S}$. Here we have treated $L_V$ as the Lie derivative along the canonical lifting of $V$ as before. That is , in local coordinate $(\mathbf{z}, \mathbf{u})$, $\Omega_{Z/S}$ is represented by
 where $\Omega_X$ is a fixed holomorphic volume form on $X$. By Proposition  \ref{corwelldefine}, the (non-extended) Gauss-Manin connection  $\nabla^{\Omega}$ can be simplified as equation \eqref{eqn_simplifiedGMconn}, namely
 $$ \nabla^\Omega_V \bbracket{s}=\bbracket{\pa_{  V}s+{\pa_{  V}F\over t}s},$$
   where we have abused the notation $V$ with its lifting $\hat V$   (see  \eqref{eqn-special-lifting}) on the right hand side.

We will let
   $$F_0: =\pi_X^*f$$
 denote the trivial unfolding.
    %%We also simply denote the higher residue pairing  $(\mbox{-}, \mbox{-})_{F}:=(\mbox{-}, \mbox{-})_{F}^\Omega$, and denote by $\nabla^{\Omega}:=\nabla^\Omega$    the associated Gauss-Manin connection.

\subsection{Formal geometry}
\subsubsection{Formal set-up}
 We fix some notions to be used in our formal set-up first. Let $\mathfrak{m}$ denote the maximal ideal of the local ring $\OO_{S, 0}$ of the reference point $0\in S$.
Let $R_N:=\OO_{S, 0}/\mathfrak{m}^N$ where $N$ is a positive integer.
We consider the inverse system
   $$\mc R:=\{R_N~|~ N\in \mathbb{Z}_{>0}\},$$
 where, for $N\geq N'$, the morphism   $\varphi_{N, N'}: R_{N} \to  R_{N'}$  is given by the natural projection.
 Clearly, the associated  system   $\{\mbox{Spec}(R_N)\}_N$ of $\mc R$ form a basis of  the formal neighborhood of the reference point $0\in S$.
 Here on after we simply denote the subscript $N\in \mathbb{Z}_{>0}$ as $N$ without confusion.
\begin{defn} By a $\mc R$-module $\mc F$, we mean a set $\{\mc F_N\}_N$, each $\mc F_N$ being a $R_N$-module, together with morphisms
$\Phi_{N+1,N}:\mc F_{N+1}\to \mc F_{N}$  of $R_{N+1}$-modules  (by viewing $\mc F_{N}$ as a $R_{N+1}$-module via $\varphi_{N+1, N}$) that induce isomorphisms of $R_{N}$-modules: $\mc  F_{N+1}\otimes_{R_{N+1}}R_{N} \overset{\cong}{\to} \mc F_{N}.$
 Direct product for $\mc R$-modules is defined in the similar fashion.

Let  $\Gamma(\mc R, \mc F)$ denote the space of   sections   of $\mc F$ on $\mc R$, where by a section $s$
  we mean a set $\{s_N\}_N$ with each $s_N$ an element in $\mc F_N$, satisfying the compatibility condition:
{\upshape $$
s_{N}= \Phi_{N+1, N}(s_{N+1}), \quad \forall\, N.
$$   }
\end{defn}

The space $\Gamma(\mc R, \mc F)$ is naturally an $\hat {\OO}_{S,0}$-module, where  $\hat {\OO}_{S,0}$ denotes  the formal completion of $\OO_{S,0}$ at the reference point $0$ with respect to the maximal ideal $\mathfrak{m}$. In particular, it is also an $\OO_{S,0}$-module.

\begin{defn} Let $\mc E$ be an  $\OO_S$-module  on $S$. We define the induced   $\mc R$--module    $\check{\mc E}:=\{\check{\mc E}_N\}_N$ by setting
    $$\check {\mc E}_N:= \mc E\otimes_{\OO_{S, 0}} R_N,
$$
via the natural projection $\varphi_{R_N}: \OO_{S, 0}\to R_N$.
\end{defn}

\begin{lem-defn}\label{lemdef-connection}
Let $\mc E$ be an $\OO_S$-module with a connection $\nabla: \mc E\to \Omega_S^1\otimes \mc E$, then there is an induced map
$$
     \check \nabla: \Gamma(\mc R, \check {\mc E})\to \Omega_S^1\otimes_{{\OO}_{S}} \Gamma(\mc R, \check {\mc E})
$$
satisfying the Leibniz rule:
$$
 \check \nabla(a s)= da\otimes s+ a\otimes \check \nabla s, \quad \mbox{for any } a\in \OO_S \,\,\mbox{ and }\,\, s\in \Gamma(\mc R, \check {\mc E}).
$$
%%%Furthermore, if $\E$ is flat with respect to $\nabla$, then we have
%%   {\upshape $$\mbox{Ker}(\check \nabla)=\Gamma(\mc R, \widecheck {\mbox{Ker}\nabla} ).$$}
Whenever there is no confusion, we will simply denote   $\check \nabla$ as $\nabla$ by abuse of notations.
\end{lem-defn}
\begin{proof} Let $s=\{s_N\}_N$ be a section in  $\Gamma(\mc R, \check{\mc E})$,  and $V$ be a local section of $\mc T_S$.
   Let $\tilde s_{N+1}\in \mc E$ be a lifting of $s_{N+1}$, and we define
$$
   (\check\nabla_Vs)_N:=\varphi_{R_N}(\nabla_V \tilde s_{N+1}).
$$
It is easy to check that $\check\nabla_V s$ does not depend on the choice of the lifting, and defines an element in $\Gamma(\mc R, \check{\mc E})$. The Leibniz rule follows easily.
\end{proof}

\subsubsection{Formal primitive forms}Recall that given the frame  $(Z, S, p, F)$, we have the free $\OO_S[[t]]$-module $\Hzero^{F, \Omega}$ and the free $\OO_S((t))$-module $\mc H^{F,\Omega}$   with the  extended   Gauss-Manin connection $\nabla^{\Omega}$. As $\OO_S$-modules, they induce the corresponding $\mc R$-modules $\check{\mc H}_{(0)}^{F,\Omega}$ and $\check{\mc H}^{F, \Omega}$ respectively.
 The space   $\Gamma(\mc R, \check {\mc H}_{(0)}^{F,\Omega})$ (resp. $\Gamma(\mc R, \check{\mc H}^{F, \Omega})$) inherits the  structure of $\C[[t]]$-module (resp. $\C((t))$-module). By Lemma \ref{lemdef-connection}, $\Gamma(\mc R, \check{\mc H}^{F,\Omega})$ is equipped with the induced non-extended Gauss-Manin connection, which is easily seen to be extended along the $t\pa_t$ direction in the similar way. Such extended flat connection on $\Gamma(\mc R, \check{\mc H}^{F, \Omega})$ is again denoted as $\nabla^\Omega$ without confusion.

\begin{rmk}\label{rmk-limit}
Given a section $\{s_N\}_N$ of $\check \E$, we can take the inverse limit $\hat s:=\varprojlim_{N} s_N$ in  $\varprojlim_{N} \E/\mathfrak{m}^N\E$.
 When $\mc E$ is given by the $\OO_S$-module $\Hzero^{F, \Omega}$, the inverse limit $\hat s$ is still a formal power series in the descendent variable $t$.
  However, when $\mc E=\mc H^{F,\Omega}$, the   power of $t$ in the  limit $\hat s$  may go to $-\infty$.
 %% Indeed, the formal higher residue pairing extents to  $\hat{\mc K}_{F}^{\Omega}: \Gamma(\mc R, \check{\mc H}^{F, \Omega})\otimes _{\hat \OO_{S,0}} \Gamma(\mc R, \check{\mc H}^{F, \Omega})\to \hat \OO_{S,0}[[t, t^{-1}]].$
 \end{rmk}

The higher residue pairing defines a natural pairing for each $N$, by tensoring with $R_N$,
  \begin{align*}%\label{def-checkKF}
    \hat{\mc K}^{F, N}_{\Omega}:  \big({\mc H}^{F, \Omega}_{(0)}\otimes_{\OO_S} R_N\big) \otimes_{R_N}\big({\mc H}^{F, \Omega}_{(0)}\otimes_{\OO_S} R_N\big)\to R_N[[t]]
  \end{align*}

\begin{lem-defn} There is an induced formal higher residue pairing
$$
  \hat{\mc K}^{F}_{\Omega}: \Gamma(\mc R, \check{\mc H}^{F, \Omega}_{(0)})\otimes _{\hat \OO_{S,0}} \Gamma(\mc R, \check{\mc H}^{F, \Omega}_{(0)})\to \hat \OO_{S,0}[[t]]
$$
satisfying similar properties as in Proposition \ref{lem-higherpairing-properties}.
\end{lem-defn}
\begin{proof} Let  $s=\{s_N\}_N, s'=\{s'_N\}_N\in \Gamma(\mc R, \check{\mc H}^{F, \Omega}_{(0)})$.
\iffalse
  Let
$\tilde s_N$ and $\tilde s'_N$ be liftings in $\Gamma(S, \mc H^{F}_{(0)})$   with  $\tilde s_N=\varphi_{R_N}(\tilde s)$ and   $\tilde s_N'=\varphi_{R_N}(\tilde s_{N}')$.  Then
$$\hat{\mc K}_{F, N}^{\Omega}(s_N, s_N'):=\varphi_{R_N} \bracket{\mc K^{F}^{\Omega}(\tilde s_N, \tilde s_N')}\in R_{N}[[t]]
$$
does not depend on   choices of the liftings.
\fi
The pairing $\hat{\mc K}^{F}_{\Omega}(s, s')$ is defined by
$$
\hat{\mc K}^{F}_{\Omega}(s, s'):=\varprojlim_{N} \hat{\mc K}^{F, N}_{\Omega}(s_N, s'_N).
$$
\end{proof}

Now we  formulate primitive forms in the formal neighborhood as well.

\begin{defn}\label{formal-primitive-form} A   section $\zeta\in \Gamma(\mc R, \check {\mc H}_{(0)}^{F,\Omega})$ is called a \emph{formal primitive form} if it satisfies the following conditions:
\begin{enumerate}
\item[(P1)${}^{\vee{}}$]The section  $\zeta$ induces an $ {\hat \OO_{S,0}}$-module isomorphism
$$
    t\nabla^{{\Omega}}\zeta:\hat \OO_{S,0}\otimes_{\OO_{S}}\mc T_{S}\to \Gamma(\mc R, \check {\mc H}_{(0)}^{F,\Omega})/t \Gamma(\mc R, \check {\mc H}_{(0)}^{F,\Omega}); \quad V\mapsto t\nabla^{\Omega}_V \zeta.
$$
\item[(P2)${}^{\vee{}}$]
For any local holomorphic vector fields  $V_1, V_2$ on $S$,
\begin{align*}
   \hat {\mc K}_\Omega^{F}\big({\nabla^{\Omega}_{V_1} \zeta, \nabla^{\Omega}_{V_2} \zeta}\big)&\in t^{-2}\hat \OO_{S,0}.
\end{align*}
\item[(P3)${}^{\vee{}}$]
For any local holomorphic vector fields  $V_1, V_2, V_3$ on $S$,
\begin{align*}
 \hat{\mc K}_\Omega^{F}\big({\nabla^{\Omega}_{V_1} \nabla^{\Omega}_{V_2} \zeta, \nabla^{\Omega}_{V_3} \zeta}\big)&\in t^{-3}\hat\OO_{S,0}\oplus t^{-2}\hat \OO_{S,0};\\
 \hat{\mc K}_\Omega^{F} \big({\nabla^{\Omega}_{t\pa_t} \nabla^{\Omega}_{V_1} \zeta, \nabla^{\Omega}_{V_2} \zeta}\big)&\in t^{-3}\hat\OO_{S,0}\oplus t^{-2}\hat \OO_{S,0}.
\end{align*}
\item[(P4)${}^{\vee{}}$] Recall     the Euler vector field is given in \eqref{def-Euler-field}. There is a constant $r\in \C$ such that
\begin{align*}%\label{primitive-homogenety}
    \bracket{\nabla^{\Omega}_{t\pa_t}+\nabla^{\Omega}_{E}}\zeta=r\zeta.
\end{align*}
\end{enumerate}
\end{defn}

There is a natural map by restricting to the formal neighborhood
$$
  \fbracket{\mbox{primitive forms in the germ of $0$ around $S$}}\to \fbracket{\mbox{formal primitive forms around}\ 0}
$$
We will see later that this map is in fact a bijection. This reduces the construction of primitive forms to the formal neighborhood.

\subsubsection{Formal splitting}
\begin{lem-defn}\label{lem-flatextension}
The map
 \begin{align} \label{expHW}
    e^{(f-F)/t}: \mc H^{f,\Omega}\to \Gamma(\mc R, \check{\mc H}^{F, \Omega}),
 \end{align}
given by
 \begin{align} \label{expHW-def} [s]\mapsto  e^{(f-F)/t}[s]:= \{[\check s_N]\}_N\quad\mbox{with}\quad \check s_N:=\varphi_{R_N}(e^{(F_0-F)/t}\pi_X^*(s)),
 \end{align}
is well-defined, independent of choices of the representatives $s$. Furthermore, $e^{(f-F)/t}[s]$ is flat with respect to
 the non-extended Gauss-Manin connection $\nabla^\Omega$, which will be called the formal flat extension of $[s]$.

The above map  $e^{(f-F)/t}$ induces an isomorphism of $R_N$-modules
 \begin{align} \label{expHWR}
   e^{(f-F)/t}|_N: \mc H^{f,\Omega}\otimes_{\C} R_N  { \to }  {\mc H}^{F, \Omega}\otimes_{\OO_S}R_N, \quad\mbox{defined by}\,\,\, [s]\otimes r\mapsto [r\check s_N],
\end{align}
whose inverse is
\begin{align} \label{expHWRinverse}
   e^{(F-f)/t}|_N: \mc H^{F,\Omega}\otimes_{\OO_S} R_N  { \to }  {\mc H}^{f, \Omega}\otimes_{\C}R_N, \quad\mbox{defined by}\,\,\, [b]\otimes 1\mapsto [\varphi_{R_N}(e^{(F-F_0)/t}b)].
\end{align}
We will skip the subscript ``$|_{N}$" whenever there is no confusion.
\end{lem-defn}

\begin{proof} %We first show that $s$ is  well-defined.
Let
  $s=\alpha+Q_{f}(\beta)$ for some  $\alpha\in \Gamma(X, \OO_{X})((t))$ and $\beta\in \PV(X)((t))$. The map in \eqref{expHW} is in fact the composition of the following maps:
     $$\mc H^{f,\Omega}\to \mc H^{f,\Omega}\otimes_\C \OO_S\overset{\cong}{\to} \mc H^{F_0,\Omega}\overset{e^{(F_0-F)/t}}{\to} \Gamma(\mc R, \check{\mc H}^{F, \Omega}),$$
  where the first map is given by $[s]\mapsto [s]\otimes 1$, the second map is an isomorphism of $\OO_S$-modules defined by $[s]\otimes g\mapsto [g\pi_X^*(s)]$, and the third map is defined by using \eqref{expHW-def}.
  By direct calculations, we have
 $$
 \varphi_{R_N}\big({e^{(F_0-F)/t}\pi_X^*\bracket{Q_{f}(\beta)}}\big) =\varphi_{R_N}\big(e^{(F_0-F)/t} Q_{F_0}(\pi_X^*(\beta)) \big)=\varphi_{R_N}\big({Q_{F}\big({e^{(F_0-F)/t}\pi_X^*\beta}}\big)\big).
 $$
Since $F_0-F\in \OO_X\otimes \mathfrak{m}$,   the series expansion of  $\varphi_{R_N}(e^{(F_0-F)/t})= \sum_{k=0}^N{1\over k!}{(F_0-F)^k\over t^k}\in R_N[t, t^{-1}]$,  whose  pole in $t$ is of finite order.
Therefore $[\check s_N]$   is a well-defined element in $ {\mc H}^{F, \Omega}\otimes_{\OO_S}R_N$, independent of choices of representatives $s$.
It is easy to see that $[\check s]:= \{[\check s_N]\}_N$ satisfies the compatibility condition, hence lies in   $\Gamma(\mc R, \check{\mc H}^{F, \Omega})$.
Moreover,
$$
  \nabla^\Omega_V [\check s_N]=\nabla^{\Omega}_V[\varphi_{R_N}\big(\big({e^{(F_0-F)/t}\pi_X^*(s)}\big)\big)]=[\varphi_{R_N}\big(\big({\pa_V+{\pa_VF\over t}}\big)\big({e^{(F_0-F)/t}}\pi_X^*(s)\big)\big)]=0.
$$
Hence, the extension $[\check s]$ of $[s]$ is flat.

 %%Note that  $\varphi_{R_N}(e^{(F_0-F)/t})\in R[t, t^{-1}]$ is invertible. Clearly,
The map in \eqref{expHWR} is  given by the composition:  $\mc H^{f,\Omega}\otimes_\C R_N\overset{\cong}{\to}\mc H^{F_0,\Omega}\otimes_{\OO_S}R_N %\overset{e^{(F_0-F)/t}}
{\to} \mc H^{F,\Omega}\otimes_{\OO_S}R_N$, where the first map is the canonical isomorphism of $R_N$-modules. The second map is defined
by sending   $[a]\otimes r$  to $[r\varphi_{R_N}(e^{(F_0-F)/t} a)]$, which is also an isomorphism of $R_N$-modules since its inverse is constructed by multiplying $e^{(F-F_0)/t}$ (whose details are left to the readers).
\end{proof}

\begin{rmk}
    Geometrically, the map $e^{(f-F)/t}$  can be considered as    the local trivialization of $\check{\mc H}^{F, \Omega}$ with respect to the Gauss-Manin connection.
\end{rmk}

Let    $\mathcal L\subset \mc H^{f,\Omega}$ be an  opposite filtration (see Defintion \ref{opposite filtration}). Recall that
$$
B:=t\mc L \cap \Hzero^{f, \Omega}\iso  \Hzero^{f, \Omega}/t \Hzero^{f, \Omega}
$$
is isomorphic to $\Jac(f)$, and we have
$$
    \mc H^{f,\Omega}= \Hzero^{f, \Omega}\oplus \mathcal L
$$
with
    $$ \Hzero^{f, \Omega}=B[[t]], \quad  \mc L=t^{-1} B[t^{-1}]\quad\mbox{ and }\quad \mc H^{f,\Omega}=B((t)).$$

\begin{lem-defn}\label{lem-def-LR} Let $\mc L\subset \mc H^{f,\Omega}$ be an opposite filtration. We define an associated   $\mc R$-module $\mc L_{\mc R}$     by
$$
      \mc L_{\mc R}:=\{\mc L_{\mc R}(R_N)\}_N\quad\mbox{with}\quad \mc L_{\mc R}(R_N):= e^{(f-F)/t}(\mc L\otimes_{\C} R_N),
$$
 using \eqref{expHWR}.
 The splitting $\mc H^{f,\Omega}=\Hzero^{f, \Omega}\oplus \mc L$
induces a splitting  of $\check{\mc H}^{F, \Omega}(R_N)$,
{\upshape
   \begin{align*}
   \mbox{(a)}\quad \check {\mc H}^{F, \Omega}(R_N)=\check{\mc H}_{(0)}^{F,\Omega}(R_N)\oplus \mc L_{\mc R}(R_N)
\end{align*}
} for each $R_N$, where
  $$ \check{\mc H}^{F, \Omega}(R_N):=\mc H^{F,\Omega}\otimes_{\OO_S}R_N\quad \mbox{and}\quad \check{\mc H}_{(0)}^{F,\Omega}(R_N):=\mc H^{F,\Omega}_{(0)}\otimes_{\OO_S}R_N.$$
The $\mc R$-submodule $\mc L_{\mc R}$ of $\check{\mc H}^{F, \Omega}$ further  satisfies the following:
  {\upshape $$\mbox{(b) }\,\, t^{-1}  \mc L_{\mc R}(R_N)\subset  \mc L_{\mc R}(R_N); \qquad \mbox{(c) }\,\, \hat{\mc K}^{F, N}_\Omega(\mc L_{\mc R}(R_N),  \mc L_{\mc R}(R_N))\subset t^{-2}R_N[t^{-1}].$$}
  $\mc L_{\mc R}$ will be called the \emph{formal opposite filtration} associated to the opposite filtration $\mc L$.

 The splitting {\upshape (a)} gives rise to a decomposition of $\mc R$-modules
 $$\check{\mc H}^{F, \Omega}=\check{\mc H}_{(0)}^{F,\Omega}\oplus \mc L_\mc R.
 $$
Consequently, every section $s$ in $\Gamma(\mc R, \check{\mc H}^{F, \Omega})$ admits a unique splitting:
  $$s=s_+\oplus s_-$$
  with $s_+\in \Gamma(\mc R, \check{\mc H}_{(0)}^{F,\Omega})$ and $s_-\in \Gamma(\mc R, \mc L_{\mc R})$.
\end{lem-defn}
\begin{proof}
To show (a), let $\{\phi_j\}_j$, $\phi_j\in B$, represent a basis of $\Jac(f)$. Let $U\subset S$ be an open neighborhood of the reference point $0\in S$, such that $\{\pi_X^*\phi_j\}_j$   remain  representatives of a  basis of  $p_*\OO_{\mc C(F)}$. Let $B_{F}:=\mbox{Span}_{\OO_S|_U}\{\pi_X^*\phi_j\}_j\subset \mc H^{F,\Omega}|_U$, which is isomorphic to $B\otimes_\C \OO_S|_U$.  It follows that there are natural identifications
$$
   \check{\mc H}_{(0)}^{F,\Omega}(R_N)=B_{F}[[t]] \otimes_{\OO_S} R_N, \quad \check{\mc H}^{F, \Omega}(R_N)=  B_{F}((t)) \otimes_{\OO_S} R_N
$$
and
 $
e^{(f-F)/t}(\mc L\otimes_{\C}  R_N)=  e^{(f-F)/t} (t^{-1}B[t^{-1}] \otimes_{\C} R_N)$. We only need to show that  the composition  $\pi_-$ in the diagram
$$
\xymatrix{
   e^{(f-F)/t}(\mc L\otimes_{\C} R_N)\ar[dr]_{\pi_-} \ar[r] & \check{\mc H}^{F, \Omega}(R_N)\ar[d] \\
   & t^{-1}B_{F}[t^{-1}]\otimes_{\OO_S}R_N
}
$$
is an isomorphism. Here we have to transform elements of the form $e^{(f-F)/t} b, b\in B$ into elements of the form $B_{F}((t))\otimes_{\C}R_N$ up to $Q_{F}$-exact term. Clearly, $\pi_-$ is $R_N[t^{-1}]$-linear, and it is an isomorphism modulo the maximal ideal of $R_N$. It follows that $\pi_-$ is an isomorphism. Thus (a) holds.

Property (b) follows by construction.

To show (c),  we only need to prove for the pairing on $\varphi_{R_N}(e^{(f-F)/t}(\mc L\otimes_{\C} 1))$ since $\hat{\mc K}^{F, N}_\Omega$ is $R_N$-linear. Let $s, s'\in \mc L\subset \mc H^{f,\Omega}$, and $\check s=\{\check s_N\}, \check s'=\{\check s_N'\}$ be the formal flat extensions of $s, s'$ respectively. Let $T$ be the $\OO_S$-module ${\OO_S((t))}$ with the natural flat connection. Then
 $$
     \{\hat{\mc K}^{F, N}_\Omega(\check s_N, \check s_N')\}_N\in \Gamma(\mc R, \check{T})
 $$
Let $V$ be a holomorphic vector field on $S$.  $\pa_V$ is defined on $\Gamma(\mc R, \check{T})$ by Lemma \ref{lemdef-connection}, and we will denote by
$$
\{\pa_V  \hat{\mc K}^{F, N}_\Omega(\check s, \check s')\}_N:=\pa_V  \{\hat{\mc K}^{F, N}_\Omega(\check s_N, \check s_N')\}_N
$$
 Since $\check s, \check s'$ are flat, we have
 $$
   \pa_V \hat{\mc K}^{F, N}_\Omega(\check s, \check s')= \hat{\mc K}^{F,N}_\Omega(\nabla^\Omega_V\check s, \check s')+ \hat{\mc K}^{F, N}_\Omega(\check s, \nabla^\Omega_V\check s')=0
 $$
for any $V$ and $N$.
Hence,   $\hat{\mc K}^{F, N}_\Omega(\check s_N, \check s'_N)\in \C((t))\subset R_N((t))$ for all $N$. This implies
$\hat{\mc K}^{F, N}_\Omega(\check s_N, \check s'_N) =\hat{\mc K}_{F, 1}^\Omega(\check s_1, \check s'_1)\in t^{-2} \C[t^{-1}]$   under  the natural identification of $\C\subset R_{N}$ with $\C= R_{1}$.       Thus (c) holds.

The splittings (a) satisfy the  compatibility condition, so that they give rise to a splitting of $\check{\mc H}^{F, \Omega}$.

 Given a section $s=\{s_N\}_N$ in $\Gamma(\mc R, \check {\mc H}^{F, \Omega})$, we have a splitting $s_N=s_+(R_N)\oplus s_-(R_N)$ in $\check{\mc H}_{(0)}^{F,\Omega}(R_N)\oplus \mc L_{\mc R}(R_N)$ for every $s_N\in \check{\mc H}^{F, \Omega}(R_N)$. Clearly, both $s_+:=\{s_+(R_N)\}_N$ and $s_-:=\{s_-(R_N)\}_N$ satisfy the compatibility condition, and hence we obtain a splitting of sections $s=s_+\oplus s_-$.
\end{proof}

\begin{rmk}
The $R_N$-components of the formal flat extensions $e^{(f-F)/t}\mc L$ generate a $R_N$-submodule of $\check{\mc H}^{F, \Omega}(R_N)$ which coincides with $\mc L_{\mc R}(R_N)$. Hence,  we do not distinguish them and simply denote both by
  $e^{(f-F)/t}\mc L\otimes_{\C} R_N$.
\end{rmk}

\begin{lem}\label{lem-GMpreserve-L} The non-extended Gauss-Manin connection {\upshape $ \nabla^{\Omega}$} preserves ${\mc L}_{\mc R}$, i.e.
{\upshape $$
\nabla^{\Omega}: \Gamma(\mc R, \mc L_{\mc R})\to \Omega_S^1\otimes_{\OO_S}   \Gamma(\mc R, \mc L_{\mc R}).
$$
}
\end{lem}
\begin{proof}
 Let $s=\{s_N\}_N$ be a section in  $\Gamma(\mc R, \mc L_{\mc R})$,  and $V$ be a local section of $\mc T_S$.
 Each $s_N$ is an element in $e^{(f-F)/t}\mc L\otimes_\C R_N$.
Therefore we can write
 $s_{N+1}=\sum_i \varphi_{R_{N+1}}(e^{(f-F)/t})a_i\otimes_\C b_i$, where $a_i\in \mc L$ and $b_i\in R_{N+1}$ for all $i$.
   Since $\varphi_{R_N}(e^{(f-F)/t}a_i)$  is a flat extension of $a_i$ by Lemma \ref{lem-flatextension}, it follows that
$$ (\nabla^\Omega_V s)(R_N)= %%%\sum_i\Big( \nabla^\Omega(\varphi_{\tilde R}(e^{(F_0-F)/t})s_i)\otimes_\C r_i+\varphi_{\tilde R}(e^{(F_0-F)/t})s_i\otimes_\C \varphi_{\tilde R}(\pa_V \tilde r_i)
\sum_i \varphi_{R_N}(e^{(f-F)/t}a_i)\otimes_\C \varphi_{R_N}(\pa_V \tilde b_i)\in \mc L_{\mc R}(R_N),$$
where $\tilde b_i\in \OO_S$ is a lifting of $b_i$ for each $i$.
 \end{proof}
\subsection{Construction of formal primitive form}\label{construction_formal}

\begin{defn}\label{def-good-oppofil} An  opposite filtration  $\mc L$ of $\mc H^{f,\Omega}$ is called \emph{good} if $\mc L$ is preserved by {\upshape $\nabla^{\Omega}_{t\pa_t}=t\pa_t-{f\over t}$}.
\end{defn}

Here $\nabla^{\Omega}_{t\pa_t}$ is the restriction of the extended Gauss-Manin connection to the central fiber, which is a well-defined operator on $\mc H^{f,\Omega}$.
%% The condition says that $\mc L$ is preserved by a version of rescaling.

\begin{defn}\label{def-primitive-element}
An element $\zeta_0\in \mc H^{f,\Omega}_{(0)}$ is called \emph{primitive} with respect to a good opposite filtration $\mc L$ if
\begin{enumerate}
\item{\upshape (primitivity) } the projection of $\zeta_0$ generates $\Jac(f)=\mc H^{f,\Omega}_{(0)}/t \mc H^{f,\Omega}_{(0)}$ as $\Jac(f)$-module;
\item{\upshape (homogeneity) } ${\nabla_{t\pa_t}^{\Omega}\zeta_0}-r\zeta_0\in \mc L$ for some constant $r\in \C$.
\end{enumerate}
\end{defn}
\noindent We will describe all good opposite filtrations and primitive elements for weighted homogeneous polynomials in next section.

\begin{rmk}\label{rmk-goodsection}
  The notion of \emph{good sections} in \cite{Saito-primitive}  can be extended to a notion of    \emph{formal good section} $v: \Jac(f)\to \Hzero^{f, \Omega}$.  The formal good section $v$ coincides with the above  notion of a  good opposite filtration $\mc L$,
 under  the  identification   $v\mapsto \mc L:= t^{-1} v(\Jac(f))[t^{-1}]$ (recall Lemma \ref{opposite filtration}). The notion of a primitive element was also
 described in  \cite{Saito-primitive} without introducing the terminology, where it is generally expected    that there exists a unique  primitive  element with respect to a given good opposite filtration, up to a nonzero complex number.
\end{rmk}

Given an element $\zeta_0$ in $\Hzero^{f, \Omega}$, we obtain a flat section
     $e^{(f-F)/t}\zeta_0$ in $\Gamma(\mc R, {\check{\mc H}^{F, \Omega}})$ by Lemma \ref{lem-flatextension}.
From this we obtain a unique splitting:
\begin{align}\label{splitting-sections}
e^{(f-F)/t}\zeta_0=\zeta_++ \zeta_-
\end{align}
with $\zeta_+\in \Gamma(\mc R, \check{\mc H}_{(0)}^{F,\Omega})$ and $\zeta_-\in \Gamma(\mc R,  {\mc L}_{\mc R})$ by Lemma \ref{lem-def-LR} .

\begin{defn}
A pair $(\mc L, \zeta_0)$ is called good if $\mc L$ is a good opposite filtration and $\zeta_0$ is a primitive element with respect to $\mc L$.
\end{defn}
\begin{thm}\label{thm-primitive-local}

There is a bijection
$$\vartheta: \{\mbox{good pairs } (\mathcal{L}, \zeta_0)\}\longrightarrow \{\mbox{ formal primitive forms}\,\},$$
 defined by the splitting \eqref{splitting-sections}, i.e., given by
  $$(\mathcal{L}, \zeta_0)\mapsto \vartheta(\mathcal{L}, \zeta_0):=\zeta_+\,.$$
\end{thm}
We will divide the proof of the theorem  into the following three lemmas.

\begin{lem}\label{lem-formpritive}
 The map $\vartheta$ is well-defined. That is,  if $(\mc L, \zeta_0)$ is  good, then  $\zeta_+$ is a formal primitive form.
\end{lem}
\begin{proof}
 We check the four  conditions for defining a formal  primitive form in Definition \ref{formal-primitive-form}.
 For simplicity, we will refer to the corresponding formulas in Definition \ref{define-primitive-form}, and simply use the notations $\nabla^\Omega$ and $\mc K^{F}_\Omega$ instead of $\check \nabla^\Omega$ and $\hat{\mc K}^{F}_{\Omega}$, respectively.  Let $R\in \mc R$.

(P1)${}^\vee$ We need to show that
$$
  t\nabla^{\Omega} \zeta_{+}(R): \mc T_{S}\otimes_{\OO_{S}} R \to \check{\mc H}_{(0)}^{F,\Omega}( R)/ t\check{\mc H}_{(0)}^{F,\Omega}(R)
$$
is an isomorphism of free $R$-modules.
For any local section  $V$ of $\mc T_{S}$, by definition we have
  $[t\nabla^{\Omega}_V \zeta_{+}(R)]=[ \varphi_R((\pa_VF)\tilde \zeta_+(R)) ]\in \check{\mc H}_{(0)}^{F,\Omega}(R)/t \check{\mc H}_{(0)}^{F,\Omega}(R)$, where $\tilde \zeta_+(R)\in \Hzero^{F, \Omega}$ is a lifting of $\zeta_+(R)$.
That is, we have
  $$ t\nabla^{\Omega} \zeta_{+}(R): \mc T_{S}\otimes_{\OO_{S}} R \to p_*\OO_{\mc C(F)}\tilde\zeta_+(R)\otimes_{\OO_S}R \subset \check{\mc H}_{(0)}^{F,\Omega}( R)/ t\check{\mc H}_{(0)}^{F,\Omega}(R)\cong p_*\OO_{\mc C(F)}\otimes_{\OO_S}R.$$
Note that when $R= \OO_{S, 0}/\mathfrak{m}\cong \C$, $\zeta_{+}({\C})=\zeta_0$ and the inclusion $\subset$ is in fact an isomorphism due to the primitivity  of $\zeta_0$ (see Definition \ref{def-primitive-element}). In this case the  arrow also becomes an isomorphism, since
  the Kodaira-Spencer map \eqref{KSmap} is an isomorphism.
That is, $t\nabla^{\Omega} \zeta_{+}(\C)$ is an isomorphism. Hence,  $t\nabla^{\Omega} \zeta_{+}(R)$ is an isomorphism of $R$-modules
 by Nakayama Lemma.

 (P2)${}^\vee$  Let $V_1, V_2$ be local sections of $\mc T_{S}$.
 Since $e^{(f-F)/t}\zeta_0$ is flat, we have
$$
t\nabla^{\Omega}_{V_1} \zeta_{+}(R)=-t \nabla^{\Omega}_{V_1} \zeta_{-}(R).
$$
Since $t\nabla^{\Omega}_{V_1} \zeta_{+}(R)\in \check{\mc H}_{(0)}^{F,\Omega}(R)$, and $t\nabla^{\Omega}_{V_1}\zeta_{-}(R)\in t \mc L_{\mc R}(R)$ by Lemma \ref{lem-GMpreserve-L}, we find
$$
t\nabla^{\Omega}_{V_1} \zeta_{+}(R)\in  \check{\mc H}_{(0)}^{F,\Omega}(R)\cap t \mc L_{\mc R}(R).
$$
 By Lemma \ref{lem-def-LR}, we have
 $$
  \mc K_\Omega^{F}\bracket{\mbox{-},\mbox{-}}: \mc L_{\mc R}(R)\otimes \mc L_{\mc R}(R)\to t^{-2} R[t^{-1}].$$
Hence,
$$
\mc K_\Omega^{F}\bracket{\mbox{-},\mbox{-}}:  \big({\check{\mc H}_{(0)}^{F,\Omega}(R)\cap t \mc L_{\mc R}(R)}\big) \otimes  \big({\check{\mc H}_{(0)}^{F,\Omega}(R)\cap t \mc L_{\mc R}(R)}\big) \to R[[t]]\cap R[t^{-1}]=R,$$
where the value does not depend on $t$. In particular,  the property in (P2)${}^\vee$ holds:
$$
\mc K_\Omega^{F}\big({\nabla^{\Omega}_{V_1}\zeta_{+}(R), \nabla^{\Omega}_{V_2} \zeta_{+}(R)}\big) \in t^{-2}R.
$$

(P3)${}^\vee$  Let $V_1, V_2$ be local sections of $\mc T_{S}$.
The identity $\nabla^{\Omega}_{V_1}\nabla^{\Omega}_{V_2} \zeta_{+}(R)=- \nabla^{\Omega}_{V_1}\nabla^{\Omega}_{V_2}\zeta_{-}(R)$
implies
$$
\nabla^{\Omega}_{V_1}\nabla^{\Omega}_{V_2}\zeta_{+}(R)\in t^{-2}\check{\mc H}_{(0)}^{F,\Omega}(R)\cap \mc L_{\mc R}(R)=\big({ t^{-2}\check{\mc H}_{(0)}^{F,\Omega}(R)\cap t^{-1}\mc L_{\mc R}(R)}\big)\oplus  \big({t^{-1}\check{\mc H}_{(0)}^{F,\Omega}(R)\cap \mc L_{\mc R}(R)}\big).
$$
This implies the first property in (P3)${}^\vee$:
$$
\mc K_\Omega^{F}\big(\nabla^{\Omega}_{V_1}\nabla^{\Omega}_{V_2}\zeta_{+}(R), \nabla^{\Omega}_{V_3}\zeta_{+}(R)\big) \in t^{-3}R\oplus t^{-2}R.
$$

For any $[s]\in \mc H^{f,\Omega}$,  it is easy to see that
\begin{align}\label{formula-in-mainthm}
   \nabla^{\Omega}_{t\pa_t}[\varphi_R(e^{(F_0-F)/t}\pi_X^*(s))]=[\varphi_R(e^{(F_0-F)/t} \pi_X^*\big(({t\pa_t-{f\over t}})s\big))],
\end{align}
the right hand side of which equals $[\varphi_R(e^{(F_0-F)/t} \pi_X^*(\nabla^{\Omega}_{t\pa_t}s))]$ by definition.
In particular if  $[s]\in \mc L\subset \mc H^{f, \Omega}$, then $\nabla^{\Omega}_{t\pa_t}[s] \in \mc L$ (since $\mc L$ is a good opposite filtration). Consequently,
$$
\nabla^{\Omega}_{t\pa_t} [\varphi_R(e^{(F_0-F)/t}\pi_X^*(s))]\in \mc L_{\mc R}(R).
$$
That is,  $\nabla^{\Omega}_{t\pa_t}$ preserves $\mc L_{\mc R}(R)$. On the other hand $\nabla^{\Omega}_{t\pa_t}: \check{\mc H}_{(0)}^{F,\Omega}(R)\to t^{-1}\check{\mc H}_{(0)}^{F,\Omega}(R)$, it follows that
    $$
  \nabla^{\Omega}_{t\pa_t} \nabla^{\Omega}_{V_1} \zeta_{+}(R)=-\nabla^{\Omega}_{t\pa_t} \nabla^{\Omega}_{V_1} \zeta_{-}(R)\in t^{-2}\check{\mc H}_{(0)}^{F,\Omega}(R)\cap \mc L_{\mc R}(R).
  $$
By the same arguments  as above, the second property in (P3)${}^\vee$  holds.

(P4)${}^\vee$ First, let us observe that the operator $\nabla^{\Omega}_{t\pa_t}+\nabla^{\Omega}_E$ preserves $\mc H^{F,\Omega}_{(0)}$. Indeed,  by the definition of the Euler vector field $E$, there exist $g_i\in \Gamma(Z, \OO_{Z}), i=1, \cdots, n,$ such that
  $
F-\pa_EF=\sum_{i}g_i\pa_{z^i}F$, where we remind that a projection $\pi_X: Z\to X$ is fixed.
 Hence, as operators on $\mc H^{F,\Omega}$, we have
\begin{align*}
   \big({\nabla^{\Omega}_{t\pa_t}+\nabla^{\Omega}_{E}}\big)[s]&=[\big(t\pa_t-{F\over t}+\pa_E+{\pa_EF\over t}\big)(s)]\\
            &= [\big(t\pa_t+\pa_E-t^{-1}\sum_i g_i\pa_{z^i}F\big)(s)]=
                 [\big(t\pa_t+\pa_E+\sum_{i}\lambda {\pa \over \pa z_i}{g_i\over \lambda}+\sum_i g_i{\pa \over \pa z_i}\big)(s)],
\end{align*}
where $[s]$ denotes a local section of $\mc H^{F, \Omega}$, and  which clearly preserves $\mc H^{F,\Omega}_{(0)}$. Therefore,
 $$
 \big({\nabla^{\Omega}_{t\pa_t}+\nabla^{\Omega}_{E}}\big)\check{\mc H}^{F, \Omega}_{(0)}(R)\subset \check{\mc H}^{F, \Omega}_{(0)}(R).
 $$
 Similarly
 $$
 (\nabla^{\Omega}_{t\pa_t}+\nabla^{\Omega}_E)\mc L_{\mc R}(R)\subset \mc L_{\mc R}(R)
 $$
 by \eqref{formula-in-mainthm} and  Lemma \ref{lem-GMpreserve-L}. We find the decomposition
  $$\big((\nabla^{\Omega}_{t\pa_t}+\nabla^\Omega_E)e^{(f-F)/t}\zeta_0\big)(R)
   =\bracket{\nabla^{\Omega}_{t\pa_t}+\nabla^{\Omega}_E}\zeta_+(R)+\bracket{\nabla^{\Omega}_{t\pa_t}+\nabla^{\Omega}_E}\zeta_-(R)\in \check{\mc H}^{F, \Omega}_{(0)}(R)\oplus \mc L_{\mc R}(R).$$

On the other hand,
 $$
 \big(\nabla_{t\pa_t}^\Omega e^{(f-F)/t}\zeta_0\big)(R)-r\big(e^{(f-F)/t}\zeta_0\big)(R)\in \mc L_{\mc R}(R)
 $$
which follows from \eqref{formula-in-mainthm} and the homogeneity of $\zeta_0$. It leads to another decomposition
 $$\big((\nabla^{\Omega}_{t\pa_t}+\nabla^\Omega_E)e^{(f-F)/t}\zeta_0\big)(R)= \big(\nabla_{t\pa_t}^\Omega e^{(f-F)/t}\zeta_0\big)(R)\in r(e^{(f-F)/t}\zeta_0)(R)+\mc L_{\mc R}(R)=r \zeta_+(R)+\mc L_{\mc R}(R).$$
By the uniqueness of the splitting of $\big((\nabla^{\Omega}_{t\pa_t}+\nabla^\Omega_E)e^{(f-F)/t}\zeta_0\big)(R)$, we find
 $$
 \bracket{\nabla^{\Omega}_{t\pa_t}+\nabla^{\Omega}_E}\zeta_+(R)=r\zeta_+(R).
$$
In addition, it is easy to see the all (P2)${}^\vee$, (P3)${}^\vee$ and (P4)${}^\vee$ above satisfy the compatibility condition with respect to the inverse system $\mc R$. Hence, $\zeta_+$ is a formal primitive form.

\end{proof}
\begin{lem-defn}\label{lem-detgoodpair}
   There is a map
$\psi: \{\mbox{ formal primitive forms}\,\}  \longrightarrow  \{\mbox{good pairs } (\mathcal{L}, \zeta_0)\}.$
\end{lem-defn}
\begin{proof}
Let $\zeta_+$ be a formal primitive form, and $\mathcal{B}(R_N)$ be the image of the map
$$
   t\nabla^{\Omega}\zeta_+(R_N): \mc T_{S}\otimes_{\OO_S} R_N\to \check{\mc H}_{(0)}^{F,\Omega}(R_N), \quad V\to t\nabla^{\Omega}_V\zeta_+(R_N),
$$
which is  a $R_N$-submodule of $\check{\mc H}_{(0)}^{F,\Omega}(R_N)$.  Then we have
$$
\check{\mc H}^{F, \Omega}(R_N)=\mathcal{B}(R_N)((t)), \quad \check{\mc H}_{(0)}^{F,\Omega}(R_N)=\mathcal{B}(R_N)[[t]]
$$
Define $\mc L_{\mc R}(R_N):=t^{-1}\mathcal{B}(R_N)[t^{-1}]\subset \check{\mc H}^{F, \Omega}(R_N)$.

 \noindent \textbf{Claim A}: $\{\mc L_{\mc R}(R_N)\}_N$ defines a sub $\mc R$-module of $\check{\mc H}^{F, \Omega}$ preserved by the Gauss-Manin connection, which is identical to the flat extension of $\mathcal L_{\mathcal{R}}(\C)$.

Then   $\mc L:=\mc L_{\mc R}( {\C})$ defines a good opposite filtration of $\mc H^{f, \Omega}$ (by using   (P2)${}^\vee$, (P3)${}^\vee$), and we obtain  a primitive  element $\zeta_0:=\zeta_+( {\C})$ with respect to such $\mc L$ (by using   (P1)${}^\vee$, (P4)${}^\vee$). The map $\psi$ is defined by $\psi(\zeta_+):=(\mc L, \zeta_0)$.

Now we prove Claim A. This is essentially reversing the argument in Lemma \ref{lem-formpritive}, with extra care on the polynomial dependence on $t^{-1}$. Frist, we have

\noindent \textbf{Claim B}: {\itshape Let $a\in \mathcal L_{\mathcal{R}}(\C)$. Then there exists a unique section $\alpha\in \Gamma( {\mathcal{R}}, \mathcal L_{\mathcal{R}})$ which is flat with respect to the non-extended Gauss-Manin connection $\nabla^\Omega$ and satisfies the initial condition $\alpha({\C})=a$.}

In fact, by the holonomicity of formal primitive forms, we have
\begin{align}\label{eqn-betaijk}
    \nabla^\Omega_i\nabla^\Omega_j \zeta_+=\beta_{ij}^k \nabla^\Omega_k \zeta_+
\end{align}
where $\nabla^\Omega_i:=\nabla^\Omega_{\pa_{u_i}}$, and
$$
   \beta_{ij}^k\in \hat {\mathcal O}_{S,0}\oplus t^{-1} \hat{\mathcal O}_{S, 0}.
$$
 Let us decompose $\beta_{ij}^k$ into
$$
  \beta_{ij}^k:=\sum_{m\geq 0} \beta_{ij,(m)}^k
$$
where $\beta_{ij,(m)}^k$ is the homogeneous component of degree $m$ in the coordinates $\{u_i\}$ of $S$.
The initial condition $\alpha({R_1})=a$ determines uniquely a set  $\{\alpha^i_{(0)}\}\subset \mathbb{C}$ such that
$$\alpha|_{R_1}= \sum_i \alpha^i_{(0)} \nabla_i^\Omega\zeta_+(R_1).$$
We define recursively by
$$
   \alpha^i_{(N)}:=-{1\over N}\sum_{m+l=N-1}\sum_{k,j} \beta_{jk,{(m)}}^i \alpha^k_{(l)}u_j,\quad\forall N\in \mathbb{Z}_{>0}
$$
and $$\alpha^i(R_N):=\sum_{j=0}^{N-1}\alpha^i_{(j)},\quad \alpha(R_N):=\sum_i \alpha^i(R_N)\nabla_i^\Omega\zeta_+(R_N),\quad\forall N\in \mathbb{Z}_{>1}.$$
 {{ All  $\beta_{jk,{(m)}}^i$ and all $\alpha_{(0)}^{i}$ are in $R_N[t^{-1}]$}},    so is   $\alpha^i(R_N)$ for all $i$ and $N$. Hence, $\alpha(R_N)\in t^{-1}\mathcal{B}(R_N)[t^{-1}]=\mathcal{L}_{\mathcal{R}}(R_N)$.
It is easy to check that $\alpha$ satisfies the compatibility condition, so that $\alpha\in \Gamma(\mathcal{R}, \mathcal{L}_{\mathcal{R}})$.
Moreover,  $\alpha:=\{\alpha(R_N)\}$ satisfies
\begin{align}\label{eqn-flat}
   \pa_{u_j} \alpha^i(R_N)+\sum_k \beta_{jk}^i(R_N) \alpha^k(R_N)=0
\end{align}
for all $j$ and $N$. That is, $ \nabla^\Omega_j \alpha(R_N)=0$ for all $j$ and $N$, due to \eqref{eqn-betaijk}.
Hence, $\alpha$ is flat with respect to $\nabla^\Omega$, and $\alpha(R_1)=a$. The uniqueness is obvious. This gives Claim B.

In Lemma \ref{lem-flatextension},  we have shown that for any $a\in \mathcal L_{\mathcal{R}}(\C)$,
$\alpha:=\{e^{(f-F)/t}|_{R_N}a\}$ is a flat section in $\Gamma(\mathcal{R}, \check{\mathcal{H}}^{F, \Omega})$ (with respect to the non-extended $\nabla^\Omega$) with initial condition $\alpha(R_1)=a$. Hence,   $\alpha\in \Gamma( {\mathcal{R}}, \mathcal L_{\mathcal{R}})$  by Claim B.
This implies that
$$
e^{(f-F)/t}|_{R_N}\mathcal L_{\mathcal{R}}(\C)\otimes_{\mathbb C}R_N\subset  \mathcal L_{\mathcal{R}}(R_N)
$$
On the other hand, both $ \mathcal L_{\mathcal{R}}(R_N)$ and $e^{(f-F)/t}|_{R_N}\mathcal L(\C)\otimes_{\mathbb C}R_N$ splits $\check{\mathcal H}_{(0)}^{F, \Omega}(R_N)$ in $\check{\mathcal H}^{F, \Omega}(R_N)$. Hence, they must be equal. This proves Claim A.

\end{proof}

\begin{lem}\label{lem-uniqueness}
The compositions  $\psi\circ \vartheta$  and $\vartheta\circ \psi$ are both identity maps.
\end{lem}

\begin{proof}
It is straight-forward to check the first composition.
Now we let $\zeta_+$ be a formal primitive form. Then $e^{(F-f)/t}\zeta_+$
  is a well-defined element of the inverse system $\{\mathcal H^{f, \Omega}\otimes_{\mathbb C}R_N
\}_N$.
By Lemma \ref{lem-detgoodpair}, we obtain a good pair  $\psi(\zeta_+)=(\mc L, \zeta_0)$ with   $(e^{(F-f)/t}\zeta_+)(\C)=\zeta_0$.
 Moreover,
$$\pa_{u_i}(e^{(F-f)/t}\zeta_+)(R_N)=e^{(F-f)/t}\nabla^\Omega_{i}\zeta_+ (R_N)\in e^{(F-f)/t}t^{-1}\mathcal{B}_{\mathcal{R}}(R_N)\subset
  \mathcal L\otimes_{\mathbb{C}}R_N.$$
   Therefore, we have
$$
  (e^{(F-f)/t}\zeta_+)(R_N)\in \zeta_0+\mathcal L\otimes_{\mathbb C}R_N, \mbox{ equivalently, }    \zeta_+(R_N)\in e^{(f-F)/t}|_N(\zeta_0)+e^{(f-F)/t}|_{N}(\mathcal L\otimes_{\mathbb C}R_N),
$$
for any $N$.
Hence, $\zeta_+$ is obtained from the splitting \eqref{splitting-sections}. That is, $\zeta_+=\vartheta(\mc L, \zeta_0)$.
\end{proof}

\bigskip
\begin{proof}[Proof of Theorem \ref{thm-primitive-local}] The statement is a direct consequence of Lemmas \ref{lem-formpritive}, \ref{lem-detgoodpair} and \ref{lem-uniqueness}.
\end{proof}
\begin{cor}\label{rmk-primitive}
Let $\mc L$ be a good opposite filtration and $\zeta_0$ be a primitive element with respect to $\mc L$.
Then there is a unique section   $\eta$ in $\Gamma(\mc R, \check{\mc H}^{F, \Omega}_{(0)})$ satisfying
$$e^{(F-f)/t}\eta(R)\in \zeta_0\otimes_\C 1+ \mc L\otimes_{\C} R$$
for all $R\in \mc R$. Furthermore,  such   $\eta$ is precisely given by the formal primitive form
$\zeta_+$ defined in \eqref{splitting-sections}.
\end{cor}

\iffalse
\begin{proof}
  For any  $\eta$ in $\Gamma(\mc R, \check{\mc H}^{F, \Omega}_{(0)})$, it follows from  Lemma \ref{lem-flatextension} that  $e^{(F-f)/t}\eta(R)$ is an element in $\mc H^{f}\otimes_{\C}R$. In particular, $\zeta_+$ satisfies the required properties. Moreover, it is unique one,
 following from the uniqueness of the splitting of $\zeta_0\otimes_\C 1$ in $\mc H^{f}\otimes_\C R=(\Hzero^{F, \Omega}\otimes R)\oplus (\mc L\otimes_\C R)$.
\end{proof}
\fi

The above corollary  provides a convenient way to compute the primitive form in practice. In particular, this leads to our algorithm in Section \ref{algorithm} and is used in Section \ref{examples} to compute primitive forms for explicit examples.

\subsection{Discussion on the analyticity}\label{subsec-analyticity} We have constructed the formal primitive form in the formal neighborhood of the universal unfolding with respect to the choice of a good opposite filtration together with a primitive  element. In this subsection, we will discuss when  the  formal primitive form as constructed in Theorem \ref{thm-primitive-local} extends to certain analytic neighborhood.   We will use  the comparison between the formal construction and the original analytic construction by the third author, as described below.

Consider the natural map
$$
     p_*\Omega^n_{Z/S}/ dF\wedge d \bracket{p_*\Omega^{n-2}_{Z/S}}\to \mc H^{F,\Omega}_{(0)}, \quad [\xi] \to [{\xi\over \Omega_{Z/S}}].
$$
 There is a canonical filtration on the left hand side defined by the power of the operator $\nabla_{\delta}^{-1}$ in \cite{Saito-primitive}, which coincides with the filtration  $\mc H^{F,\Omega}_{(-k)}$ on the right hand side under the above map. Actually, there are also the Gauss-Manin connection and the higher residue pairing defined on the left hand side, with which the above map is equivariant. Such map becomes an isomorphism if we take the formal completion of the left hand side with respect to the above filtration \cite{Saito-residue}.

The space $p_*\Omega^n_{Z/S}/ dF\wedge d \bracket{p_*\Omega^{n-2}_{Z/S}}$ is sometimes called the \emph{Brieskorn lattice}, which has a purely analytic nature. The original primitive form   is formulated as an analytic object inside this space satisfying the same properties (as in Definition \ref{define-primitive-form}), where there is the  notion of $\emph{good sections}$ and the primitive  element was similarly described.
An (analytic) good section $v$  gives a (formal) good opposite filtrations $\mc L$  as constructed in  Remark \ref{rmk-goodsection}.   It was shown in \cite{Saito-primitive}  the analyticity of the primitive form via solving a version of Riemann-Hilbert-Birkhoff problem for $f$ with being an isolated singularity. The explicit construction we provide in the present paper is the formulation of this correspondence in the formal setting. In particular, we have the maps
$$
\xymatrix{
  \fbracket{\mbox{analytic primitive forms}} \ar[r]\ar[d] & \fbracket{\mbox{formal primitive forms}}\ar[d]\\
  \fbracket{\mbox{good sections}} \ar[r] & \fbracket{\mbox{good opposite filtrations}}
}
$$
Here the top row is defined by pullback to the formal neighborhood.  The second vertical map  is fibered by the primitive  elements.  As indicated by M. Saito  \cite{Mo.Saito-uniqueness}, the first vertical map might not be surjective for general $f$, and the notion of very good sections will be needed. As Hertling pointed to us, the bottom  row is not surjective in general. Nevertheless, if $f$ is a   weighted homogeneous polynomial with an isolated singularity, then the  first  vertical map  is also  surjective,  fibered by the primitive  elements. Moreover, the bottom row becomes a bijection, and the following holds.

\begin{thm}\label{thm-analytic} Let $f$ be a weighted homogeneous polynomial with an isolated critical point, and $(\mc L, \zeta_0)$ be a good pair.
Then the inverse limit of the formal primitive form  $\zeta_+$ defined in \eqref{splitting-sections} is the Taylor series expansion of an analytic primitive form.
\end{thm}

The reason we use the formal setting is that the explicit formula $e^{(f-F)/t}$ allows us to compute the Taylor series expansion of the primitive form up to arbitrary finite order. In the next section we will study the moduli space of primitive forms for all
 weighted homogeneous polynomials, and provide an explicit computation algorithm for the primitive forms.

  \section{Primitive forms for weighted homogeneous polynomials}

  In this section, we will present applications of Theorem \ref{thm-primitive-local} when $f$ is a weighted homogeneous polynomial.
  We will describe the moduli space of all primitive forms. In addition, we will   provide a concrete algorithm to compute the Taylor series expansions of primitive forms up to an arbitrary order.

 \subsection{Weighted-homogeneity}
A polynomial   $f: (\C^n, \mathbf{0})\to (\C, 0)$ is called  a
 \textit{ weighted homogeneous  polynomial} (of  total degree $1$) with weights $\bracket{q_1,\cdots, q_n}$, where each $q_i$ is a rational number with $0<q_i\leq {1\over 2}$,  if     $f(\lambda^{q_1}z_1,\cdots, \lambda^{q_n}z_n)=\lambda  f(z_1,\cdots, z_n)$ holds for all $\lambda\in \mathbb{R}_{>0}$.
  The rational numbers $q_1, \cdots, q_n$ are called \textit{ weight  degrees}. It defines the weight degree for polynomials by assigning $
    \deg z_i=q_i
$. In particular, $
  \deg f=1.
$

\begin{rmk}
 Since $\sum_{i=1}^nq_iz_i\pa_{z_i}f=f$,  $f$ belongs to
   its Jacobian ideal $(\pa_{z_1}f, \cdots, \pa_{z_n} f)$. In fact, more is true:   a holomorphic function-germ  $g: (\C^n, \mathbf{0})\to (\C, 0)$, with $\mathbf{0}$ being an isolated critical point,  is analytically   equivalent to a weighted homogeneous polynomial  if and only if
     $g$ belongs to its Jacobian ideal    \cite{Saito-quasihomogeneous}.
\end{rmk}

  Throughout this section, we will fix one such $f$. Let
\begin{align}\label{invariant}
  s(f):=\sum\limits_{i=1}^n\bracket{1-2q_i},
\end{align}
which was introduced in \cite{Saito-simplyElliptic} to classify  weighted homogenous polynomials $f$ satisfying $0<s(f)\leq 1$. The invariant $s(f)$ sometimes borrows the name \emph{central charge} since it coincides with the notion in the superconformal field theory associated to   the Landau-Ginzburg  model defined by $f$.

\begin{eg}[Calabi-Yau hypersurface] Every weighted homogenous polynomial $f(z_1, \cdots, z_n)$ with $\sum_{i=1}^nq_i=1$ defines a Calabi-Yau hypersurface $f=0$ in the weighted projective space $\mathbb{P}^{q_1,\cdots,q_n}$, whose dimension is equal to $s(f)=n-2$.
\end{eg}

  Let $\{\phi_1, \cdots, \phi_\mu\}$ be a set of weighted homogeneous polynomials that represent a   basis of $\Jac(f)$ with $\deg \phi_1\leq \deg\phi_2\leq\cdots \leq \deg \phi_\mu$.
  As a well-known fact, we have  $0= \deg \phi_1<\deg \phi_i<\deg \phi_\mu=s(f)$ {for any} $1<i<\mu$, and
$$
   \label{eqn-degIJ}     \deg \phi_i+\deg\phi_j=s(f) \quad\mbox{whenever}\quad i+j=\mu+1.
$$

  Take $Z\subset X\times S\subset \C^n\times \C^\mu$ with $S\subset \C^\mu$ a (small)  Stein domain containing the origin, such that the holomorphic function $F: Z\to \C$, defined by
  \begin{align}\label{universeal-unfolding-weight-homo}
    F(\mathbf{z}, \mathbf{u})=f(\mathbf{z})+\sum_{j=1}^\mu u_j\phi_j(\mathbf{z}),
  \end{align}
 is a universal unfolding. Then $(Z, S, p, F)$ is a frame (see Definition \ref{def-fourtuple}).
 We will extend the weight degree to deformation parameters
 \begin{align}
 \deg u_j :=1-\deg \phi_j
 \end{align}
 for each $j$. $F$ becomes a weight-homogeneous polynomial in $(\mathbf{z},\mathbf{u})$,
    $$
    \bracket{{\sum_{i=1}^n q_i z_i\pa_{z_i}+\sum_{j=1}^\mu (\deg u_j) u_j \pa_{u_j}}}F=F.
$$
 As a consequence, the Euler vector field (defined by \eqref{def-Euler-field}) is given by
\begin{align}\label{euler-field}
   E=\sum_{j=1}^\mu (\deg u_{j}) u_{j} \pa_{u_j}.
\end{align}
We will discuss   primitive forms for weighted homogeneous polynomials $f$ with respect to $(Z, S, p, F)$ and the choice
  $$\Omega_{Z/S}:=dz_1\wedge\cdots \wedge dz_n.$$

Now we briefly review a partial classification  of weighted homogeneous polynomials   (of total degree $1$).
 Two function-germs $g_1:(\C^n, \mathbf{0})\to (\C, \mathbf{0})$ and $g_2:(\C^m, \mathbf{0})\to (\C, \mathbf{0})$ are called \textit{stably equivalent}
if $g_1(z_1, \cdots, z_n)+z_{n+1}^2+\cdots+z_k^2$ and $g_2(\tilde z_1, \cdots, \tilde z_m)+\tilde z_{m+1}^2+\cdots +\tilde z_k^2$ define  the same function-germ
$(\C^k, \mathbf{0})\to (\C, \mathbf{0})$ for some $k\geq \max\{m, n\}$ up to a local analytic coordinate transformation. By Proposition \ref{prop-stableequiv}, it is sufficient to deal with $f$ up to stable equivalence. The classification  of $f$, which  was started  by V.I. Arnold,   is still open in general. However, the case of $s(f)\leq 1$
 was already completely classified. For $s(f)<1$, there are in total  5 lists of $f$, which  are called  \textit{{simple singularities}}  \cite{Arnold-Goryunov-Lyashko-Vasilev} or \textit{$ADE$-singularities};
for $s(f)=1$, there are in total  3 lists, which are called \textit{simple elliptic singularities} in \cite{Saito-simplyElliptic}. Among those weighted homogeneous polynomials satisfying $1<s(f)<2$,  there are 14 polynomials (in three variables) named \textit{exceptional unimodular singularities}   of type $E_{12}, E_{13}, E_{14}, Z_{11}, Z_{12}, Z_{13}, W_{12}, W_{13}$, $Q_{10}, Q_{11}, Q_{12}, S_{11}, S_{12}$ and $U_{12}$ respectively \cite{Arnold-Goryunov-Lyashko-Vasilev}. The degrees of the weighted homogeneous basis for these 14 polynomials are given by
\begin{equation} \label{eqn-degPhiExpo}
   \deg \phi_i={m_i-a-b-c\over h}+1
\end{equation}
 with %a complete list of
   $a, b, c, h$ and the \textit{exponent} $m_i$  read off directly  from Table 3 of \cite{Saito-exceptional}.

\iffalse
\begin{prop}[\cite{Saito-simplyElliptic}]
Let $W$ be a weighted homogeneous polynomial. If $\hat C_W<1$, then $W$ is stably equivalent to one of the following list:

%\begin{eg}[ADE-singularity] $\mbox{}$%%The weighted-homogeneous superpotential with central charge $<1$ is classified by the famous ADE-singularities.
 \begin{itemize}
\item $A_m$-singularity: % ($n\geq 1$):
$
    {}\quad\! W(z)=z^{m+1},  \quad\quad\,\,\,\,\,\, q_z={1\over m+1}, \qquad\qquad \quad \! \hat C_W={m-1\over m+1}.
$
\item $D_m$-singularity:% ($n\geq 4$):
${}\,W(x, y)=x^{m-1}+x y^2,   q_x={1\over m-1}, q_y={m-2\over 2m-2},   \hat C_W={m-2\over m-1}.
$
\item $E_6$-singularity:
$
  {}\, W(x,y)=x^3+y^4, \quad\,\,\,\,\, q_x={1\over 3},\, q_y={1\over 4}, \quad\quad\,\,\, \hat C_W={5\over 6}.
$
\item $E_7$-singularity:
$
  {}\, W(x,y)=x^3+xy^3, \quad\,\, q_x={1\over 3},\,  q_y={2\over 9},  \quad\quad\,\,\,\hat C_W={8\over 9}.
$
\item $E_8$-singularity:
$
 {}\,  W(x,y)=x^3+y^5, \quad\,\,\,\,\, q_x={1\over 3}\,, q_y={1\over 5},   \quad\quad\,\,\,\hat C_W={14\over 15}.
$
\end{itemize}
If $\hat C_W=1$, then  $W$ is stably equivalent to one of the following list:
 \begin{itemize}
  \item $E_6^{(1,1)}: x^3+y^3+z^3+axyz,\qquad q_x=q_y=q_z={1\over 3}$.
  \item $E_7^{(1,1)}: x^4+y^4+z^2+axyz,\qquad q_x=q_y={1\over 4}, q_z={1\over 3}$.
  \item $E_8^{(1,1)}: x^6+y^3+z^2+axyz,\qquad q_x={1\over 6}, q_y={1\over 3}, q_z={1\over 2}$.
 \end{itemize}
%\end{eg}
\end{prop}

\fi

\subsection{A grading operator} We can adapt the weight degrees to polyvector fields.

  \begin{defn}  Define an    { operator} $\mathbb{E}_f: \PV(X)((t))\to \PV(X)((t))$,    by
$$
    \mathbb{E}_f:=t{\pa_t}+\sum_{i=1}^n q_i \big({z_i\pa_{z_i}-\bar z_i\pa_{\bar z_i}}\big) +\sum_{i=1}^n \big({(1-q_i)\pa_{i}\wedge {\pa\over \pa \pa_{i}}-q_i d\bar z_i{\pa\over \pa d\bar z_i}}\big),
$$
where the operator ${\pa\over \pa d\bar z_i}$ is defined similarly to ${\pa\over \pa \pa_i}$.
\end{defn}
\noindent Clearly, $\mathbb{E}_f$ perserves $\PV_c(X)((t))$.

We may formally treat $t$ (resp. $\bar z_i, \pa_i, d\bar z_i$) as of weight degree $1$ (resp. $-q_i$, $1-q_i$, $-q_i$).  Then the differential $Q_{f}$ is an operator of weight degree $0$ on
$\PV(X)((t))$. The  next lemma becomes a trivial consequence of this observation.
\begin{lem}\label{lemmaHodgeweightcomm}
The  operator $\mathbb{E}_f$ commutes with the coboundary operator $Q_{f}$ (recall \eqref{Qf-fiber}).
\end{lem}

It follows that     $\mathbb{E}_f$ is well defined on   the cohomology $\mc H^{f,\Omega}$, and it induces  a $\mathbb{Q}$-grading on $\Hzero^{f, \Omega}$ and $\mc H^{f,\Omega}$.
  Classes  in $\mc H^{f,\Omega}$ are represented by elements in $\Gamma(X, \OO_X)[[t]]$, for which
  $$
  \mathbb{E}_f(ht^k)=
      \big({t{\pa_t}+\sum_i q_iz_i{\pa_{z_i}}}\big)(ht^k),\quad\mbox{where }\,h\in \Gamma(X, \OO_X).
      $$
The  $\Q$-grading is just    the weight degree of the representatives. %, by extending $t$ with weight degree  $1$.
In the following discussions, we will always assume the extended weight degree assignments
$$
  \deg z_i=q_i,\quad \deg u_j=1-\deg\phi_j, \quad \deg t=1, \quad 1\leq i\leq n, 1\leq j\leq \mu.
$$

Recall that the higher residue $\widehat{\mbox{Res}}_f$ is defined by $\eqref{higher-residue-fiberversion}$.
\begin{prop}\label{degree-of-highresidue}
 The higher residue map {\upshape $\widehat{\mbox{Res}}_f: \Hzero^{f, \Omega}\to \C[[t]]$} is homogeneous  of  weight degree $-s(f)$:  if $\deg \alpha=m$, then
  {\upshape $\deg (\widehat{\mbox{Res}}_f(\alpha))=m-s(f)$}.
\end{prop}

\begin{proof} This follows from a direct calculation on integration by part and we leave it to the reader.
\end{proof}

The next corollary is  a direct consequence of the above proposition

 \begin{cor}\label{degree-of-higherresiduePairing-fiberversion}
    The higher residue pairing   $\mc K_\Omega^{f}\bracket{\mbox{-}, \mbox{-}}:  \Hzero^{f, \Omega}\otimes \Hzero^{f, \Omega}\to \C[[t]]$ is of weight degree $-s(f).$
  \end{cor}

The higher residue map $\widehat{\mbox{Res}}^f$ induces a map  on  the associated graded vector space of   the Hodge filtration, the restriction of which to
  $\Hzero^{f, \Omega}/ \mc H^{f,\Omega}_{(-1)}=\Jac(f)$ coincides with   the usual residue map by  Proposition \ref{compatible-residue}. As a direct consequence of Proposition \ref{degree-of-highresidue}, we reprove the next well-known fact.
 \begin{cor}
  The residue  $\Res: \Jac(f) \to \C$
  is only non-zero at the weighted homogeneous component of   (top) weight degree $s(f)$ in $\Jac(f)$.
 \end{cor}

 The grading operator $\mathbb{E}_f$ on $\PV(X)((t))$ is  naturally extended to the universal unfolding
   $$\mathbb{E}_{F}:=\mathbb{E}_f+\sum_{j=1}^\mu (\deg u_{j}) \big(u_{j} \pa_{u_j}-\bar u_j\pa_{\bar u_j}-d\bar u_j{\pa\over \pa d\bar u_j} \big),$$
 which acts on $\PV(Z/S)((t))$ and  commutes  with $Q_{F}$. Hence, $\mathbb{E}_{F}$ acts   on $\Hzero^{F, \Omega}$, and induces a $\mathbb{Q}$-grading on it.
 As a family version of Proposition \ref{degree-of-highresidue}, we have
 \begin{prop}\label{degree-of-highresidue-family}
 The higher residue map {\upshape $\widehat{\mbox{Res}}^{F}: \Hzero^{F, \Omega}\to \OO_S[[t]]$} is of weight degree $-s(f)$.
\end{prop}
\begin{proof}
This follows from the homogeneity of the choice of the holomorphic form $\Omega_{Z/S}=dz_1\wedge\cdots \wedge dz_n$ and a similar calculation on integration by part.
\end{proof}

  \begin{cor}\label{degree-of-higherresiduePairing}
    The higher residue pairing   $\mc K_\Omega^{F}\bracket{\mbox{-}, \mbox{-}}:  \Hzero^{F, \Omega}\otimes \Hzero^{F, \Omega}\to \OO_{S}[[t]]$ is of weight degree $-s(f).$
  \end{cor}

  \begin{prop}\label{euler-field-homog}
     A  section $[\eta]\in \Gamma(S, \Hzero^{F, \Omega})$ satisfies property {\upshape \mbox{(P4)}} of Definition \ref{define-primitive-form} if and only if $[\eta]$ is weighted homogeneous with respect to the $\mathbb{Q}$-grading induced by $\mathbb{E}_{F}$.
  \end{prop}

  \begin{proof}
 Write $\eta={\sum\limits_{j=0}^\infty a_jt^j}, a_j\in \Gamma(Z, \OO_Z)$. Modulo the coboundary operator $Q_F$,   we have
$$
-\bracket{\sum_iq_iz_i\pa_{z_i}F}\eta\equiv t\sum_iq_iz_i\pa_{z_i}\eta+t\bracket{\sum_iq_i}\eta \mod \Im(Q_F).
$$
   %% , all of which are of course $Q_{F}$-closed.
 Hence,
 \begin{align*}
& \big({\nabla^{\Omega}_{t\pa_t}+\nabla^{\Omega}_{E}}\big)[\eta]=
          [\big({t{\pa\over\pa t}-{F\over t}+\pa_E+{\pa_EF\over t}}\big)\eta]\\
          =&[\big({t{\pa\over\pa t}+\pa_E-{1\over t}\big({\sum_{i=1}^n q_iz_i\pa_{z_i}F}}\big)\big)\eta]
          =  [ \big({t{\pa\over\pa t}+\pa_E+\sum_i q_iz_i\pa_{z_i}+\sum_iq_i}\big)\eta]\\
                  =& [\sum_{q}  \big({t{\pa\over\pa t}+\pa_E+\sum_i q_iz_i\pa_{z_i}}\big)\eta+\big({\sum_iq_i}\big)\eta]
                  =[\mathbb{E}_{F}\eta+\big({\sum_iq_i}\big)\eta].
 \end{align*}
%(where each $\eta_q$ is a finite sum of       elements of weight-degree $q$ of the form $b_kt^k\in t^k\OO_{S}$
        Hence, property (P4) holds if and only if
 $
  \mathbb{E}_{F} [\eta]= (r-\sum_i q_i)[\eta],
$
i.e., $[\eta]$ is weighted homogeneous of degree $r-\sum_i q_i$.
      \end{proof}

\subsection{Moduli space of primitive forms}\label{moduli_space}
  As we have shown in the previous section,
a    primitive form in  $\Gamma(S, \mc H^{F,\Omega}_{(0)})$ is equivalent to a good opposite filtration $\mc L$ of $\mc H^{f,\Omega}$ together with a primitive element  with respect to $\mc L$.
\begin{prop}\label{prop-ofthm-primitiveform}
 Let     $\mc L\subset \mc H^{f, \Omega}$ be an  opposite filtration. The following are equivalent:
   $$(1) \,\, \mc L \mbox{ is good}; \quad (2)\,\,  \mathbb{E}_{f}  \mbox{ preserves } \mc L;\quad(3)\,\,  \mathbb{E}_{f}  \mbox{ preserves } \Hzero^{f, \Omega}\cap t\mc L.$$
 Furthermore, $1\in \Hzero^{f, \Omega}$ is the unique primitive element with respect to a given good opposite filtration, up to a scalar multiple by nonzero complex numbers.

\end{prop}
\begin{proof}
   Denote $r:=\sum_i q_i$.
   For any $[\eta]\in  \Hzero^{f, \Omega}$, we have
   $$
   \nabla^{\Omega}_{t\pa_t} [\eta]= [\big({t\pa_t-{f\over t}}\big)\eta]=[\big({t\pa_t-\sum_i{q_iz_i\pa_{z_i} f\over t}}\big)\eta]=[\big({{{r+t\pa_t+\sum_i q_i z_i\pa_{z_i}}}}\big)\eta]=\bracket{r+\mathbb{E}_{f}}[\eta].
   $$
   Hence, $\mc L$ is good if and only if
     $\mathbb{E}_{f}\mc L\subset \mc L$, which holds if and only if   $\mathbb{E}_{f}$ preserves $B:=\Hzero^{f, \Omega}\cap t\mc L$ \,\, (since $\mc L=t^{-1}B[t^{-1}]$ and  $\mathbb{E}_{f}\Hzero^{f, \Omega}\subset \Hzero^{f, \Omega}$).

    Since  $\zeta_0\in \Hzero^{f, \Omega}$ is primitive  only if $\zeta_0$ is a weighted homogenous element (of degree zero) such that $\Jac(f) \zeta_0=\Jac(f)$,   the latter statement follows.
\end{proof}

We now construct all good opposite filtrations for any given weighted homogeneous singularity. The method is a generalization of the classical construction of the flat coordinates for the case of finite reflection group quotient \cite{Saito-Yano-Sekiguchi}, by including the parameter $t$.
Let us fix one set $\{\phi_1, \cdots, \phi_\mu\}\subset \C[\mathbf{z}]$  of   weighted homogeneous polynomials  such that
  \begin{enumerate}
  \item $\{\phi_1, \cdots, \phi_\mu\}$    represent  a basis   of $\Jac(f)$.

  \item     $\deg \phi_1< \deg\phi_2\leq\cdots \leq \deg \phi_{\mu-1}< \deg \phi_\mu$, and  $\phi_1=1$. %, and $\deg \phi_\mu=s(f)$.
  \item   The residue pairing $(\phi_i, \phi_j)$ is zero unless $i+j=\mu+1$, which implies that each $(\phi_i, \phi_{\mu+1-i})$ is nonzero due to the non-degeneracy of the residue pairing.
\end{enumerate}
Let $$ r(i, j):=\deg \phi_i-\deg \phi_j.$$
Whenever referring to an element    $\mathbf{c}=\big(c_{ij}\big)_{1\leq i, j\leq \mu}\in \C^{\mu^2}$ below, we always require
  $c_{ij}=0$ unless $r(i, j) \in \mathbb{Z}_{>0}$. For such $\mathbf{c}$, we set
$$B(\mathbf{c}) :=\mbox{Span}_{\C}\{\Phi_1(\mathbf{c}), \cdots, \Phi_\mu(\mathbf{c})\}\subset \Hzero^{f, \Omega}  $$
 with $$ \qquad\qquad  \Phi_i=\Phi_i(\mathbf{c}):=\phi_i+\sum_{j=1}^{i-1} c_{ij}\phi_jt^{r(i, j)}. $$
Let $$Y:=\{\mathbf{c}\in \C^{\mu^2}~|~ \mc K_\Omega^{f}\big(B(\mathbf{c}), B(\mathbf{c})\big)\subset \mathbb{C},  c_{ij}=0  \mbox{ if } r(i, j)\notin \mathbb{Z}_{>0}\}\subset \mathbb{C}^{\mu^2}.$$

\begin{prop}\label{prop-Y-bij}
  There is a bijection  map   $$Y\overset{\cong}{\longrightarrow} \{\mbox{good opposite filtrations of } \mc H^{f, \Omega}\}  $$
defined by    $\mathbf{c}\in Y\mapsto \mc L(\mathbf{c}):=t^{-1}B(\mathbf{c})[t^{-1}]$.
\end{prop}

\begin{proof}
For any $\mathbf{c}\in Y$, we have $\mc K_\Omega^{f}(B(\mathbf{c}), B(\mathbf{c})) \subset \C$.
Moreover, it follows by construction  that
 $\Hzero^{f, \Omega}=B(\mathbf{c})[[t]]$, $\mc H^{f,\Omega}=B(\mathbf{c})((t))= \Hzero^{f, \Omega}\oplus\mc L(\mathbf{c})$ and $t^{-1}\mc L(\mathbf{c}) \subset \mc L(\mathbf{c})$.
 Hence, $\mc L(c)$ is an   opposite filtration (recall Definition \ref{opposite filtration}). Clearly, $B(\mathbf{c})$ is preserved by $\mathbb{E}_f$, since $\{\Phi_j\}$ are weighted homogeneous with respect to $\mathbb{E}_f$. Therefore $\mc L(\mathbf{c})$ is good by
 Proposition \ref{prop-ofthm-primitiveform}, and the given map is  defined. The injectivity  of such map follows from the linear independence of $\{\phi_1, \cdots, \phi_\mu\}$ in $\mc H^{f, \Omega}$.

 Now  we assume  $\mc L$ to be a   good opposite filtration of $\mc H^{f,\Omega}$ and let $B:=\Hzero^{f, \Omega}\cap t\mc L$.  Then we have $\mc K^f_{\Omega}(B, B)\subset \C$ by definition, and $\mathbb{E}_f$ preserves $B$     by Proposition \ref{prop-ofthm-primitiveform}.
   It follows that  $B=B(\mathbf{c})$ for some $\mathbf{c}\in Y$. That is, the given map is surjective.
 \end{proof}

Let $\mathbb{Z}_{>0}^{\scriptsize\mbox{odd}}$ be the set of positive odd integers,
  and set $$D:=\sharp\{(i, j)~|~r(i, j)\in \mathbb{Z}_{>0}, i+j<\mu+1 \}+\sharp\{(i, j)~|~r(i, j)\in \mathbb{Z}_{>0}^{\scriptsize\mbox{odd}}, i+j=\mu+1 \} .$$
\begin{prop}\label{propspaceY}
  $Y$ is an     algebraic subvariety of $\C^{\mu^2}$ defined by quadratic equations, which is biregular   to  $\C^D$.
  In particular if $D=0$, then $Y=\{\mathbf{0}\}$ consists of the origin of $\C^{\mu^2}$ only.
\end{prop}

\begin{proof}
   For $1\leq i, j\leq \mu$, we let $a_{ij}(t):=\mc K_\Omega^{f}(\phi_i, \phi_j)$. We also let  $d_{ij}(t):= c_{ij}t^{r(i, j)}$ if $r(i, j)\in \mathbb{Z}_{>0}$, or $1$ if $i=j$, or $0$ otherwise.
%%%    $$  a_{ij}(t):=(\phi_i, \phi_j)_W\qquad\mbox{and}\qquad d_{ij}(t):=\begin{cases}
%%%       c_{ij}t^{r(i, j)},&\mbox{if } r(i, j)\in \mathbb{Z}_{>0},\\
%%%       1, &\mbox{if } i=j,\\
%%%       0, &\mbox{otherwise}.
%%%    \end{cases}$$
It follows from Corollary \ref{degree-of-higherresiduePairing-fiberversion} that every $a_{ij}(t)$ is a monomial in $\C[t]$ of degree
   $\deg \phi_i+\deg \phi_j-s(f)$. Our assumption on $\{\phi_i\}$ implies that $a_{ij}(t)=0$ if $i+j<\mu+1$,   and $a_{ij}(t)$ is a nonzero constant if $i+j=\mu+1$. Since $\Phi_i=\sum\limits_{k=1}^id_{ik}(t)\phi_k$, the pairing
$$
\mc K_\Omega^{f}(\Phi_i, \Phi_j)=\begin{cases} 0, & \text{if}\ i+j<\mu+1\\=a_{ij}\in \mathbb{C}^*, & \text{if}\ i+j=\mu+1\end{cases}.
$$ By the skew-symmetric property of the higher residue pairing,
   $\mathbf{c}\in Y$ if and only if all the monomials
     $\mc K_\Omega^{f}(\Phi_i, \Phi_{j'})$ are constant functions for those entries $(i, j')$
     with $i\leq j'$ and $i+j'>\mu+1$.   Write $j'=\mu+1-j$.
     The condition is equivalent to that    $\mc K_\Omega^{f}(\Phi_i, \Phi_{\mu+1-j})$  is   a constant   for all $i, j$ with  $i-j>0$ and $i\leq \mu+1-j$. Namely the  monomial
    $K_{ij}:= \mc K_\Omega^{f}(\Phi_i, \Phi_{\mu+1-j})$  in $t$ of degree $r(i, j)$ satisfies
  $$K_{ij}=d_{i, j}(t)a_{j, \mu+1-j}+ a_{i, \mu+1-i}d_{\mu+1-j, \mu+1-i}(-t)+\!\!\!\!\!\!\!\sum\limits_{j<k\leq i\atop \mu+1-i<l\leq \mu+1-j}\!\!\!\!\!\!\!\!d_{ik}(t)a_{kl}(t)d_{\mu+1-j, l}(-t)\in \C.
     $$
  Let us refer to $k-l$ as the step of a coordinate $c_{kl}$.
   Then   $K_{ij}$ contains precisely two (possibly the same)  coordinates  $c_{ij}$, $c_{\mu+1-j, \mu+1-i}$ of step $i-j$,    and other coordinates of   step strictly less than    $i-j$. We   analyze the constraints  $K_{ij}\in \C$ according to $r(i, j)=r(\mu+1-j, \mu+1-i)$  as follows.

    Assume  $r(i, j)\notin \mathbb{Z}$ or $r(i, j)=0$. Then $K_{ij}\in \C$ is a trivial condition.

   Assume  $r(i, j)\in \mathbb{Z}_{>0}$ and $i<\mu+1-j$, then the parameter $c_{\mu+1-j, \mu+1-i}$ (which is of step $i-j$ with $(\mu+1-j)> \mu+1-(\mu+1-i)$) is uniquely determined by the free parameter $c_{ij}\in \C$ and those $c_{k',l'}$  with smaller steps.

    Assume $r(i, j)\in \mathbb{Z}_{>0}$ and $i=\mu+1-j$. Then $\mc K_\Omega^{f}(\Phi_i, \Phi_{\mu+1-j}) =\mc K_\Omega^{f}(\Phi_i, \Phi_{i})$ is an even function in $t$ by the skew-symmetric property of the higher residue pairing. Thus
    %%%Then  $a_{j, \mu+1-j}=a_{\mu+1-j, j}=a_{i, j}=a_{i, \mu+1-i}\in \C^*$.
    if $r(i, j)$ is odd, %%then $d_{i, j}(t)a_{j, \mu+1-j}(t)+ a_{i, \mu+1-i}(t)d_{\mu+1-j, \mu+1-i}(-t)$ is zero polynomial. Since
    the pairing   must vanish. In this case,   LHS $\in\C$ is a trivial condition with a free parameter $c_{ij}$ of the largest step.
      If $r(i, j)$ is even, then $c_{ij}$ is uniquely determined by the summation in LHS, which involves
   coordinates  of step strictly less than $i-j$ only.

   In summary, there are precisely $D$ free parameters. In particular if $D=0$, then  $c_{ij}=0$ unless $r(i, j)$ is a positive even integer and $i=\mu+1-j$.
   However, such $c_{ij}$ is equal to a summation of quadratic terms in coordinates of step strictly less than $i-j$, which are equal to zero. Hence, we  have $c_{ij}=0$ for all $i, j$. That is, $\mathbf{c}=\mathbf{0}$ if $D=0$.
    Hence, the statement follows.
\end{proof}

 We say two primitive forms $\zeta_1$, $\zeta_2$ of  $\mc H^{F,\Omega}_{(0)}$ are \textit{equivalent}, denoted as  $\zeta_1\sim \zeta_2$,   if  there exists $\lambda \in \C^*$ and open neighborhood $U\subset S$ of $0$ such that $\zeta_1|_U=\lambda\zeta_2|_U\in \Gamma(U, \mc H^{F,\Omega})$.
 Let
    $$\mc M:=\{\mbox{primtive forms of }  \mc H^{F,\Omega}_{(0)} \}/\sim\,\,.$$

\begin{thm}\label{thm-space-of-primitiveforms}
  There is a bijection map $Y\overset{\cong}{\to} \mc M$, inducing a  bijection map
 $  \mc M\overset{\cong}{\rightarrow}\mathbb{C}^D.$

\end{thm}

 \begin{proof}
  By using Theorem  \ref{thm-primitive-local}, Theorem \ref{thm-analytic} and Proposition \ref{prop-ofthm-primitiveform}, we obtain a bijection
  between $\mc M$ and the set of   good opposite filtrations of $\mc H^{f,\Omega}$. The first statement follows from Proposition \ref{prop-Y-bij}, and the second statement follows from Proposition \ref{propspaceY}.
  \end{proof}

\begin{cor}\label{thm-primitive-exceptionalsing}When $D=0$, the representatives of any two weighted homogeneous bases of $\Jac(f)$ %, with highest degree $s(f)$,
  span the same vector subspace $B$ in $\mc H^{f, \Omega}$. This defines the  unique  good opposite filtration  of $\mc H^{f, \Omega}$ by
 {\upshape $$\mc L_{\scriptsize\mbox{can}}=t^{-1}B[t^{-1}].$$}

   If $f$ is given by one of the exceptional unimodular singularities,
  then there exists a unique primitive form $\zeta_+$, up to a scalar of nonzero complex numbers.
   Furthermore,   the Taylor expansion of $\zeta_+$ up to order $N-1$ is given by
   the projection   of  the element $[\varphi_{R_N}(e^{F_0-F\over t})]$ in  {\upshape $ {\mc H^{F,\Omega}}\otimes_{\OO_S}R_N=({\mc H}_{(0)}^{F,\Omega}\otimes_{\OO_S}R_N)\oplus \big(e^{f-F\over t}\mc L_{\scriptsize\mbox{can}}\otimes_{\C}R_N\big)$} to ${\mc H}_{(0)}^{F,\Omega}\otimes_{\OO_S}R_N$,
   where   $R_N=\OO_{S, 0}/\mathfrak{m}^N$,  and $\varphi_{R_N}: \OO_{S, 0}\to R_N$ is the natural projection.
\end{cor}

 \begin{proof}
     By Proposition \ref{prop-Y-bij},    there is a unique  good opposite filtration    $\mc L(\mathbf{0}) =\mathbf{0} =t^{-1}B(\mathbf{0})[t^{-1}]$ with respect to the choice $B=\mbox{Span}_\C\{\phi_1, \cdots, \phi_\mu\}$.
           Let $B'$ be spanned by $\{\phi_1', \cdots, \phi_\mu'\}\subset \Gamma(X, \OO_X)$, which represent another basis of $\Jac(f)$  with $\deg \phi_1'\leq \cdots \deg \phi_\mu'=s(f)$. It follows from Proposition \ref{prop-Y-bij} again that $\mc L'(\mathbf{0})=t^{-1}B'(\mathbf{0})[t^{-1}]$ is the unique good opposite filtration.
     Hence, $\mc L'(\mathbf{0})=\mc L(\mathbf{0})$ by the uniqueness, implying $B'=B$.

 Now we consider the case when   $f$ is given by one of the exceptional unimodular singularities.   All the degrees $\deg\phi_i$ can be read off directly from     Table 3 of \cite{Saito-exceptional} by using   equality \eqref{eqn-degPhiExpo}.
 It follows that  $\deg \phi_i+\deg \phi_j-s(f)$
  is not an integer unless it vanishes, by   direct calculations.
   Hence,     we have $D=0$, and therefore   $Y=\{\mathbf{0}\}$.
Then  the statement follows immediately from
  Theorem \ref{thm-space-of-primitiveforms}, Theorem \ref{thm-primitive-local} and Theorem \ref{thm-analytic}.
 \end{proof}

\subsection{The perturbative formula of a primitive form}\label{algorithm}
Let $\{\phi_1, \cdots, \phi_\mu\}\subset \C[\mathbf{z}]$ be a set of weighted homogeneous polynomials as in the previous subsection, which represent   a basis of $\Jac(f)$.
As we have seen,  a point  $\mathbf{c}\in Y$ corresponds to a good opposite filtration $\mc L(\mathbf{c})$ of $\mc H^{f, \Omega}$, and hence a unique primitive form $\zeta_+$ of $\mc H^{F,\Omega}_{(0)}$ with respect to $\mc L(\mathbf{c})$ together with the primitive element $1$.
 In this subsection, we provide an explicit algorithm to compute $\zeta_+$, %%associated to $\mc L(\mathbf{c})$ together with the primitive element $1$,
 up to an arbitrary finite order $N$.

%% \subsubsection{The truncated case}
 Let   $R=\mathbb{C}[\mathbf{u}]/\mathfrak{m}^{N+1}$.  Consider the natural projection $\pi_X: Z\subset X\times S \to X$.
 The map
    $$e^{F-f\over t}: \Hzero^{F, \Omega}\otimes_{\OO_S}R\longrightarrow \mc H^{f,\Omega}\otimes_{\C}R=B((t))\otimes_{\C} R$$
(which is the inverse of \eqref{expHWR} in Lemma \ref{lem-flatextension}) defines  a $\mu\times \mu$ matrix  $A(\mathbf{u}, t, t^{-1})$ valued in $R$, by
         $$e^{F-F_0\over t} \left(\begin{array}{c}\pi_X^*\Phi_1\\ \vdots \\ \pi_X^*\Phi_\mu\end{array}\right)= A(\mathbf{u}, t, t^{-1}) \left(\begin{array}{c}\Phi_1\\ \vdots \\ \Phi_\mu\end{array}\right).$$
Here we treat   $\Phi_i=\Phi_i(\mathbf{c})$ and $\pi^*_X\Phi_i$ as elements in $\mc H^{f,\Omega}\otimes_\C R$ and
           $ \mc H^{F,\Omega}\otimes_{\OO_S} R$, respectively.
   Let
   $$
   a:=\max \fbracket{[N(s(f)-1)+s(f)], [s(f)]}
   $$
where by $[r]$ we mean  the largest integer that is less than or equal to $r$.
 It follows from degree reason that
 $$A(\mathbf{u}, t, t^{-1})= \sum_{k=-\infty}^\infty t^k A^{(k)}(\mathbf{u})= \sum_{k=-N}^a t^k A^{(k)}(\mathbf{u})
 $$
where every $A^{(k)}$ ($k\neq 0$)  is a $\mu\times \mu $ matrix with entries in $\mathfrak{m}/\mathfrak{m}^{N+1}$; so is $A_{(0)}-\mbox{Id}$. As a matrix valued in $R$, $A^{(k)}$ is   zero  whenever $k>a$ or $k<-N$.

We consider the following $(a+1)\mu\times (a+1)\mu$ matrix
 {\upshape $$\Psi:=\left(\begin{array}{cccc}A_{(0)}-\mbox{Id}&A^{(1)}&\cdots &A^{(a)}\\
               A^{(-1)}&A_{(0)}-\mbox{Id}&\cdots &A^{(a-1)}\\
                \cdots    &\cdots &\cdots &\cdots \\
               A^{(-a)}&A^{(-a+1)}&\cdots &A_{(0)}-\mbox{Id}
         \end{array}\right)$$
         }
and the $1\times (a+1)\mu$ matrix
 {\upshape $$\mathbf{e}:=\left(1,0,\cdots,0\right).$$
         }

\begin{thm}\label{thm-algorithm}
The first $N+1$ terms of the Taylor series expansion of the primitive form $\zeta_+$  around the reference point $0$ is  given by
$$
\zeta_+(R)=\sum_{i=0}^{a} t^i\vec g^{(i)}\left(\begin{array}{c}\Phi_1\\ \vdots \\ \Phi_\mu\end{array}\right)
$$
where
        $$(\vec g_{(0)}, \vec g^{(1)}, \cdots, \vec g^{(a)})=\mathbf{e} \sum_{k=0}^N (-1)^k\Psi^k. $$

          \end{thm}

\begin{proof}
  Let $\mc L =t^{-1}B(\mathbf{c})[t^{-1}]$. By  Corollary \ref{rmk-primitive},
the hypothesis
 $$\varphi_R(e^{F-f\over t}  \eta)\in 1+\mc L\otimes_\C R$$
admits a unique solution given by the
  primitive form  $\eta=\zeta_+(R)$ constructed with respect to  $\mc L$. Since $\{\pi^*_X\Phi_i\}_i$ are a basis  of $\mc H^{F,\Omega}\otimes_{\OO_S}R$, we can write
   $$\zeta_+(R)=\sum_{i=1}^\mu g_i\Phi_i, \mbox{ where } g_i=g_i(\mathbf{u}, t)\in \C[[t]]\otimes_\C R.$$
Since $\zeta_+(R)$ is of weight degree zero,
$$
(g_1, \cdots, g_\mu)=:\vec g=\sum_{k=0}^{a} \vec g^{(k)}t^k,
$$
where  each $\vec g^{(i)}$ is a $1\times \mu$ matrix with entries in $R$.
 Therefore,  we have
$$(g_1, g_2, \cdots, g_\mu)\sum_{k=-N}^at^k A^{(k)}=(1+h_1, h_2, \cdots, h_\mu)$$
for some %%Write $e^{F-F_0\over t} \eta_+(R)=(1+g_1)\Phi_1+\sum_{j=2}^\mu g_j \Phi_j$, where
 $h_j=h_j(\mathbf{u}, t)\in t^{-1}\C[t^{-1}]\otimes_\C R, j=1, 2, \cdots, \mu$. Equivalently,
we have the system of equations:

 {\upshape
 $$(\vec g_{(0)}, \vec g^{(1)}, \cdots, \vec g^{(a)})
           \left(\begin{array}{cccc}A_{(0)}&A^{(1)}&\cdots &A^{(a)}\\
               A^{(-1)}&A_{(0)}&\cdots &A^{(a-1)}\\
                \cdots    &\cdots &\cdots &\cdots \\
               A^{(-a)}&A^{(-a+1)}&\cdots &A_{(0)}
         \end{array}\right)= \mathbf{e}
   $$
         }
Since
 {\upshape
$$
        1+\Psi= \left(\begin{array}{cccc}A_{(0)}&A^{(1)}&\cdots &A^{(a)}\\
               A^{(-1)}&A_{(0)}&\cdots &A^{(a-1)}\\
                \cdots    &\cdots &\cdots &\cdots \\
               A^{(-a)}&A^{(-a+1)}&\cdots &A_{(0)}
         \end{array}\right)
   $$
}
and $\Psi$ is a matrix with entries in $\mathfrak{m}/\mathfrak{m}^{N+1}$, we have
$$
   (1+\Psi)^{-1}=\sum_{k=0}^N (-\Psi)^k.
$$
Thus the statement follows.
\end{proof}

\section{Examples}\label{examples} In this section, we provide some concrete examples of primitive forms. Whenever a good opposite filtration $\mc L$ is specified below, we always consider the unique primitive form associated to $\mc L$ together with the primitive element $1$.
\subsection{$ADE$-singularities}
  Let $f$ be a weighted homogeneous polynomial of type  $ADE$-singularity. That is,    $s(f)<1$. It is easy to see  that $D=0$, $Y=\{0\}$.
  As a consequence,  $\mc L_{\scriptsize\mbox{can}}$ is the only good opposite filtration that is given by $Y$, which consists elements of negative weight degrees.
   Here we only refer to the degree of the variables $z_i, t$. Therefore
$\deg \bracket{(F-F_0)/t}<0$.  It follows that
$$
  e^{(F-F_0)/t}\in 1+\mc L_{\scriptsize\mbox{can}}\otimes_{\C} R
$$
The next proposition \cite{Saito-primitive} follows from Corollary \ref{rmk-primitive}.
\begin{prop}
 The element $1\in \Gamma(S, \mc H^{F,\Omega}_{(0)})$ is the unique  primitive form up to a nonzero scalar.
\end{prop}

\subsection{Simple elliptic singularity} Let $f={1\over 3}\bracket{z_1^3+z_2^3+z_3^3}$.
Then $f$ is a weighted homogeneous polynomial with  $s(f)=1$. It is called a simple elliptic singularity of type $E_6^{(1,1)}$.
  Take a (ordered) monomial basis by
 $$\{\phi_1, \cdots, \phi_8\}=\{1, z_1, z_2, z_3, z_1z_2, z_2z_3, z_3z_1, z_1z_2z_3\},$$
 and consider the universal unfolding
   $$F={1\over 3}z_1^3+{1\over 3}z_2^3+{1\over 3}z_3^3+\sigma \phi_8+\sum_{i=1}^7u_i\phi_i,$$
where $(u_1, \cdots, u_7, \sigma)$ are the local coordinates of $S$. We take $\Omega_{Z/S}=dz_1\wedge dz_2\wedge dz_3$.

  By definition,  $D=1$,  and hence $Y\cong \C$. There is only one free parameter in $\mathbf{c}=(c_{ij})$, say $c$. We have
$$\mc L({c}) = t^{-1}B(c)[t^{-1}]\quad\mbox{where}\quad B(c)=\mbox{Span}_{\C}\fbracket{\phi_1, \cdots, \phi_7, \phi_8+ c t}\subset\Hzero^{f, \Omega}.$$

Every  $u_j$ in $F$ belongs to the maximal ideal of $R$. By Corollary \ref{rmk-primitive},  the associated primitive form $\zeta_+(R)\in \mc H^{F,\Omega}(R)$ is determined by the relation
$$
  e^{(F-f)/t}\zeta_+(R)\in 1+\mc L(c)\otimes_{\C}R.
$$

  \begin{lem}
As elements in $\mc H^{f,\Omega}$, the class $e^{{\sigma \phi_8}/t}$ is well-defined ($\phi_8=z_1z_2z_3$), and it is  given by
  $$
    e^{{\sigma \phi_8}/t}=g(\sigma)+ t^{-1}h(\sigma) \phi_8,
$$
 where $g(\sigma)$ and $h(\sigma)$ are respectively given by $$g(\sigma)=1+\sum_{r=1}^\infty{(-1)^r\sigma^{3r} \prod_{j=1}^r(3j-2)^3\over (3r)!},
     \quad h(\sigma)= \sigma+\sum_{r=1}^\infty{(-1)^r\sigma^{3r+1} \prod_{j=1}^r(3j-1)^3\over (3r+1)!}.$$
Furthermore,    $g(\sigma), h(\sigma)$ are  the fundamental solutions to  the following Picard-Fuchs equation:
 \begin{equation}\label{eqn-PFeqn-Ellip}
   \big({(1+\sigma^3)\pa_\sigma^2 + 3\sigma^2 \pa_\sigma +\sigma}\big) \nu(\sigma)=0.
  \end{equation}
\end{lem}
\begin{proof}
  Note $Q_{f}(z_2^2z_3^2\pa_1)=(\pa_{z_1}f)z_2^2z_3^2=z_1^2z_2^2z_3^2=\phi_8^2$. That is, $[\phi_8^2]=[0]$ in $\mc H^{f,\Omega}$. Similarly, for $k\geq 3$, it follows from direct calculations that
    $$
    [\phi_8^k]=[z_1^{k-2}z_2^kz_3^k\pa_{z_1}f]=[-t(k-2)z_1^{k-3}z_2^kz_3^k]=\cdots=[-t^3(k-2)^3\phi_8^{k-3}]
    $$
  Hence, in    $\mc H^{f,\Omega}$ we have
    $$
     e^{{\sigma \phi_8}/t} =\sum_{k=0}^\infty {\sigma^k\phi_8^k\over k!t^k}=\sum_{r=0}^\infty {\sigma^{3r}\phi_8^{3r}\over (3r)!t^{3r}}+\sum_{r=0}^\infty {\sigma^{3r+1}\phi_8^{3r+1}\over (3r+1)!t^{3r+1}} =
     g(\sigma)+t^{-1}h(\sigma)\phi_8 .
    $$

It is easy to check that $g, h$ satisfy the  differential equation \eqref{eqn-PFeqn-Ellip}. Since $g, h$ are linearly independent, they are the fundamental solutions. In particular, they are convergent.
  %%Since $A$ is the generator in $\Jac(W_0+\sigma A)$ of (top) degree 1,   for every $k$ we have an equality of the following form
%%$$   A^k = a_k t^k + b_k t^{k-1}A$$ as elements in $\mc H^{W_0+ \sigma A}$, where $a_k, b_k$ depend  only on $\sigma$.
\end{proof}
 As a consequence, in $\mc H^{f,\Omega}$, the class
    $$
    e^{(F-F_0)/t}-(g(\sigma)+ t^{-1}h(\sigma) \phi_8)=e^{\sigma \phi_8\over t}(e^{(F-F_0-\sigma \phi_8)/t}-1)
    $$
 consists of elements of negative degree (viewed as elements in $\mc H^{f,\Omega}$), and hence lies in $\mc L(c)\otimes_{\C} R$.
 Therefore
  $
    e^{(F-F_0)/t}-(g(\sigma)-c\cdot h(\sigma)) \in \mc L(c)\otimes_{\C} R.
  $
 Then for
 $
\zeta_{+}(R):={1\over g(\sigma)-c\cdot h(\sigma)}\in \Hzero^{F, \Omega}\otimes_{\OO_S} R,
 $
we have
$$
   e^{(F-f)/t} \zeta_+(R)\in 1+\mc L(c)\otimes_{\C} R.
$$
It follows that
$$
  \zeta_+(R)={1\over g(\sigma)-c\cdot  h(\sigma)}
$$ is a primitive form by Corollary  \ref{rmk-primitive}.

\bigskip

The above hypergeometric equation  \eqref{eqn-PFeqn-Ellip} is precisely the Picard-Fuchs equation for period integrals on elliptic curves. Consider the elliptic curve $E_\sigma\subset \mathbb P^2$
defined by the equation $f+\sigma z_1z_2z_3=0$, with holomorphic top form
$$
   \Omega_{E_\sigma}=\Res {z_1 dz_2dz_3-z_2 dz_1dz_3+z_3 dz_1dz_2\over f+\sigma z_1z_2z_3}
$$
where
$\Res: \Omega^2_{\mathbb P^2}(E_\sigma)\to \Omega^1_{E_\sigma}$  denotes the residue map. Let $\Gamma$ be a topological 1-cycle in $E_\sigma$, then the period integral $\int_\Gamma \Omega_{E_\sigma}$ satisfies the Picard-Fuchs equation \eqref{eqn-PFeqn-Ellip}
 In particular, an arbitrarily fixed basis $\{\Gamma_1, \Gamma_2\}$ of $H_1(E_\sigma, \mathbb Z)$ gives
  gives the fundamental solutions $\int_{\Gamma_1} \Omega_{E_\sigma}, \int_{\Gamma_2} \Omega_{E_\sigma}$ to the Picard-Fuchs equation. Hence, we reprove  the next proposition   \cite{Saito-primitive}.
\begin{prop}      $\zeta\in \Gamma(S, \Hzero^{F, \Omega})$   is a primitive form if and only if
$$
  \zeta= {1 \over a\int_{\Gamma_1} \Omega_{E_\sigma}+ b\int_{\Gamma_2} \Omega_{E_\sigma}}
$$
for some     $(a, b)\in\C^*\times \C$.
\end{prop}

\subsection{$E_{12}$-singularity}
The exceptional singularity of type $E_{12}$ is given by the weighted homogeneous polynomial:
   $$f(x, y)=x^3+y^7.$$
with  $s(f)={22\over 21}$. Take the (ordered) set $$\{\phi_1, \cdots, \phi_{12}\}:=\{1, y, y^2, x, y^3, xy, y^4, xy^2, y^5, xy^3,  x y^4, xy^5\}\subset \C[x, y],$$
which represent  a basis of $\Jac(f)$. Consider the universal unfolding
   $$F(x, y, \mathbf{u})=x^3+y^7 +\sum_{i=1}^{12}u_i\phi_i $$
 and take $\Omega_{Z/S}=dx\wedge dy$.
In this case, $Y=\{\mathbf{0}\}$ by  Theorem \ref{thm-space-of-primitiveforms}. That is, there is a unique primitive form up to a nonzero scalar.  Hence, $\Phi_j=\phi_j$ for all $1\leq j\leq 12$.
As an example of Theorem  \ref{thm-algorithm}, we  compute $\zeta_+(R)$ with $R=\C[x, y]/\mathfrak{m}^{11}$ where $\mathfrak{m}=(x,y)$ is the maximal ideal at the origin.

   %Using Theorem/Algorithm \ref{thm-algorithm} and

   By direct calculations (with a computer), we have
\begin{align*}
   \zeta_+(R)&=1+{4\over 3\cdot 7^2}u_{11}u_{12}^2-{64\over 3\cdot 7^4}u_{11}^2u_{12}^4-{76\over 3^2\cdot 7^4}u_{10}u_{12}^5+{937\over  3^3\cdot 7^5}u_9u_{12}^6+{218072\over 3^4\cdot 5\cdot 7^6}u_{11}^3u_{12}^6\\
         &\qquad+{1272169\over 3^4\cdot 5\cdot 7^7}u_{10}u_{11}u_{12}^7+{28751\over 3^4\cdot  7^7}u_{8}u_{12}^8-{1212158\over 3^4 \cdot 7^8}u_9u_{11}u_{12}^8-{38380\over 3^3\cdot 7^8}u_{7}u_{12}^9\\
         &\quad+\big({1\over 7^2} u_{12}^3-{ 101\over 5\cdot 7^4} u_{11}u_{12}^5+ {1588303\over 3^4\cdot 5\cdot 7^7}u_{11}^2u_{12}^7+
                {378083\over 3^4\cdot 5\cdot 7^7}u_{10}u_{12}^8-{ 108144\over 3\cdot 7^8}u_{9}u_{12}^9\big)x\\
         &\quad+\big({1447\over 3^3\cdot 7^6} u_{12}^7-{ 71290\over 3^3\cdot 7^8}u_{11}u_{12}^9\big)y   -{45434\over 3^4\cdot 7^8} u_{12}^{10}xy\\
         &\quad -\big({53\over 3^2\cdot 7^4}u_{12}^6-{ 46244\over 3^3\cdot 7^7}u_{11}u_{12}^8\big) x^2 +{22054\over 3^4\cdot 7^7}u_{12}^9 x^3
\end{align*}
Using our perturbative formula, we can     compute the leading term of the potential $\mc F_0$ of the Frobenius manifold structure on $S$ associated to $\zeta_+$.
For instance, the homogeneous degee 4 component $\mc F_{0}^{(4)}$ of $\mc F_0$ in flat coordinates $(t_1, \cdots, t_{12})$ is given by \cite{LLS}
\begin{align*}
   -\mc F_{0}^{(4)}&={1\over 14} t_5 t_6 t_7^2+{1\over 18} t_6^3t_8+{1\over 7}t_5^2t_7t_8+{1\over 7} t_3t_7^2t_8+{1\over 6} t_4t_6t_8^2+{1\over 14} t_5^2t_6t_9+{1\over 7} t_3t_6t_7t_9\\
                 &\quad +{1\over7} t_3t_5t_8t_9+{1\over 7} t_2t_7t_8t_9+{1\over 14}t_2t_6t_9^2+{1\over 14} t_5^3t_{10}+{1\over 6} t_4t_6^2t_{10}+{2\over 7}t_3t_5t_7t_{10}+{1\over 14}t_2t_7^2 t_{10}\\
                 &\quad  +{1\over 6}t_4^2t_8t_{10}+{1\over 14}t_3^2 t_9t_{10}+{1\over 7}t_2t_5t_9t_{10}+{1\over 7}t_3t_5^2t_{11}+{1\over 6}t_4^2t_6t_{11}+{1\over 7}t_3^2t_7 t_{11}+{1\over 7}t_2t_5t_7t_{11} \\
               &\quad +{1\over 7}t_2t_3t_9t_{11}+{1\over 18}t_4^3t_{12} +{1\over 14} t_3^2t_5t_{12}+{1\over 14}t_2t_5^2t_{12}+{1\over 7}t_2t_3t_7t_{12}+{1\over 14}t_2^2t_9t_{12}.
\end{align*}
Furthermore, this   matches the FJRW 4-point corellators of $f$ with respect to the maximal group symmetry $G_{\rm max}$. By a reconstruction technique using the WDVV equations, we obtain the following in a forthcoming joint work with Shen \cite{LLS}
\begin{thm}  Let $f(x, y)=x^3+y^7$.
   There is an isomorphism of Frobenius manifolds between $S$ and the FJRW space $H_{f, G_{\rm max}}$. Here the Frobenius manifold structure on $S$ is induced from the primitive form associated to the good pair $(\mc L_{\rm can}, (-1)^{22\over 21})$. The  FJRW space $H_{f, G_{\rm max}}$ is equipped with a formal Frobenius manifold structure by genus zero FJRW corellators.
\end{thm}

We remark that a  weighted homogeneous polynomial $f$ is dual to a distinct $f^T$ in general. The  above theorem and its generalization to all genera    hold whenever either A-side $f$ or B-side $f^T$ is given by Arnold's  exceptional unimodular singularities. That is, under the hypothesis on $f$ or $f^T$,
the LG-LG mirror symmetry between   Saito-Givental theory (LG B-model) of $f^T$ and its mirror FJRW theory (LG A-model) for $(f, G_{\rm max})$ holds. All relevant details are given in  \cite{LLS}.

\subsection{Mirror of $\mathbb{P}^1$}\label{subsec-mirror-P1}
Consider the mirror Landau-Ginzburg model $(X, f)$ of the complex projective line $\mathbb{P}^1$. Here
  $X=\mathbb{C}^*$ and the superpotential $f:X\to \C$  is given by
    $\displaystyle f(z)=z+{q\over z},$
where   $q$ is  a nonzero constant in $\C^*$. % with $0<|q|<1$.
The quantum cohomology ring $QH^*(\mathbb{P}^1, \mathbb{C})$ of $\mathbb{P}^1$ is isomorphic to  the Jacobian ring  $\Jac(f)=\mathbb{C}[z, z^{-1}]/\langle z{\pa f\over \pa z}\rangle$ of $f$.
  We will consider the Milnor fibration  $(Z, S, p, F)$, where
     $Z=\mathbb{C}^*\times \C^2$, $S=\C^2$ %(where, strictly speaking, $S$ is a small Stein open neighborhood of the origin),
      and the universal unfolding  is given by
       $$F(z, u_0, u_1)=u_0+z+{q\exp(u_1)\over z}.$$
 Consider primitive forms with respect to the choice     $\displaystyle \Omega_{Z/S}:={dz\over z}$.
 Then Theorem 3.7 of \cite{Takahashi} (where $Z$ and $S$ are chosen smaller) can be restated  as follows.  Here we  provide an alternative proof, by using Theorem \ref{thm-primitive-local}.
\begin{rmk}
  The pair $((\C^*)^n, f)$ with $f(\mathbf{z}):=z_1+\cdots+z_n+{q\over z_1\cdots z_n}$ is a mirror Landau-Ginzburg model of $\mathbb{P}^n$.
   In this case, it has been shown in \cite{Barannikov-projectivespace} that the identity section $1$ is a primitive form with respect to
   the relative holomorphic volume form ${dz_1\over z_1}\wedge\cdots\wedge {dz_n\over z_n}$.
\end{rmk}

\begin{prop}\label{primitive-mirrorP1}
  The identity section $1\in\Gamma(S, \Hzero^{F, \Omega})$ is a primitive form.
\end{prop}

 \begin{proof} To prove the statement, we define a subspace $\mc L$ of $\mc H^{f,\Omega}$, given by
 $$
\mc L:= t^{-1}B[t^{-1}]\quad \mbox{where}\quad B:=\mbox{Span}_{\C}\fbracket{1, {q\over z}}\subset\mc H^{f,\Omega}.
$$
We first show that $\mc L$ is an  opposite filtration. Note that $\{1, {q \over z}\}$ represent  a basis of $\Jac(f)$.   Clearly, we have
$t^{-1}\mc L\subset \mc L\,\mbox{ and }\,\mc H^{f,\Omega}=\Hzero^{f, \Omega}\oplus \mc L.$
Assume the next claim on the higher residue pairings on  $B\subset\mc H^{f,\Omega}$ first.

\noindent\textbf{Claim: }%%\label{lem-forprojP1}
$
  {}\quad  \mc K^f_\Omega\bracket{1,1} =0,\quad \mc K^f_\Omega\bracket{{q\over z}, {q\over z}} =0, \quad \mc K^f_\Omega\bracket{1, {q\over z}}=-1.
$

\noindent Then It follows  that $\mc K^{f}_\Omega(B, B)=\C$, %and consequently   $\Res_{t=0}\big(\mc K^f^\Omega(\mc L, \mc L)dt\big)=0$.
and hence  $\mc L$ is an opposite filtration by definition.

Since  $Q_{  f}(z \pa_z)=z{\pa_z  f}=z-{q\over z}$,   we have $[z]=[{q\over z}]$ in $\Hzero^{f,\Omega}$.  Hence,
 $$\nabla^{\Omega}_{t\pa_t} 1-0\cdot 1=\nabla^{\Omega}_{t\pa_t} 1= [-{f\over t}]=-[{z+{q\over z}\over t}]=-2[{q\over z}]t^{-1}\in \mc L.$$
Since
$Q_{ f}(\pa_z)=\pa_zf-{t\over z}=1-{q\over z^2}-{t\over z}$,
we have    $[-{q\over z^2}]=[{t\over z}-1]$ in $\Hzero^{f,\Omega}$.
Hence,  $$\nabla^{\Omega}_{t\pa_t} [{q\over z}]= [-{z+{q\over z}\over t}\cdot {q\over z}]=[-{q\over t}+{1\over t}({qt\over z}-q)]=[-{2q\over t}+{q\over z}]\in t^{-1}B+B.$$
It follows that  $\nabla^{\Omega}_{t\pa_t} \mc L\subset \mc L$.
Hence, the opposite filtration  $\mc L$ is good and $1$ is a primitive  element with respect to $\mc L$.

 Note ${F-F_0\over t}=t^{-1}\big(u_0+{q\over z}(e^{u_1}-1)\big)\in \mc L\otimes_{\C} R$.
For $k\geq 2$, in  $\Hzero^{f, \Omega}$, we have
 $${q\over z^k}={1\over z^{k-2}}-{1\over z^{k-2}}\pa_z f={1\over z^{k-2}}+t\cdot z \pa_z {1\over z^{k-2} \cdot z} ={1\over z^{k-2}}-(k-1)t   {1\over z^{k-1}}.$$
It follows that  $1-e^{F-F_0\over t} \in \mc L\otimes_{\C} R$.
 Hence,
  $$e^{F_0-F\over t}=1 +(e^{F_0-F\over t}-1)\in (\Hzero^{F, \Omega}\otimes_{\OO_S}R)\oplus \big(e^{f-F\over t}\mc L\otimes_{\C} R\big).$$
 Therefore, the identity section  $1$ is a primitive form by Theorem \ref{thm-primitive-local}.

 It remains to prove the above claim as follows.

 We simply denote $f':={\pa_z  f}$.  It is easy to compute that
$$
V_{f}={1\over  f^\prime}\pa_z\wedge,\quad\mbox{and}\quad
 \bbracket{Q, V_{f}}g=t{1\over   f^\prime}z{\pa\over \pa z}\big({g\over z}\big)  \quad \mbox{  for any }  g\in \Gamma(X, \OO_X).
$$
\noindent Denote    $B_\varepsilon^{\pm}:=\{z\in X~|~ |z-(\pm \sqrt{q})|=\varepsilon\}$ and $C_\varepsilon^{\pm}:=\pa B_\varepsilon^{\pm}$,  where $0<\varepsilon <<1$.  Take a smooth cut-off function $\rho$ on $X $ with   $\rho|_{B_\varepsilon^\pm}\equiv 1$ and $\rho|_{X\setminus(B_{2\varepsilon}^+\cup B_{2\varepsilon}^-)}\equiv 0$.
Similar to   Example \ref{example-An}, we have
\begin{align*}
    \mc K_\Omega^{f}\big({{1\over z^i}, {1\over z^j}}\big)    &=-\sum_{r\geq 0}t^r\oint_{C_\varepsilon^+\cup C_\varepsilon^-}{1\over z^i}{dz\over z^2\cdot f^\prime} \bracket{-z{\pa\over \pa z}{1\over z   f'}}^r \bracket{{1\over z^j}} \\
    &=\sum_{r\geq 0}t^r\Res_{z\in \{0, \infty\}}{1\over z^i}{dz \over z^2\cdot f^\prime} \bracket{-z{\pa\over \pa z}{1\over z   f'}}^r \bracket{{1\over z^j}} \\
    &=\sum_{r\geq 0}(-t)^r\Res_{z\in \{0, \infty\}} {dz\over z}{1\over z^i}{1\over z-q/z}\bracket{z{\pa\over \pa z}{1\over z-q/z}}^r {1\over z^j}
\end{align*}
where we use the observation that poles of the integrant  can only occur at $\pm\sqrt{q}$, $0$ or $\infty$. By direct calculations, the leading term (up to a constant) of the Laurant series expansion of
$$
A_r(z):={1\over z^i}{1\over z-q/z}\bracket{z{\pa\over \pa z}{1\over z-q/z}}^r {1\over z^j}.
$$
 around the origin $z=0$ is given by  $z^{r+1-i-j}$.
The leading term (up to a constant) of the Laurant series expansion of
$A_r(z)$
 around the origin $z=\infty$ is given by  $({1\over z})^{r+1+i+j}$.
It follows easily that
$$
  \mc K_\Omega^{f}\bracket{1,1}=0, \quad \mc K^f_\Omega\bracket{1, {q\over z}} =q\Res_{z=0}{dz\over z(z^2-q)}=-1,
$$
$$
 \mbox{and}\quad\mc K_\Omega^{f}\big({{1\over z}, {1\over z}}\big)=(-t)\cdot \Res_{z=0}{dz\over z(z^2-q)}  {z{\pa \over \pa z}}({1\over z^2-q})=0.
$$
 \end{proof}

\begin{rmk}
  We thank Atsushi Takahashi for pointing out that $\{z^m\}_{m\in \mathbb{Z}}$ are all primitive forms with respect to the above choice $\Omega_{Z/S}={dz\over z}$.
  Indeed, we can set $B_m:=\mbox{Span}_{\C}\{z^m, q z^{m-1}\}$, $\mc L_m:=t^{-1}B_m[t^{-1}]\subset \mc H^{f, \Omega}$, and $\zeta_0^{(m)}:=z^m$. Then $\mc L_m$ is a good opposite filtration, $\zeta_0^{(m)}$ is a primitive element with respect to $\mc L_m$, together with $m$ being  the constant for the homogeneity property. Furthermore, $z^m$ is the primitive form associated to the pair $(\mc L_m, \zeta_0^{(m)})$.
\end{rmk}

%\bibliography{biblio}
%%%%%%%%%%%%%%%%%%new citations%%%%%%%%%%%

\end{document}